\documentclass[12pt]{amsart}

\usepackage{graphicx}

\usepackage{amssymb}

\usepackage{graphics,color}

\usepackage{dsfont}

\usepackage{psfrag}

\usepackage{amsthm}

\def\twomatrix(#1 #2\\#3 #4){\begin{pmatrix}#1&#2\\#3&#4\end{pmatrix}}   

\def\<#1,#2>{\left< #1, #2\right>}

\newcommand{\PSLR}{{\bf {PSL}}(2,\R)}

\newcommand{\M}{{\bf {M}}^{3}}

\newcommand{\Ha}{{\mathbb {H}}^{2}}

\newcommand{\Hao}{ {\mathbb {H}}^{3}}

\newcommand{\R} {\mathbb {R} }

\newcommand{\J} {\mathds {1}}

\newcommand{\Q} {\mathbb {Q} }

\newcommand{\Z} {\mathbb {Z}}

\newcommand{\N}{\mathbb {N}}

\newcommand{\TT} {\mathbb {T}}

\newcommand{\ph}{\widehat{\partial}}

\newcommand{\pt}{\widetilde{\partial}}

\newcommand{\wt}{\widetilde}

\newcommand{\wh}{\widehat}

\newcommand{\ul}{\underline}

\newcommand{\ov}{\bar}

\newcommand{\cd}{\cdot}

\newcommand{\Su}{{\bf S}}

\newcommand{\Pant}{{\bf {\Pi}}}

\newcommand{\hl}{{\bf {hl}}}

\newcommand{\len}{{\bf l}}

\newcommand{\Ang}{{\bf {\Theta}}}

\newcommand{\Col}{{\mathcal {C}}}

\newcommand{\TB}{T^{1}}

\newcommand{\NB}{N^{1}}

\newcommand{\Mes}{{\mathcal {M}}}

\newcommand{\Ne}{{\mathcal {N}}}

\newcommand{\Rel}{{\mathcal {R}}}

\newcommand{\Me}{{\mathcal {M}}}

\newcommand{\Ze}{{\mathcal {Z}}}

\newcommand{\Le}{{\mathcal {L}}}

\newcommand{\KG}{{\mathcal {G}}}

\newcommand{\Xe}{{\mathcal {X}}}

\newcommand{\Area}{\operatorname{Area}}

\newcommand{\foot}{\operatorname{foot}}

\newcommand{\diam}{{\operatorname {diam}}}

\newcommand{\dist}{{\operatorname {dis}}}

\newcommand{\id}{{\operatorname {id}}}

\newcommand{\Orb}{{\operatorname {Orb}}}

\newcommand{\proj}{{\operatorname {proj}}}

\newcommand{\Conn}{{\operatorname {Conn}}}

\newcommand{\FConn}{{\operatorname {{\mathcal F}Conn}}}

\newcommand{\Vol}{{\operatorname {Vol}}}

\newcommand{\at}{\mathbin{@}}

\newcommand{\from}{\colon}

\newtheorem{corollary}{Corollary}[section]

\newtheorem{theorem}{Theorem}[section]

\newtheorem{definition}{Definition}[section]

\newtheorem{lemma}{Lemma}[section]

\newtheorem{proposition}{Proposition}[section]

\newtheorem*{claim}{Claim}

\newtheorem*{remark}{Remark}

\newtheoremstyle{rr}% name
  {3pt}%      Space above
  {3pt}%      Space below
  {\slshape}%         Body font
  {}%         Indent amount (empty = no indent, \parindent = para indent)
  {\bfseries}% Thm head font
  {:}%        Punctuation after thm head
  {.5em}%     Space after thm head: " " = normal interword space;
  {}%         Thm head spec (can be left empty, meaning `normal')

\theoremstyle{rr}

\newtheorem*{random}{Randomization}

\begin{document}

\let\johnny\thefootnote
\renewcommand{\thefootnote}{}

\footnotetext{JK is supported by NSF grant number DMS 0905812}
\let\thefootnote\johnny

\title[The Ehrenpreis conjecture] {The good pants homology and the Ehrenpreis conjecture}

\author[Kahn]{Jeremy Kahn}

\address{Mathematics Department \\  Brown University \\ Providence, 02912 \\  RI, USA}

\email{kahn@math.brown.edu}

\author[Markovic]{Vladimir Markovic}

\address{Mathematics Department  \\ California Institute of Technology \\ Pasadena, 91105 \\ CA, USA}

\email{markovic@caltech.edu}

\today


\begin{abstract} We develop the notion of the good pants homology  and show that it agrees with the standard  homology on closed surfaces 
(good pants are pairs of pants whose cuffs have the length nearly equal to some large number $R>0$).  Combined   with  our previous work on the Surface Subgroup Theorem \cite{kahn-markovic-1}, 
this yields a proof of the Ehrenpreis conjecture.

\end{abstract}

\maketitle

\section{Introduction} Let $S$ and $T$ denote two closed Riemann surfaces (all
closed surfaces in this paper are assumed to have genus at least $2$). The
well-known Ehrenpreis conjecture asserts that for any $K > 1$, one can find
finite degree covers $S_1$ and $T_1$, of $S$ and $T$ respectively, such that
there exists a $K$-quasiconformal map $f : S_1 \to T_1$. The purpose of this
paper is to prove this conjecture. Below we outline the strategy of the proof.

Let $R > 1$ and let $\Pi (R)$ be the hyperbolic pair of pants (with geodesic boundary) whose  three cuffs have
the length $R$. We define the surface $S(R)$ to be the genus two surface that
is obtained by gluing two copies of $\Pi (R)$ along the cuffs with the twist
parameter equal to $+ 1$ (these are the Fenchel-Nielsen coordinates for 
$S(R)$). By $\Orb (R)$ we denote the quotient orbifold of the surface $S(R)$
(the quotient of $S (R)$ by the group of automorphisms of $S (R)$). For a
fixed $R > 1$, we sometimes refer to $\Orb (R)$ as the model orbifold. The
following theorem is the main result of this paper.

\begin{theorem}
  \label{thm-main-1}Let $S$ be a closed hyperbolic Riemann surface. Then for
  every $K > 1$, there exists $R_0 = R_0 (K, S) > 0$ such that for every $R >
  R_0$ there are finite covers $S_1$ and $O_1$ of the surface $S$ and the
  model orbifold $\Orb (R)$ respectively, and a $K$-quasiconformal map $f :
  S_1 \to O_1$.
\end{theorem}

The Ehrenpreis conjecture is an immediate corollary of this theorem.

\begin{corollary}
  \label{cor-main-1}Let $S$ and $T$ denote two closed Riemann surfaces. For
  any $K > 1$, one can find finite degree covers $S_1$ and $T_1$ of $S$ and
  $T$ respectively, such that there exists a $K$-quasiconformal map $f : S_1
  \to T_1$.
\end{corollary}

\begin{proof}
  Fix $K > 1$. It follows from Theorem \ref{thm-main-1} that for $R$ large
  enough, there exist
  \begin{enumerate}
    \item Finite covers $S_1$, $T_1$, of $S$ and $T$ respectively, and
    
    \item Finite covers $O_1$ and $O'_1$ of $\Orb (R)$,
    
    \item $\sqrt{K}$-quasiconformal mappings $f : S_1 \to O_1$, and $g : T_1
    \to O'_1$.
  \end{enumerate}
  Let $O_2$ denote a common finite cover of $O_1$ and $O'_1$ (such $O_2$
  exists since $O_1$ and $O'_1$ are covers of the same orbifold $\Orb (R)$).
  Then there are finite covers $S_2$ and $T_2$, of $S_1$ and $T_1$,
  respectively, and the $\sqrt{K}$-quasiconformal maps $\wt f : S_2 \to O_2$,
  and $\wt g : T_2 \to O_2$, that are the lifts of $f$ and $g$. Then $\wt g^{-
  1} \circ \wt f : S_2 \to T_2$ is $K$-quasiconformal map, which proves the
  corollary.
\end{proof}

In the remainder of the paper we prove Theorem \ref{thm-main-1}. This paper builds on our previous paper \cite{kahn-markovic-1} where we used immersed skew pants in a given hyperbolic 3-manifold to prove the Surface Subgroup Theorem. We note that Lewis Bowen \cite{bowen} was the first to attempt to build finite covers of Riemann surfaces by putting together immersed pairs of pants.
We also note that it follows from the work of Danny Calegari \cite{calegari} that the pants homology is equal to the standard homology. 
This means that every sum of closed curves on a closed surface 
$S$ that is zero in the standard homology ${\bf H}_1(S)$ is the boundary of a sum of immersed pairs of pants in $S$.

We are grateful to Lewis Bowen for carefully reading the manuscript and suggesting numerous improvements and corrections. The second named author would like to acknowledge that the first named author has done most of the work in the second part of the paper concerning the Correction Theory.

\subsection*{Outline} In our previous paper \cite{kahn-markovic-1} we proved a theorem very similar to Theorem \ref{thm-main-1}, namely
that given a closed hyperbolic 3-manifold $\M$, and $K>1$, $R>R_0(K,\M)$, we can find a finite cover $O_1$ 
of $\Orb(R)$ and a map $f:O_1 \to \M$ that lifts to a map  $\wt{f}:\Ha \to  \Hao$ such that $\partial \wt{f}:\partial{\Ha} \to \partial{\Hao}$
has a $K$-quasiconformal extension.   We proved that theorem by finding a large collection of  ``skew pairs of pants" whose cuffs have complex half-lengths close to $R$,
and which are ``evenly distributed" around every closed geodesic that appears as a boundary.

We then assemble these pants by (taking two copies of each and then) pairing off the pants that have a given geodesic as boundary, so that the resulting complex twist-bends (or reduced Fenchel-Nielsen coordinates) are close to 1. We can then construct a cover $O_1$ of Orb(R) and a function $f \from O_1 \to M$ whose image is the closed surface that results from this pairing. The function $f$ will then have  the desired property.

We would like to proceed in the same manner in dimension 2, that is when a 3-manifold $\M$ is replaced with a closed surface $\Su$. 
We can, as before, find a collection of good immersed pants (with cuff length close to $R$) that is ``evenly distributed" around each  good geodesic (of length close to $R$)
that appears as boundary. If and when we can assemble the pants to form a closed surface with (real) twists close to 1, we will have produced a $K$-quasiconformal immersion of a cover of $\Orb(R)$ into $\Su$.

There is only one minor detail: the unit normal bundle of a closed geodesic in $\Su$ has two components. In other words, an immersed pair of pants that has a closed geodesic 
$\gamma$ as a boundary can lie on one of the two sides of $\gamma$. If, in our formal sum of pants, we find we have more pants on one side of $\gamma$ than the other, then we have no chance to form a closed surface out of this formal sum of pants. It is for this very reason that the Ehrenpreis Conjecture is more difficult to prove than the Surface Subgroup Theorem.

Because our initial collection is evenly distributed, there are almost the same number of good  pants on both sides of any good geodesic, so it is natural  to try to adjust the number of pairs of pants so that it is balanced (with the same number of pants on both sides of each geodesic). This leads us to look for a ``correction function" $\phi$ from formal sums of (normally) oriented closed geodesics (representing the imbalance) to formal sums of good pants, such that the boundary of $\phi(X)$ is $X$.

The existence of this correction function then implies that ``good boundary is boundary good", that is any sum of good geodesics that is a boundary in $H_1(\Su)$ (the first homology group of $\Su$)
is the boundary of a sum of good pants. Thus we define the good pants homology to be formal sums of good geodesics (with length close to $R$) modulo boundaries of formal sums of good pants. We would like to prove that the good pants homology is the standard homology on good curves.

The natural approach is to show that any good curve is homologous in good pants homology to a formal sum of $2g$ particular good curves that represent a basis in $H_1(\Su)$ ($g$ is the genus of $\Su$).
That is, we want to show that there are $\{ h_1,...,h_{2g} \}$ good curves that generate $H_1(\Su)$ (here $H_1(\Su)$ is taken with rational coefficients) such that every good curve $\gamma$ is homologous in the good pants homology to a formal sum $\sum a_i h_i$, for some $a_i \in \Q$. Then any sum of good curves is likewise homologous to a sum of good generators $h_i$, but if the original sum of good curves is zero in $H_1(\Su)$ then the corresponding sum of $h_i$'s is zero as well.

To prove the Good Correction Theorem , we must first develop the theory of good pants homology. Let $*$ denote a base point on $\Su$. For $A \in \pi_1(\Su,*) \setminus \{\id \}$, we let $[A]$ denote the closed geodesic freely homotopic to $A$. Our theory begins with the Algebraic Square Lemma, which states that, under certain geometric conditions, 
$$
\sum_{i,j=0,1} (-1)^{i+j}[A_iUB_jV] = 0,
$$
in the good pants homology. (The curves $[A_iUB_jV]$ must be good curves, the words  $A_iUB_jV$ reasonably efficient, and the length of $U$ and $V$ sufficiently large). This then permits us to define, for $A,T \in \pi_1(\Su,*)$,
$$
A_T=\frac{1}{2}\left( [TAT^{-1}U]-[TA^{-1}T^{-1}U] \right),
$$
where $U$ is fairly arbitrary. Then $A_T$ in good pants homology is independent of the choice of $U$.

We then show through a series of lemmas that 
$(XY)_T=X_T+Y_T$ in good pants homology, and therefore 
$$
X_T=\sum\sigma(j)(g_{i_{j}})_T,
$$
where by $X=g^{\sigma(1)}_{i_{1}}...g^{\sigma(k)}_{i_{k}}$ we have written $X$ as a product of generators (here  $g_1,...,g_{2g}$ are the generators for $\pi_1(\Su,*)$ and $\sigma(j)=\pm 1$). 
With a little more work we can show 
$$
[X]=\sum \sigma(j)(g_{i_{j}})_T
$$
as well, and thus we can correct any good curve to good generators. 

We are then finished except for one last step: We must show that our correction function, which gives an explicit sum of pants with a given boundary, is well-behaved in that it maps sums of curves, with bounded weight on each curve, to sums of pants, with bounded weight on each pair of pants. We call such a function semirandom, because if we pick a curve at random, the expected consumption of a given pair of pants is not much more than if we picked the pair of pants at random.

We  define the correction function implicitly, through a series of lemmas, each of which asserts the existence of a formal sum of pants with given boundary, and which is in principle constructive.
The notion of being semirandom is sufficiently natural to permit  us to say that the basic operations, such as the group law, or forming $[A]$ from $A$, as well as composition and formal addition, are all semirandom. So in order to verify that our correction function is semirandom, we need only to go through each lemma observing that the function we have defined is built out of the functions that we have previously defined  using the standard operations which we have proved are semirandom.

To make the paper as easy to read as possible, we have relegated the verification of semi-randomness to the ``Randomization remarks" which follow our homological lemmas, and which use the notation and results (that basic constructions are semirandom) that we have placed in the Appendix. We strongly recommend that the reader skip over these Randomization remarks in the first reading, and to interpret the word ``random" in the text to simply mean ``arbitrary".

\subsection*{A word on notation} When we use the letter K to denote a constant then we mean a universal constant or if K depends on parameters $X,Y,Z...$ then we write $K(X,Y,Z...)$ or we may leave out some of the parameters. In sections where we fixed certain parameters we may leave out the dependence of constants on these parameters.

\section{Constructing good covers of a Riemann surface}

\subsection{The reduced Fenchel-Nielsen coordinates and the model orbifolds}

Let $S^0$ be an oriented closed topological surface with a given pants
decomposition $\Col$, where $\Col$ is a maximal collection of disjoint simple
closed curves that cut $S^0$ into the corresponding pairs of pants. We will say that
$\Col$ makes $S^0$ into a panted surface.

Denote by $\Col^{\ast}$ the set of oriented curves from $\Col$ (each curve in $\Col$ is
taken with both orientations). The set of pairs $(\Pi, C^{\ast})$, where $\Pi$
is a pair of pants from the pants decomposition and $C^{\ast} \in \Col^{\ast}$
is an oriented boundary cuff of $\Pi$, is called the set of marked pants and
it is denoted by $\Pant (S^0)$. For $C \in \Col$ there are exactly two pairs
$(\Pi_i, C^{\ast}_i) \in \Pant (S^0)$, $i = 1, 2$, such that $C^{\ast}_1$ and
$C^{\ast}_2$ are the two orientations on $C$ (note that $\Pi_1$ and $\Pi_2$
may agree as pairs of pants).

Let $(S, \Col)$ be a panted Riemann surface. Then for every cuff $C \in \Col$ we can define  the reduced Fenchel-Nielsen coordinates $( \hl (C), s (C))$ from
{\cite{kahn-markovic-1}}. Here $\hl (C)$ is the half-length of the geodesic
homotopic to $C$, and $s (C) \in \R / \hl(C) \Z$ is the reduced twist parameter
which lives in the circle $\R / \hl(C)  \Z$ (when we write $s(C) = x \in \R$, we
really mean $s (C) \equiv  x  \, \text{mod} ( \hl (C) \Z)$). The following theorem was
proved in {\cite{kahn-markovic-1}} (see Theorem 2.1 and Corollary 2.1 in
{\cite{kahn-markovic-1}}).

\begin{theorem}
  \label{thm-geometry}There exist constants $\wh \epsilon, \wh R > 0$ such
  that the following holds.  Let $S$ denote a panted Riemann
  surface whose reduced Fenchel-Nielsen coordinates satisfy the inequalities.
  \[ | \hl (C) - R| < \epsilon, \hspace{0.25em} \hspace{0.25em} \text{and}
     \hspace{0.25em} \hspace{0.25em} |s (C) - 1| < \frac{\epsilon}{R}, \]
  {\noindent}for some $\wh \epsilon > \epsilon > 0$ and $R > \wh R$. Then
  there exists a marked surface $M_R$, with the reduced Fenchel-Nielsen
  coordinates $\hl (C) = R$ and $s (C) = 1$, and a $K$-quasiconformal map $f :
  S \to M_R$, where
  \[ K = \frac{\wh \epsilon + \epsilon}{\wh \epsilon - \epsilon} . \]
\end{theorem}

Let $R > 1$, and let $\Orb (R)$ denote the corresponding model orbifold
(defined in the introduction). In the next subsection we will see that the
significance of the above theorem comes from the observation that any Riemann
surface $M_R$ with  reduced Fenchel-Nielsen coordinates $\hl (C) = R$ and
$s (C) = 1$ is a finite cover of $\Orb (R)$.

\subsection{A proof of Theorem \ref{thm-main-1}} Below we state the theorem
which is then used to prove Theorem \ref{thm-main-1}.

\begin{theorem}
  \label{thm-main-2}Let $S$ denote a closed Riemann surface, and let $\epsilon
  > 1$. There exists $R (S, \epsilon) > 1$ such that for every $R > R (S,
  \epsilon)$ we can find a finite cover $S_1$ of $S$ that has a pants
  decomposition whose reduced Fenchel-Nielsen coordinates satisfy the
  inequalities
  \[ | \hl (C) - R| < \epsilon, \hspace{0.25em} \hspace{0.25em} \text{and}
     \hspace{0.25em} \hspace{0.25em} |s (C) - 1| < \frac{\epsilon}{R} . \]
\end{theorem}

This theorem will be proved at the end of section. 

$Proof$ $of$  Theorem \ref{thm-main-1}: Let $K > 1$. It follows from Theorem \ref{thm-geometry} that
for $\epsilon > 0$ small enough, and every $R$ large enough, there is a
$K$-quasiconformal map $f : S_1 \to M_R$, where $M_R$ is a Riemann surface
with reduced Fenchel-Nielsen coordinates $\hl (C) = R$ and $s (C) = 1$, and
$S_1$ is the finite cover of $S$ from Theorem \ref{thm-main-2}. Recall the
corresponding model orbifold $\Orb (R)$ (defined in the introduction). As we
observed, the surface $M_R$ is a finite cover of $\Orb (R)$. This completes
the proof of the theorem.
 
\subsection{The set of immersed pants in a given Riemann surface} From now on
$\Su = \Ha / \KG$ is a fixed closed Riemann surface and $\KG$ a suitable
Fuchsian group. By $\Gamma$ we denote the collection of oriented closed
geodesics in $\Su$. By $-\gamma$ we denote the opposite orientation of an oriented geodesic
$\gamma \in \Gamma$. We sometimes write $\gamma^{*} \in \Gamma$ to emphasize a choice of 
orientation.

Let $\Pi_0$ denote a hyperbolic pair of pants (that is $\Pi_0$ is equipped with
a hyperbolic metric such that the cuffs of $\Pi_0$ are geodesics). Let $f :
\Pi_0 \to \Su$ be a local isometry (such an $f$ must be an immersion). We say
that $\Pi = (f, \Pi_0)$ is an immersed pair of pants in $\Su$. The set of all
immersed pants in $\Su$ is denoted by $\Pant$. Let $C^{\ast}$ denote an
oriented cuff of $\Pi_0$ (the geodesic $C^{\ast}$ is oriented as a boundary
component of $\Pi_0$). Set $f (C^{\ast}) = \gamma \in \Gamma$.
We say that $\gamma$ is an oriented cuff of $\Pi$. The set of such
pairs $(\Pi, \gamma)$ is called the set of marked immersed pants and
denoted by $\Pant^{\ast}$. The half-length $\hl (\gamma)$ associated to the
cuff $\gamma$ of $\Pi$ is defined as the half-length $\hl (C)$ associated to
the cuff $C$ of $\Pi_0$.

Let $\gamma \in \Gamma$ be a closed oriented geodesic in $\Su$. Denote by $\NB
(\gamma)$ the unit normal bundle of $\gamma$. Elements of $\NB (\gamma)$ are
pairs $(p, v)$, where $p \in \gamma$ and $v$ is a unit vector at $p$ that is
orthogonal to $\gamma$. The bundle $\NB (\gamma)$ is a differentiable manifold
that has two components which we denote by $\NB_+ (\gamma)$ and $\NB_{-}
(\gamma)$ (the right-hand side and the left-hand side components). Each
component inherits the metric from the geodesic $\gamma$, and both $\NB_{+}
(\gamma)$ and $\NB_{-} (\gamma)$ are isometric (as metric spaces) to the circle
of length $2 \hl (\gamma)$. By $\dist$ we denote the corresponding distance
functions on $\NB_{+} (\gamma)$ and $\NB_{-} (\gamma)$.

Let $(p, v) \in \NB (\gamma)$, and denote by $(p_1, v_1) \in \NB (\gamma)$ the
pair such that $(p, v)$ and $(p_1, v_1)$ belong to the same component of $\NB
(\gamma)$, and $\dist (p, p_1) = \hl (\gamma)$. Set $A (p, v) = (p_1, v_1)$.
Then $A$ is an involution that leaves invariant each component of $\NB
(\gamma)$. Define the bundle $\NB ( \sqrt{\gamma}) = \NB (\gamma) / A$. The
two components are denoted by $\NB_{+} ( \sqrt{\gamma})$ and $\NB_{-} (
\sqrt{\gamma})$, and both are isometric (as metric spaces) to the circle of
length $\hl (\gamma)$. The disjoint union of all such bundles is denoted by
$\NB ( \sqrt{\Gamma})$.

We now define the foot of a pair of pants. Let $\Pi \in \Pant$ be an immersed pants and $f : \Pi_0 \to \Pi$ the
corresponding local isometry. Let $C^{\ast}$ denote an oriented cuff of
$\Pi_0$ and $\gamma = f (C^{\ast})$. Let $C_1$ and $C_2$ denote the
other two cuffs of $\Pi_0$, and let $p'_1, p'_2 \in C^{\ast}$ denote the two
points that are the feet of the shortest geodesic segments in $\Pi_0$ that
connect $C$ and $C_1$, and $C$ and $C_2$, respectively. Let $v'_1$ denote the
unit vector at $p'_1$ that is orthogonal to $C$ and points towards the
interior of $\Pi_0$. We define $v'_2$ similarly. Set $(p_1, v_1) = f_{*} (p'_1,
v'_1)$ and $(p_2, v_2) = f_{*} (p'_2, v'_2)$. Then $(p_1, v_1)$ and $(p_2, v_2)$
are in the same component of $\NB (\gamma)$, and the points $p_1$ and $p_2$
separate $\gamma$ into two intervals of length $\hl (\gamma)$. Therefore, the
vectors $(p_1, v_1)$ and $(p_2, v_2)$ represent the same point $(p, v) \in \NB
( \sqrt{\Gamma})$, and we set
\[ \foot (\Pi, \gamma) = (p, v) \in \NB (\sqrt\gamma) . \]
We call the vector $(p, v)$ the foot of the immersed pair of pants $\Pi$ at
the cuff $\gamma$. This defines the map
\[ \foot : \Pant^{\ast} \to \NB ( \sqrt{\Gamma}) . \]

\subsection{Measures on pants and the $\ph$ operator} By $\Mes ( \Pant)$ we
denote the space of real valued Borel measures with finite support on the set
of immersed pants $\Pant$, and likewise, by $\Mes ( \NB ( \sqrt{\Gamma}))$ we
denote the space of real valued Borel measures with compact support on the
manifold $\NB ( \sqrt{\Gamma})$ (a measure from $\Mes^+ ( \NB (
\sqrt{\Gamma}))$ has a compact support if and only if its support is contained
in at most finitely many bundles $\NB ( \sqrt{\gamma}) \subset \NB (
\sqrt{\Gamma})$). By $\Mes^+ ( \Pant)$ and $\Mes^+ ( \NB ( \sqrt{\Gamma}))$,
we denote the corresponding spaces of positive measures.

We define the operator
\[ \ph : \Mes ( \Pant) \to \Mes ( \NB ( \sqrt{\Gamma})), \]
as follows. The set $\Pant$ is a countable set, so every measure from $\mu \in
\Mes ( \Pant)$ is determined by its value $\mu (\Pi)$ on every $\Pi \in
\Pant$. Let $\Pi \in \Pant$ and let $\gamma_i \in \Gamma$, $i =
0, 1, 2$, denote the corresponding oriented geodesics so that $(\Pi,
\gamma_i) \in \Pant^{\ast}$. Let $\alpha^{\Pi}_i \in \Mes ( \NB (
\sqrt{\Gamma}))$ be the atomic measure supported at the point $\foot (\Pi,
\gamma_i) \in \NB ( \sqrt{\gamma_i})$, where the mass of the atom is
equal to $1$. Let
\[ \alpha^{\Pi} = \sum_{i = 0}^2 \alpha^{\Pi}_i, \]
and define
\[ \ph \mu = \sum_{\Pi \in \Pant} \mu (\Pi) \alpha^{\Pi} . \]
We call this the $\ph$ operator on measures. If $\mu \in \Mes^+ ( \Pant)$,
then $\ph \mu \in \Mes^+ ( \NB ( \sqrt{\Gamma}))$, and the total measure of
$\ph \mu$ is three times the total measure of $\mu$.

We recall the notion of equivalent measures from Section 3 in
{\cite{kahn-markovic-1}}. Let $(X, d)$ be a metric space. By $\Mes^+ (X)$ we
denote the space of positive, Borel measures on $X$ with compact support. For
$A \subset X$ and $\delta > 0$ let
\[ \Ne_{\delta} (A) =\{x \in X : \hspace{1em} \text{there exists} \hspace{1em}
   a \in A \hspace{1em} \text{such that} \hspace{1em} d (x, a) \le \delta\},
\] be the $\delta$-neighbourhood of $A$.

\begin{definition}
  Let $\mu, \nu \in \Mes^+ (X)$ be two measures such that $\mu (X) = \nu (X)$,
  and let $\delta > 0$. Suppose that for every Borel set $A \subset X$ we have
  $\mu (A) \le \nu ( \Ne_{\delta} (A))$. Then we say that $\mu$ and $\nu$ are
  $\delta$-equivalent measures.
\end{definition}

\begin{remark}
We observe that this definition is symmetric because $\nu (A) \le \mu(N_\delta(A))$ whenever $\mu(X \setminus N_\delta(A)) \le \nu (N_\delta(X \setminus N_\delta(A))$.
\end{remark}

In our applications $X$ will be either a 1-torus (a circle) or $\R$. In this case, 
$\mu(A) \le \nu(\Ne_{\delta}(A))$ for all Borel sets $A$ if it holds for all intervals $A$.
We recall that if $\mu$ and $\nu$ are discrete and integer valued  measures that are $\epsilon$-equivalent
then there is a ``matching"  between $\mu$ and $\nu$ that matches each point $x$ to a point within $\epsilon$ of $x$.
In other words, letting $E_n$ be $\{1,2...,n\}$ with the counting measure, if $\mu=f_{*} E_n$ and $\nu=g_{*} E_n$ and 
$\mu$ and $\nu$ are $\epsilon$-equivalent, then we can find $\sigma:E_n \to E_n$ such that $d(f(k),g(\sigma(k))) \le \epsilon$ for each
$k \in E_n$. This holds when $\mu$ and $\nu$ are measures on any metric space (by the Hall's Marriage Theorem) and is even more elementary 
when the metric space is a 1-torus (that is a circle) or an interval.     
  
We observe that if $\mu$ and $\nu$ are $\epsilon_1$-equivalent and $\nu$ and $\rho$ are $\epsilon_2$-equivalent then
$\mu$ and $\rho$ are $(\epsilon_1+\epsilon_2)$-equivalent.

Let $\gamma \in \Gamma$ and $\alpha \in \Mes ( \NB ( \sqrt{\gamma}))$. The
bundle $\NB ( \sqrt{\gamma})$ has the two components $\NB_+ ( \sqrt{\gamma})$
and $\NB_- ( \sqrt{\gamma})$ (the right-hand and left-hand side components),
each isometric to the circle of length $\hl (\gamma)$. The restrictions of
$\alpha$ to $\NB_+ ( \sqrt{\gamma})$ and $\NB_- ( \sqrt{\gamma})$ are denoted
by $\alpha_+$ and $\alpha_-$ respectively. In particular, by $\ph_+ \mu$ and
$\ph_- \mu$ we denote the decomposition of the measure $\ph \mu$.

\begin{definition}
  Fix $\gamma \in \Gamma$, and let $\alpha, \beta \in \Mes^+ ( \NB (
  \sqrt{\gamma}))$. We say that $\alpha$ and $\beta$ are $\delta$-equivalent
  if the pairs of measures $\alpha_+$ and $\beta_+$, and $\alpha_-$ and
  $\beta_-$, are respectively $\delta$-equivalent. Also, by $\lambda (\gamma)
  \in \Mes^+ ( \NB ( \sqrt{\gamma}))$, we denote the measure whose components
  $\lambda_+ (\gamma)$ and $\lambda_- (\gamma)$, are the standard Lebesgue
  measures on the metric spaces $\NB ( \sqrt{\gamma})^+$ and $\NB (
  \sqrt{\gamma})^-$, respectively. In particular, the measure $\lambda
  (\gamma)$ is invariant under the full group of isometries of $\NB (
  \sqrt{\gamma})$.
\end{definition}

Let $\epsilon, R > 0$. By $\Gamma_{\epsilon, R} \subset \Gamma$ we denote the
closed geodesics in the Riemann surface $\Su$ whose half-length is in the
interval $[R - \epsilon, R + \epsilon]$. We define $\Pant_{\epsilon, R}
\subset \Pant$, as the set of immersed pants whose cuffs are in
$\Gamma_{\epsilon, R}$. We will often call $\Gamma_{\epsilon, R}$ the set of ``good curves" and $\Pant_{\epsilon, R}$ the set of ``good pants''.
Our aim is to prove the following theorem, which in
turn yields the proof of Theorem \ref{thm-main-2} stated above.

We adopt the following convention. In the rest of the paper by $P (R)$ we
denote a polynomial in $R$ whose degree and coefficients depend only on the
choice of $\epsilon$ and the surface $\Su$.

\begin{theorem} \label{thm-mes} Let $\epsilon > 0$. There exists $q > 0$
(depending only  on the surface $\Su$ and $\epsilon$)
so that for every $R > 0$ large enough, there exists a measure 
$\mu \in \Mes^+ ( \Pant_{\epsilon, R})$ with the following properties. Let $\gamma
\in \Gamma_{\epsilon, R}$ and let $\ph \mu (\gamma)$ denote the restriction
of $\ph \mu$ to $\NB ( \sqrt{\gamma})$. If $\ph \mu (\gamma)$ is not the
zero measure then there exists a constant $K_{\gamma} > 0$ such that the
measures $\ph \mu (\gamma)$ and $K_{\gamma} \lambda (\gamma)$ are $P (R)
e^{- q R}$-equivalent.
\end{theorem}

\begin{remark} We say that $\alpha \in \Mes (\NB(\sqrt{\gamma}))$ is $\delta$ symmetric (for $\delta>0$)  if for every 
isometry $\iota:\NB(\sqrt{\gamma}) \to \NB(\sqrt{\gamma})$ the measures $\alpha$ and $\iota_{*} \alpha$ are $\delta$-equivalent. 
If $\alpha$ and  $K_{\gamma} \lambda (\gamma)$  are $\delta$-equivalent then $\alpha$ is $2\delta$ symmetric because $\lambda(\gamma)$ is $0$ symmetric.
\end{remark}

The proof of Theorem \ref{thm-main-2} follows from Theorem \ref{thm-mes} in
the same way as it was done in Section 3 in {\cite{kahn-markovic-1}}. The
brief outline is as follows. We may assume that the measure $\mu$ from the
above theorem has integer coefficients. Then we may think of $\mu$ as a formal
sum of immersed pants such that the restriction of the measure $\ph \mu$ on
any $\NB ( \sqrt{\gamma})$ is $P (R) e^{- q R}$-equivalent with
some multiple of the Lebesgue measure (unless the restriction is the zero
measure). Considering $\mu$ as the multiset (formally one may use the notion
of a labelled collection of immersed pants) we can then define a pairing between
marked immersed pants, such that two marked pants $(\Pi_1, \gamma_1)$
and $(\Pi_2, \gamma_2)$, are paired if $\gamma_1 = -
\gamma_2$, and the corresponding twist parameter between these two
pairs is $P (R) e^{- pR}$ close to $+ 1$, for some universal constant $p > 0$.
After gluing all the marked pants we have paired we obtain
the finite cover from Theorem \ref{thm-main-2}.

\section{Equidistribution and self-Correction}

We introduce in this section the Equidistribution Theorem (Theorem \ref{thm-equidistribution}) and the Correction Theorem
(Theorem \ref{thm-correction}), and we use them to prove Theorem \ref{thm-mes}. Theorem \ref{thm-equidistribution} follows from our previous work \cite{kahn-markovic-1},
and provides us with an evenly distributed collection of good pants. The Correction Theorem allows us to correct the slight imbalance (as described in the introduction) that may be found in the original collection of pants.

\subsection{Generating  essential immersed pants in $\Su$} Let us first describe how we generate a pair of pants from a $\Theta$-graph. Let $\Pi$ denote a pair of pants whose three cuffs have the same length, and  let $\omega:\Pi \to \Pi$ denote the standard (orientation preserving) isometry of order three that permutes the cuffs of $\Pi$. Let $a$ and $b$ be the fixed points of $\omega$. Let $\gamma_0 \subset \Pi$ denote a simple oriented geodesic arc that starts at $a$ and terminates at $b$.  Set $\omega^{i}(\gamma_0)=\gamma_i$.  The union of two different arcs $\gamma_i$ and $\gamma_j$ is a closed curve in $\Pi $ homotopic to a cuff. One can think of the union of these three segments as the spine of $\Pi$. Moreover, there is an obvious projection from $\Pi$ to the spine $\gamma=\gamma_0 \cup \gamma_1 \cup \gamma_2$, and this projection is a homotopy equivalence.

Let $p$ and $q$ be two (not necessarily) distinct points in $\Su$, and let $\alpha_0$, $\alpha_1$, and $\alpha_2$, denote three distinct 
oriented geodesic arcs, each starting at $p$ and terminating at $q$.  We let  $\alpha= \alpha_0 \cup \alpha_1 \cup \alpha_2$ (we call $\alpha$ a $\Theta$-graph). Let $i(\alpha_j) \in \TB_{p} \Su$ and $t(\alpha_j) \in \TB_{q} \Su$ denote  the initial and terminal unit tangent vectors to $\alpha_j$ at $p$ and $q$ respectively. Suppose that the triples of vectors $(i(\alpha_0),i(\alpha_1),i(\alpha_2))$ and $(t(\alpha_0), t(\alpha_1),t(\alpha_2))$, have opposite cyclic orders on the unit circle. 

We define the map $f:\Pi \to \Su$ by first projecting the pants $\Pi$ onto its spine $\gamma$, and then by mapping  $\gamma_j$ onto $\alpha_j$ by a map that is a local (orientation preserving) homeomorphism.  Then the induced conjugacy class of maps  $f_*:\pi_1(\Pi) \to \pi_1(\Su)$ is  injective. 

Moreover we can homotop the map $f$ to an immersion $g:\Pi \to \Su$ as follows. We can write the pants $\Pi$ as a (non-disjoint) union of three strips $G_0,G_1,G_2$, where each $G_i$ is a fattening of the geodesic arc $\gamma_i$. Then we define a map $g_i:G_i \to \Su$ to be a local homeomorphism on each $G_i$ by extending the restriction of the map $f$ on $\gamma_i$. 
The condition on the cyclic order of the $\alpha_i$'s at the two vertices enables us to define $g_i$ and $g_j$ on $G_i$ and $G_j$ respectively, so that $g_i=g_j$ on $G_i \cap G_j$. Then we set $g=g_i$.  

We say that  $g:\Pi \to \Su$ is the essential immersed pair of pants generated by the three geodesic segments $\alpha_0,\alpha_1$ and $\alpha_2$.

Often we will be given two geodesic segments, say $\alpha_0$ and $\alpha_1$, and then find a third geodesic segment $\alpha_2$ so to obtain an essential immersed pair of pants. 
We then say that $\alpha_2$ is a third connection. In this paper we will often be given a closed geodesic $C$ on a Riemann surface $\Su$, with two marked points $p,q \in C$. Then every geodesic arc $\alpha$ between $p$ and $q$, that meets $p$ and $q$ at the same sides of $C$ will be called a third connection, since then $C$ and $\alpha$ generate an immersed pair of pants as described above.
In particular, this represents an efficient way of generating pants that contain a given closed geodesic $C$ as its cuff.

\subsection{Preliminary lemmas} Let $\TB \Ha$ denote the unit tangent bundle.
Elements of $\TB ( \Ha)$ are pairs $(p, u)$, where $p \in \Ha$ and $u \in
\TB_p \Ha$. Sometimes we write $u = (p, u)$ and refer to $u$ as a unit vector
in $\TB_p \Ha$. By $\TB \Su$ we denote the unit tangent bundle over $\Su$. For
$u, v \in \TB_p \Ha$ we let $\Ang (u, v)$ denote the unoriented angle between
$u$ and $v$. The function $\Ang$ takes values in the interval $[0, \pi]$.

For $L, \epsilon > 0$, and $(p, u), (q, v) \in \TB \Su$, we let
$\Conn_{\epsilon, L} ((p, u), (q, v))$ be the set of unit speed geodesic
segments $\gamma : [0, l] \to \Su$ such that
\begin{itemize}
  \item $\gamma (0) = p$, and $\gamma (l) = q$,
  
  \item $|l - L| < \epsilon$,
  
  \item $\Ang (u,  \gamma' (0)), \Ang (v, \gamma' (l)) < \epsilon$.
\end{itemize}

The next lemma will be referred to as the Connection Lemma. 
It provides a good lower bound on the size of the Connection set we have just defined.
This lemma also follows from discussion in the appendix.

\begin{lemma}
  \label{lemma-connection} Given $\epsilon > 0$, we can find $L_0 = L_0 ( \Su,
  \epsilon)$ such that for any $L > L_0$, and any two vectors $(p,u)$ and $(q,v)$, the set $\Conn_{\epsilon, L}\big( (p,u),(q,v)  \big)$ is
  non-empty and
  \[ \big| \Conn_{\epsilon, L}\big( (p,u),(q,v)  \big) \big| \ge e^{L - L_0} . \]
\end{lemma}

\begin{proof} The unit tangent bundle $\TB \Su$ is naturally identified with $G  \backslash \PSLR$, where $G$ is a lattice. By $\dist$ we denote a distance function on $\TB \Su$ (we defined $\dist$ explicitly later in the paper). Then we can find a neighbourhood $U$ of the identity in $\PSLR$ so that if $(q,v) \in  \TB \Su = G  \backslash \PSLR$, and $\xi \in U$, then $\dist((q,v),(q,v) \cdot \xi)<\frac{\epsilon}{16}$. 

We let $\psi:U \to [0,\infty)$ be a $C^{\infty}$ function with compact support in $U$, with $\int_{U} \psi=1$. For any $(q,v) \in \TB \Su$ we let $\Ne_U(q,v)=\{(q,v) \cd s: s \in U \}$, and let 
$\psi_{(q,v)}((q,v)\cd s)=\psi(s)$ on $\Ne_U(q,v)$. (If $\epsilon$ is small, then $s \mapsto  (q,v) \cd s$ is injective). The $C^{k}$ norm of $\psi_{(q,v)}$ is independent of $(q,v)$. 

Let $g_t:\TB \Su \to \TB \Su$ be the geodesic flow. By uniform mixing for uniformly $C^{\infty}$ functions on $\Su$,
\begin{equation}\label{jed-1}
\int\limits_{\TB \Su} \psi_{(q,v)}(g_t(x,w))\psi_{(p,u)}(x,w) \, d(x,w) \to 1,
\end{equation}
uniformly in $(p,u)$ and $(q,v)$, as $t \to \infty$ (we always assume that the Liouville measure is normalized so that the total measure of the tangent bundle is one). If $\psi_{(q,v)}(g_t(x,w))\psi_{(p,u)}(x,w)>0$, then $(x,w) \in \Ne_U(p,u)$ and $g_t(x,w) \in \Ne_U(q,v)$. 

The segment $g_{[0,t]}(x,w)$ is $\epsilon$-nearly homotopic (see the definition after this proof) to a unique geodesic segment $\alpha$ connecting $p$ and $q$. The reader can verify that $\alpha \in \Conn_{\epsilon,t}((p,u),(q,v))$. We let $E_{\alpha,t} \subset \Ne_U(p,u)$ be the set of $(x,w)$ for which $g_t(x,w) \in \Ne_U(q,v)$, and $g_{[0,t]}(x,w)$ is $\epsilon$-homotopic to $\alpha$.
Then by (\ref{jed-1}),
$$
\sum_{\alpha} \int\limits_{E_{\alpha,t}} \psi_{(q,v)}(g_t(x,w))\psi_{(p,u)}(x,w)\, d(x,w)=1+o(1).
$$

On the other hand, we can easily verify that $V(E_{\alpha,t}) \le K(\psi)e^{-t}$ (here $V(E_{\alpha,t})$ is the volume of $E_{\alpha,t}$). Hence

\begin{align*}
\int\limits_{E_{\alpha,t}}  \psi_{(q,v)}(g_t(x,w))\psi_{(p,u)}(x,w)\, d(x,w)   &\le \int\limits_{E_{\alpha,t}} (\sup_U  \psi)^{2}\, d(x,w) \\
&\le K(\psi)e^{-t}.
\end{align*}

So the number of good $\alpha$ is at least $K(\psi)e^{t}$, as long as $t$ is large given $\Su$ and $\epsilon$.

\end{proof}

\begin{definition}
  Let $E \ge 0$. We say that  two paths $A$ and $B$ in $\Ha$ are
  $E$-nearly homotopic if the distance between the endpoints of 
  $A$ and $B$ is at most $E$. Two paths on the closed surface
  $\Su$ are $E$-nearly homotopic if they have lifts to $\Ha$ that are $E$-nearly homotopic.
\end{definition}

The following  lemma gives the estimate for the number of good pants that are bounded by a given (good) cuff.

\begin{lemma}
  \label{lemma-number}Let $0<\epsilon<1$. We let $\gamma \in \Gamma_{\epsilon,
  R}$ and let $\Pant_{\epsilon, R} (\gamma)$ denote the set of pants in $\Pant_{\epsilon,R}$ that
  contain $\gamma$ as a cuff. Then
  \[ | \Pant_{\epsilon, R} (\gamma) | \asymp Re^R \]
  where the constant for $\asymp$ depends only on $\Su$ and $\epsilon$.
\end{lemma}

\begin{proof}
  For the upper bound, let $F_{\gamma}$ denote a set of $\lceil 2 R \rceil$
  evenly distributed points on $\gamma$. If $\Pi \in \Pant_{\epsilon, R}$ and
  if $\gamma \in \partial \Pi$, we let $\alpha$ be the geodesic segment in
  $\Pi$ that is orthogonal to $\gamma$ at its endpoints and simple on $\Pi$. 
  Then we let $\alpha'$ be a geodesic segment connecting two points of
  $F_{\gamma}$, such that $\alpha'$ is $\frac{1}{2}$-nearly homotopic to $\alpha$, and hence $\len (\alpha') \le \len (\alpha) + 1$. It can easily be verified that the length of $\alpha$ is at most $R+9$ 
thus the length of $\alpha'$ is at most $R+10$. 

We leave it to the reader to verify that the number of geodesic segments of length at most $L$ 
that connect two given points on the closed surface $\Su$ is at most $e^{R+\diam(\Su)} / \Area(\Su)$.
  
  If two pants produce identical $\alpha'$, then the two pants are the same.
  Between any two points of $F_{\gamma}$ there are at most $Ne^R$ such arcs
  $\alpha'$ (the constant $N$ depends only on $\Su$), and the endpoints of $\alpha'$ are $R \pm 1$ apart,
  so the total number of such arcs is at most $10 NRe^R$.
  
  For the lower bound: By Lemma \ref{lemma-connection} we can find $N ( \Su,
  \epsilon) e^R$ geodesic segments $\wh \alpha$ connecting two given
  diametrically opposite points of $\gamma$, of length within
  $\frac{\epsilon}{100}$ of $R + \log 4$, and such that the angle between $\wh
  \alpha$ and $\gamma$ is within $\frac{\epsilon}{100}$ of $\frac{\pi}{2}$.
  Here we assume that the two vectors (at the two diametrically opposite
  points) at $\gamma$ that are tangent to $\wh \alpha$ are on the same side of
  $\gamma$. Then to any such $\wh \alpha$ there is an
  $\frac{\epsilon}{10}$-nearly homotopic $\alpha$ with endpoints on $\gamma$
  (homotopic on $\Su$ through arcs with endpoints on $\gamma$) and such that
  $\alpha$ is orthogonal to $\gamma$. Each such $\alpha$ produces a pair of
  pants $\Pi \in \Pant_{\epsilon, R}$ that contains $\gamma$ as a cuff.
  
  Different $\alpha$'s give different pants. Two $\wh \alpha$ with the same
  $\alpha$ must have nearby endpoints along $\gamma$ and be $10 \epsilon$-nearly homotopic. So we get at least
  \[ \frac{2 RN ( \Su, \epsilon)}{10 \epsilon} e^R \]
  of the $\alpha$'s, and hence of the pants.
\end{proof}

\begin{remark}
  Let $M > 1$ and let $X_{\gamma} (M)$ denote the number of pants in $\Pant_{\epsilon,
  R} (\gamma)$ whose two other cuffs are in $\Gamma_{\frac{\epsilon}{M}, R}$.
  Then $X_{\gamma} (M) \asymp Re^R$. The upper bound follows from the upper bound of
  the lemma. If the segment $\wh \alpha$ is of length within
  $\frac{\epsilon}{100 M}$ of $2 R - \hl (\gamma) + \log 4$, and if that the
  angle between $\wh \alpha$ and $\gamma$ is within $\frac{\epsilon}{100}$ of
  $\frac{\pi}{2}$, then the induced pants have the desired property that the
  other two cuffs are in $\Gamma_{\frac{\epsilon}{M}, R}$. On the other hand
  it follows from the Connection Lemma that the number of such $\wh \alpha$'s
  is $\asymp$ to $N ( \Su, \epsilon, M) e^R$ for some constant $N ( \Su,
  \epsilon, M)$.
\end{remark}

\subsection{The Equidistribution Theorem} The following is the  Equidistribution Theorem. This is a stronger equidistribution result than the one proved in our previous paper
{\cite{kahn-markovic-1}}.

\begin{theorem}\label{thm-equidistribution} Let $\epsilon > 0$. Let $\mu$ be the measure on $\Pant_{\epsilon,R}$ that assigns to each pants in $\Pant_{\epsilon,R}$ the value $1$. 
Then for $R$ large enough the measure $\mu$ has the following properties:
\begin{enumerate}
\item $\mu (\Pant_{\epsilon, R}) \asymp e^{3R}$.
\item For every $\gamma \in \Gamma_{\epsilon, R}$ the measure $\ph_{\pm}\mu (\gamma)$ is $Ce^{- {q}  R}$ equivalent to 
$K^{\pm}_{\gamma} \lambda_{\pm} (\gamma)$, for some constants $K^+_{\gamma}$ and $K^-_{\gamma}$ that satisfy the inequality
\[ 
\bigg{|} \frac{K^+_{\gamma}}{K^-_{\gamma}} - 1 \bigg{|} < C e^{-q R}, 
\]
where $C,q>0$ depend only on $\Su$ and $\epsilon$.
\item Moreover, $K^{\pm}_{\gamma} \asymp Re^{R}$ for $\gamma \in \Gamma_{\epsilon, R}$.
\end{enumerate}

\end{theorem}   

\begin{proof} The well know result of Margulis \cite{margulis} asserts that 
$$
\big|\Gamma_{\epsilon,R}\big| \asymp \frac{e^{2R}}{R}. 
$$
The claim (1) of the theorem follows from this estimate and  Theorem \ref{lemma-number}. The claim (3) follows from the claim (2) and Lemma \ref{lemma-number}. Thus, it remains to prove (2).
To prove (2) we need be able to estimate the number of good pants that bound gamma and whose feet belong to a given interval of $\gamma$. Moreover, we need to show that the number of such pants that are to the left of $\gamma$ is very close to the number of such pants that are to the right of $\gamma$. We first explain how to effectively estimate the two numbers.

Let $\eta_1$ and $\eta_2$ be two geodesic arcs on $\Su$ that connect the same two points on $\Su$. Denote by $[\cdot \eta_1 \cdot \eta_2 \cdot]$ the corresponding  closed broken geodesic on $\Su$ (see the second paragraph of Section 4.1  for more on this notation).  We assume that $\eta_1$ and $\eta_2$ meet at the right angles and  that we can orient $[\cdot \eta_1 \cdot \eta_2 \cdot]$ so that  both right turns are ``to the right'', or both right turns are ``to the left''.

Let $\gamma$ be the closed geodesic freely homotopic to $[\cdot \eta_1 \cdot \eta_2 \cdot]$. Then we can write
$$
l(\gamma) = h(l (\eta_1), l (\eta_2) )
$$
for some smooth symmetric function $h$ (where $l$ denotes the length function), for which
$$
h(e_1, e_2) = e_1 + e_2 - \log 2 + O(e^{-\min(e_1, e_2)/2}) .
$$
(the function $h$ can be computed explicitly from the basic formulas in hyperbolic geometry).

Suppose that $\gamma$ is a (good) closed geodesic on $S$, and suppose that $\eta$ is a ``third connection" for $\gamma$. So $\eta$ is a geodesic segment which meets $\gamma$ at two points $x$ and $y$, and is orthogonal to $\gamma$ at $x$ and $y$, and lies on the same side of $\gamma$ at both points: the two normal vectors to $\gamma$ pointing into $\eta$ at $x$ and $y$ are on the same component of $N^1(\gamma)$. 
Then there is a unique pair of pants $\Pi$ for which $\gamma \in \partial \Pi$, and for which $\eta$ is an orthogeodesic on $\Pi$ that lifts to be embedded in $\Pi$. Letting $\sigma_1$ and $\sigma_2$ be the two segments of $\gamma \setminus \{x, y \}$, we find that the two other cuffs $\gamma_1, \gamma_2$ of $\Pi$ are freely homotopic to $[\cdot \sigma_i \cdot \eta^{\pm 1} \cdot]$. So we have
$$
l (\gamma_i) = h(l(\eta), l(\sigma_i) )
$$
where $h$ is defined as above. 
Moreover, the two feet of $\Pi$ on $\gamma$ lie at the two midpoints of $x$ and $y$ on $\gamma$. 
(Really, we should think of $x$ and $y$ as lying on the parameterizing 1-torus for $\gamma$, and likewise the feet, but we will say that they're on $\gamma$, by a mild abuse of notation). 

Suppose that $I \subset N^1(\sqrt \gamma)$ is an interval. So $I$ comprises a choice of component of $N^1(\gamma)$---a \emph{side} of $\gamma$---along with a pair of intervals in $\gamma$, of equal length and placed halfway along $\gamma$ from each other. We should think of $\gamma$ as long, say longer than 10, and $I$ as short, say shorter than 1/10. 

We define the region
$$
\Rel(\gamma, I) \subset \gamma \times \gamma \times \R^+
$$
as the set of $(x, y, l)$ for which 
$$
h(l,s_i) \in [2R - 2\epsilon, 2R + 2 \epsilon]
$$
and for which the two midpoints of $x$ and $y$ lie in the two intervals on $\gamma$ associated to $I$. 
Here $s_1, s_2$ are the lengths of the two arcs in $\gamma$ between $x$ and $y$. 
Then suppose $\eta$ is a third connection for $\gamma$ and that $\eta$ lies on the same side of $\gamma$ as $I$ does. Let $x$ and $y$ be the endpoints of $\eta$, and let $\Pi_\eta$ be the associated pair of pants for $\eta$ (and $\gamma$). 
Then $(x, y, l(\eta)) \in \Rel(\gamma, I)$ if and only if $\Pi_\eta$ is a good pair of pants, and the pair of feet of $\Pi_\eta$ on $N^1(\sqrt \gamma)$ lies in $I$.  
Thus, the number of good pants whose feet belong to the interval $I$ is equal to the  number of third connections $\eta$ (on a given side of $\gamma$) for which the associated triple $(x, y, l(\eta))$ lies in $\Rel(\gamma, I)$. 
So our goal is simply to count the number of third connections $\eta$ (on a given side of $\gamma$) for which the associated triple $(x, y, l(\eta))$ lies in $\Rel(\gamma, I)$.

One can see that the volume of $\Rel(\gamma,I)$ is on the order of $\epsilon^2 |I|$ as follows. For any choice of a point in $I$, $x$ and $y$ are determined by $s_0$, and possible pairs $(l,s_0)$ lie in a diamond of  size about $\epsilon$. Since the area of the diamond is about $\epsilon^2$ it follows that the volume of $\Rel(\gamma,I)$ is about $\epsilon^2|I|$.

If we denote by $C_{\gamma}$ the set of associated triples $(x,y,l(\eta))$ for all third connections $\eta$, our goal is simply to count $C_{\gamma}\cap \Rel(\gamma, I)$. 
The following counting formula is the main ingredient we need to finish the proof.

Let $A$ and $B$ be two oriented geodesic segments on $\Su$ of lengths $a>0$ and $b>0$ respectively and let $0<L_0<L_1$. Define
$$
\Conn_{A, B}(L_0, L_1) 
$$
to be the set of \emph{geodesic connections} between $A$ and $B$. That is, $\eta \in \Conn_{A, B}(L_0, L_1)$ if $\eta$ is a geodesic segment on $\Su$ of length at least $L_0$ and at most $L_1$ such that $\eta$ connects the right side of $A$ and the left side of $B$ and is orthogonal to the arcs $A$ and $B$ ($\eta$ is an orthogeodesic connecting the appropriate sides of $A$ and $B$). The following theorem is stated and proved in the Appendix as Theorem \ref{thm-count-1}.

\begin{theorem} \label{main estimate-0}  There exist constants $C=C(\Su), q=q(\Su)>0$  with the following properties. Let $\delta=e^{-qL}$, and suppose $a=b=\delta^2$. Then 
$$
\#\Conn_{A, B}(L, L+\delta^2) =\frac{1}{8\pi^2 \chi(\Su)} \delta^6 e^{L}\big(1+O(\delta) \big).
$$ 
\end{theorem}

This type of counting result appears often in literature  (for example see \cite{e-m}  \cite{p-p}) and it goes back to Margulis \cite{margulis}.

Now let $Q$ be any cube of the form $J_1\times J_2 \times [L_0,L_1] \subset \gamma\times \gamma \times \R^{+}$, with $|J_1|=|J_2|=L_1-L_0=\delta^2$ (and suppose $L_0$, $L_1$ are about $R$ which is large). Then Theorem   \ref{main estimate-0} implies

\begin{align*}
\# (Q \cap C_{\gamma})  &= E_{\Su}\delta^6e^L\big( 1+O(\delta)\big) \\
&=\left(1+O(\delta) \right) \int_Q E_{\Su} e^{L} \, dx \, dy \, dL,
\end{align*}
where 
$$
E_{\Su}=\frac{1}{8\pi^8|\chi(\Su)|}.
$$

It follows that for any region $\Rel \subset \gamma\times \gamma \times [R-1,R+1]$ tiled by such cubes
$$
\#(\Rel \cap C_{\gamma})=\big(1+O(\delta) \big) \int_{\Rel} E_{\Su} e^{L} \, dx \, dy \, dL
$$
and, more generally, for any region    $\Rel \subset \gamma\times \gamma \times [R-1,R+1]$ 
$$
\#(\Rel \cap C_{\gamma})=\big(1+O(\delta ) \big) \left( \int_{\Rel} E_{\Su} e^{L} \, dx \, dy \, dL \pm \int_{\Ne_{3\delta^{2}}(\partial {\Rel}  } E_{\Su}  e^{L} \, dx \, dy \, dL \right),
$$
where $\Ne_{3\delta^{2}}(\partial{\Rel})$ is the neighbourhood of $\partial{\Rel}$. (Here $A\pm B$ means a number in $[A-B,A+B]$).

Now let $\Rel=\Rel(\gamma,I)$ and we assume that $\epsilon \ge |I|> \delta $. Then 
$$
\Vol\big( \Ne_{3\delta^{2}} (\Rel) \big) \approx \epsilon^{2} \delta^2=O \big( \delta \text{Vol}(\Rel) \big).
$$

Therefore
$$
\#(\Rel \cap C_{\gamma})=\big( 1+ O(\delta) \big) \int_{\Rel} E_{\Su} e^{L} \, dx \, dy \, dL \approx e^{R}\epsilon^{2}|I|.
$$

On the other hand, 
$$
\int_{\Rel(\gamma,I)} E_{\Su} e^{L} \, dx \, dy \, dL=K_{\gamma}|I|,
$$
for some $K_{\gamma}>0$ because the integral clearly depends only on  $|I|$ and
$$
\int_{\Rel(\gamma,I)} \cdot= \int_{\Rel(\gamma,I_{1})} \cdot  \, +  \, \int_{\Rel(\gamma,I_{2})} \cdot
$$
where $I_1, I_2$ is a partition of $I$.

Therefore
$$
\#(\Rel(\gamma,I) \cap C_{\gamma})=K_{\gamma}|I|\big(1+O(\delta) \big)
$$
for every interval $I$ of length at least $\delta$, and the claim (2) of the theorem follows. 

\end{proof}

\subsection{The Good Correction Theorem}

To prove Theorem \ref{thm-mes}  we need to produce a measure $\mu$ on good pants such that for each good geodesic $\gamma$, $\ph \mu(\gamma)$ is $P(R)e^{-qR}$-equivalent to some $K_{\gamma}\lambda(\gamma)$, where $\lambda$ is the Lebesgue measure on $\NB (\gamma)$.  In particular, $\ph \mu(\gamma)$ must have the same total measure on both sides of $\gamma$. In other words there must be the same number of pants on both sides of $\gamma$. We can write this as $\partial{\mu}(\gamma)=0$.

Now let us construct the measure $\mu_0$ on good pants as the counting measure on the good pair of pants. Theorem \ref{thm-equidistribution} says that $\ph \mu(\gamma)$ is $Ce^{-qR}$ equivalent to $K^{\pm}\lambda_{\pm}(\gamma)$, with 
$$
\left| \frac{K^{+}_{\gamma}}{K^{-}_{\gamma}} -1\right|<Ce^{-qR}.
$$
So we have the desired equidistribution on each side of $\gamma$, and we have a small discrepancy between the number of pants on the two sides of $\gamma$. What we want to do is to make a small change in the number of each pair of good pants in order to correct the discrepancy. So we want to replace $\mu_0$ with $\mu_0+X$ where $\partial{X}=-\partial{\mu}_0$, and $X$ is small compared to $\mu_0$.

To do this we consider the more general problem of finding $X$ such that
\begin{equation}\label{eq-newest}
\partial{X}=\alpha,
\end{equation}
and ask two questions:

1. For which $\alpha$ can we solve (\ref{eq-newest})?

2. What bound can we get on the size of $X$ given a suitable bound on the size of $\alpha$? 

It turns out that we can get fairly sharp answers for both questions. First, if $\partial{X}=\alpha$, then $\alpha\equiv 0$ in $H_1(\Su)$; we will prove that we can always solve $\partial{X}=\alpha$ when $\alpha$ is zero in singular homology. Second,
if $\gamma$ is a single good closed geodesic, and $\partial{X}=\gamma$, then
$$
|X|(\{\Pi \, : \gamma \in \pt \Pi \} \ge 1,
$$
and therefore 
$$
||X||_{\infty} \ge \frac{1}{Re^{R}}.
$$

We prove that we can solve $\partial{X}=\alpha$ such that
$$
||X||_{\infty} \le P(R)e^{-R}||\alpha||_{\infty}
$$
for some polynomial $P(R)$ depending on $\Su$ and $\epsilon$. These two results are stated essentially as Theorems \ref{thm-correction}, \ref{thm-correction-0}, and the proof of these theorems will be the object of the remainder of this paper.

Having proven these theorems, we can correct the discrepancy and prove Theorem \ref{thm-mes}. 

We let $\mu_1=\mu_0+X$, where $\partial{X}=-\partial{\mu}_0$, and 
\begin{align*}
||X||_{\infty} &\le P(R)e^{-R}||\partial{\mu}_0||_{\infty}  \\
& \le P(R)e^{-R}e^{(1-q(\Su))R} \\
&<<1.
\end{align*}
Then $\mu_1$ is a positive sum of good pants, has the same number of pants of both sides of every closed geodesic, and has the same equidistribution property as $\mu_0$ because it is a small perturbation. Therefore it satisfies the conclusions of
Theorem \ref{thm-mes}, and we can use it to build a good cover.

Recall that if $A$ is an Abelian group and $X$ any set then $AX$ is the group of $A$-weighted finite formal sums of elements from $X$.

\begin{definition} 
Let $s_0,s_1 \in \R \Gamma_{\epsilon,R}$ and let $M \ge 1$. We say that $s_0=s_1$ in $\Pant_{M\epsilon,R}$ homology if there exists $W \in \R \Pant_{M\epsilon,R}$ such that $\partial W=s_1-s_0$.
\end{definition}

The following theorem summarizes the main idea of this paper. It implies
that every sum of good curves that is zero in the standard homology is the
boundary of a sum of good pants. That is, if $s_0,s_1 \in \R \Gamma_{\epsilon,R}$ and $s_0=s_1$ in the standard homology on the surface $\Su$, then $s_0=s_1$ in $\Pant_{300\epsilon,R}$ homology. By $\bf{H}_1 ( \Su)$ we denote the first homology on $\Su$  with rational coefficients.

\begin{theorem}\label{thm-correction-1} Let $\epsilon > 0$. There exists $R_0=R_0(\Su,\epsilon)> 0$ such that for every $R > R_0$ the following holds. There 
exists a set $H =\{h_1, ..., h_{2 g} \} \subset \Q \Gamma_{ \epsilon, R}$ that form a basis  of $\bf{H}_1 ( \Su)$, such that  for every $\gamma \in \Gamma_{\epsilon,R}$ there are $a_i \in \Z$ so that
$$
\gamma=\sum_{i=1}^{2g} a_i h_i
$$
in $\Pant_{300\epsilon,R}$ homology. 
\end{theorem}

\begin{remark} Observe that if $\gamma=0$ in $\bf{H}_1 ( \Su)$ then  $\gamma=0$ $\Pant_{300\epsilon,R}$ homology. 
\end{remark}

The proof of this theorem occupies the primary text of sections 4-9. But to prove the Ehrenpreis conjecture we require the following stronger result.

\begin{theorem}
  \label{thm-correction} Let $\epsilon > 0$. There exists $R_0 > 0$ (that
  depend only on $\Su$ and $\epsilon$) such that for every $R > R_0$ there
  exists a set $H =\{h_1, ..., h_{2 g} \} \subset \Q \Gamma_{\epsilon, R}$
  and a map $\phi : \Gamma_{\epsilon, R} \to \Q \Pant_{300 \epsilon, R}$ such
  that
  \begin{enumerate}
    \item $h_1,...,h_{2g}$ is a basis for  $\bf{H}_1 ( \Su)$,
    
    \item $\partial (\phi (\gamma)) - \gamma \in \Z H$,
    
    \item
    \[ \sum _{\gamma \in \Gamma_{\epsilon, R}} | \phi (\gamma) (\Pi) | < P
       (R) e^{- R} , \]
for every $\Pi \in \Pant_{300\epsilon,R}$, where the polynomial $P(R)$ depends only on $\Su$ and $\epsilon$.
  \end{enumerate}
\end{theorem}

\begin{remark}
  Note that the map $\phi$ depends on $\epsilon$ and $R$, so sometimes we
  write $\phi = \phi_{\epsilon, R}$.
\end{remark}

\begin{remark} An inequality of the form 
$$
\sum_{\gamma \in \Gamma_{\epsilon,R}}|q(\gamma)(\Pi)| \le A
$$
is equivalent to saying 
$$
||q(\alpha)||_{\infty}\le A||\alpha||_{\infty}
$$
for all $\alpha \in \Q\Gamma_{\epsilon,R}$.
\end{remark}

\begin{remark} Observe that $(1)$ and $(2)$ imply that $\pt \phi \pt=\pt$.
\end{remark} 

The existence of the function $\phi$ that satisfies the conditions (1) and (2) follows from Theorem \ref{thm-correction-1}. The condition (3) will be proved using our randomization theory (see  Appendix 1). As we go along, after every relevant homological statement  we  make Randomization remarks. Theorem \ref{thm-correction} then follows from  these randomization remarks as we explain at the end of Section 9.

\begin{remark}  The estimate (3) in the statement of the theorem can be reformulated as follows. Consider the standard measures
$\sigma_{\Gamma}$ on $\Gamma_{\epsilon,R}$ and $\sigma_{\Pant}$ on $\Pant_{300\epsilon,R}$. Then the map $\phi$ is $P(R)$-semirandom with respect to $\sigma_{\Gamma}$ and $\sigma_{\Pant}$. See the Appendix for definitions of the standard measures and the notion of semirandom maps.
\end{remark}

The image of $\phi$ lies in $\Pant_{300\epsilon,R}$, and we want it to lie in $\Pant_{\epsilon,R}$, so we require the following:

Let $M>1$. The following lemma states that any curve $\gamma \in \Gamma_{\epsilon,R}$ is homologous to some  $s \in \R \Gamma_{\frac{\epsilon}{M},R}$ in $\Pant_{\epsilon,R}$ homology.

\begin{lemma}
  \label{lemma-correction}Let $\epsilon, M > 0$. Then there exists $R_0 > 0$
  such that for every $R > R_0$ we can find a map $q_M = q :
  \Gamma_{\epsilon, R} \to \Q^+ \Pant_{\epsilon, R}$ such that
  \begin{enumerate}
    \item For every $\gamma \in \Gamma_{\epsilon, R}$, $q (\gamma)$
    is a positive sum of pants, all of which have $\gamma$ as one
    boundary cuff (with the appropriate orientation), and two other cuffs in
    $\Gamma_{\frac{\epsilon}{M}, R}$, and  $\gamma - \partial
    q (\gamma) \in \Q \Gamma_{\frac{\epsilon}{M}, R}$,
    
    \item For every $\Pi \in \Pant_{\epsilon,R}$, we have
    \[ \sum _{\gamma \in \Gamma_{\epsilon, R}} |q
       (\gamma) (\Pi) | \le \frac{K}{R} e^{- R}, \]
    for some constant $K \equiv K ( \Su, \epsilon, M) > 0$, where $q
    (\gamma) (\Pi) \in \Q^+$ is the coefficient of $\Pi$ in $q
    (\gamma)$.
  \end{enumerate}
\end{lemma}

\begin{proof} It follows from the remark after Lemma \ref{lemma-number} that the number of pants  which have $\gamma$ as one
boundary cuff, and two other cuffs in  $\Gamma_{\epsilon, \frac{R}{M}}$ is of the order $Re^{R}$. Let $q(\gamma)$ be the average of these pants (the average of a finite set $S$ is the formal sum of elements from $S$ where each element in the sum has the weight $\frac{1}{|S|}$). The inequality in the condition (2) follows from the fact that for any $\Pi \in \Pant_{1,R}$ the sum  

\[ 
\sum _{\gamma \in \Gamma_{\epsilon, R}} |q
       (\gamma) (\Pi) |, 
\]
has at most 3 non-zero terms.

\end{proof}

We can now state the following improved version of Theorem \ref{thm-correction}. This new theorem is  exactly the same as the previous one except that the new function $\phi$, which is denoted by  $\phi_{\text{new}}$, maps  $\Gamma_{\epsilon, R}$ to 
$\Q \Pant_{\epsilon, R}$ whereas the old $\phi$ maps $\Gamma_{\epsilon, R}$ to  $\Q \Pant_{300\epsilon, R}$.

\begin{theorem}
  \label{thm-correction-0} Let $\epsilon > 0$. There exists $R_0 > 0$ (that
  depend only on $\Su$ and $\epsilon$) such that for every $R > R_0$ there
  exists a set $H =\{h_1, ..., h_{2 g} \} \subset \Q \Gamma_{\epsilon, R}$
  and a map $\phi_{\text{new}} : \Gamma_{\epsilon, R} \to \Q \Pant_{\epsilon, R}$ such
  that
  \begin{enumerate}
    \item $h_1,...,h_{2g}$ is a basis for  $\bf{H}_1 ( \Su)$,
    
    \item $\partial (\phi_{\text{new}} (\gamma)) - \gamma \in \Z H$,
    
    \item
    
\[ 
\sum _{\gamma \in \Gamma_{\epsilon, R}} | \phi_{\text{new}} (\gamma) (\Pi) | < 
P(R) e^{- R} . 
\]
where the polynomial $P(R)$ depends only on $\Su$ and $\epsilon$.
\end{enumerate}
\end{theorem}

\begin{proof} We extend the function $\phi$ to $\R \Gamma_{\epsilon, R}$ by linearity. For $\gamma \in \Gamma_{\epsilon,R}$ we let 
$$
\phi_{\text{new}}(\gamma)=\phi(\gamma-\partial{q}(\gamma))+q(\gamma),
$$
where $\phi$ is from the previous theorem and $q=q_{300}$ is the improvement function from Lemma \ref{lemma-correction}. Then since $H$ is a generating set for $\bf{H}_1 ( \Su)$ it follows that
$\partial (\phi_{\text{new}} (\gamma))=\partial (\phi(\gamma))$ and thus we obtain $\partial (\phi_{\text{new}} (\gamma)) - \gamma \in \Z H$. It remains to verify the inequality $(3)$ of the theorem.

For each $\Pi \in \Pant_{\epsilon,R}$ we have 
\begin{equation}\label{nova-1}
\sum\limits_{\gamma} \big|\phi_{\text{new}}(\gamma)(\Pi) \big| \le \sum\limits_{\gamma} \big| q(\gamma)(\Pi) \big| +\sum\limits_{\gamma} \big| \phi(\gamma-\partial{q}(\gamma))(\Pi) \big|.
\end{equation}
On the other hand, for each $\eta \in \Gamma_{\epsilon,R}$ we have the inequality
$$
\sum\limits_{\gamma} \big| \partial{q}(\gamma)(\eta) \big| \le C_1,
$$
for some universal constant $C_1>0$. In other words, the total weight of $\eta$ in the formal sum of curves 
$$
\sum\limits_{\gamma} \partial{q}(\gamma) \, \in \Q \Gamma_{\epsilon,R}
$$
is at most $C_1$. This follows from the last inequality of Lemma \ref{lemma-correction} and Lemma \ref{lemma-number}.

Thus, we have
$$
\sum\limits_{\gamma} \big| \phi(\gamma-\partial{q}(\gamma))(\Pi) \big|\le \sum\limits_{\gamma} \big| \phi(\gamma)(\Pi) \big|+
C_1\sum\limits_{\gamma} \big| \phi(\gamma)(\Pi) \big|.
$$
We replace this inequality in (\ref{nova-1}) and the theorem follows.

\end{proof}

Of course $\phi_{\text{new}}$ extends linearly to $\Q \Gamma_{\epsilon, R}$.

We observe that if $\gamma$ is zero in $\bf{H}_1 ( \Su)$ then $\partial{\phi_{\text{new}}(\gamma)}$ is equal to $\gamma$ because  $\partial (\phi_{\text{new}} (\gamma)) - \gamma \in \Z H$ and $\partial (\phi_{\text{new}} (\gamma))$ differs from $\gamma$ by a boundary. In particular, for any $\mu \in \Q \Pant_{\epsilon, R}$ the equality 
$$
\partial{\phi_{\text{new}}} \partial{\mu}=\partial{\mu}
$$
holds.

\subsection{A proof of Theorem \ref{thm-mes} } First we state and prove the
following lemma about equivalent measures on the circle $\R / 2 R \Z$, where
$R > 0$ is a parameter. Recall that $\lambda$ denotes the Lebesgue measure on
the circle $\R / 2 R \Z$.

\begin{lemma}
  \label{lemma-transport} If $\alpha$ is $\delta$-equivalent to $K \lambda$ on
  $\R / 2 R \Z$, for some $K > 0$, and $\beta$ is a measure on $\R /2R \Z$, then $\alpha + \beta$ is $( \frac{| \beta
  |}{2 K} + \delta)$-equivalent to $(K + \frac{| \beta |}{2 R}) \lambda$ on
  $\R / 2 R \Z$.
\end{lemma}

\begin{proof}
  Recall the definition of $\delta$-equivalent measures from the previous
  section. We need to prove that $(\alpha + \beta) (I) \le (K + \frac{| \beta
  |}{2 R}) |\Ne_{\delta'} (I)|$, where $\delta' = \delta + \frac{| \beta |}{2
  K}$, for any interval $I$ such that $\delta'$ neighbourhood of $I$ is a
  proper subset of the circle $\R / 2 R \Z$.
  
  We have 
  
  \begin{align*}
    ({\alpha}+{\beta})(I) & {\le}K(|I|+2{\delta})+|{\beta}|\\
    &
    {\le}\left(K+{\frac{|{\beta}|}{2R}}\right)\left(|I|+2{\delta}+{\frac{|{\beta}|}{K}}\right)\\
    &=\left(K+ \frac{|{\beta}|}{2R} \right)\big|{\Ne}_{{\delta}'}(I)\big|.
  \end{align*}
\end{proof}

We give the following definitions. To any measure $\alpha \in \Mes ( \NB (
\sqrt{\gamma}))$ we associate the number  $|\alpha|(\gamma)=|\alpha_{+}(\gamma)|+|\alpha_{-}(\gamma)|$. 

We proceed with the proof of Theorem \ref{thm-mes}. Let $\mu$ be the measure on $\Pant_{\epsilon, R}$ from Theorem \ref{thm-equidistribution}. 
Define the measure $\mu_1$ on $\Pant_{\epsilon, R}$  by letting $\mu_1=\mu-\phi(\partial{\mu})$, where $\phi=\phi_{\text{new}}$ is from Theorem \ref{thm-correction-0}. We show that $\mu_1$ is the measure that satisfies the conclusions of Theorem \ref{thm-mes}.

As observed before $\partial{\mu}_1=0$. It remains to show that $\ph{\mu}_1(\gamma)$ is $P(R)e^{-qR}$ equivalent to some multiple of the Lebesgue measure on  $\NB(\sqrt{\gamma})$.
Recall from Theorem \ref{thm-equidistribution} that  the measure $\ph_{\pm}\mu (\gamma)$ is $Ce^{-q R}$ equivalent to 
$K^{\pm}_{\gamma} \lambda_{\pm} (\gamma)$, for some constants $K^+_{\gamma}$ and $K^-_{\gamma}$ that satisfy the inequality
\begin{equation}\label{eq-novo-1}
\bigg{|} \frac{K^+_{\gamma}}{K^-_{\gamma}} - 1 \bigg{|} < C e^{-q R}, 
\end{equation}
and that $K^{\pm}_{\gamma} \asymp Re^{R}$. Then for all $\gamma$, from (\ref{eq-novo-1}) we get 
$$
\big| \partial \mu(\gamma) \big| \le CR e^{(1-q)R},
$$
which together with Theorem \ref{thm-correction-0} yields the inequality
$$
|\phi (\partial{\mu}) |(\Pi)\le P(R)e^{-qR},
$$
for each pair of pants $\Pi$, and we obtain
$$
|\ph{\phi (\partial{\mu})}|(\gamma)\le P(R)e^{(1-q)R},
$$
for all $\gamma$. Since $K^{\pm}_{\gamma} \asymp Re^{R}$ we conclude from Lemma \ref{lemma-transport} that $\ph{\mu}_1(\gamma)$ is $P(R)e^{-qR}$ equivalent to some multiple of the Lebesgue measure $\lambda(\gamma)$ on  $\NB(\sqrt{\gamma})$, and we are done.

\section{The theory of inefficiency} In this section we develop the theory of inefficiency.  This theory supports the geometric side of the correction theory that is used to prove our main technical result the Good Correction Theorem.

Before we begin with the estimates of this section, we will provide a brief overview of the remainder of this paper. Our goal is to prove identities in the good pants homology, which means that we need to generate a formal sum of good pants whose boundary is a certain given formal sum of good curves. How do we generate a pair of good pants?  

We generate a pair of good pants by constructing a $\Theta$-graph made out of geodesic segments; the Connection Lemma insures that we have enough geodesic segments with the desired properties, and the Theory of Inefficiency allows us to estimate the length of the cuffs of the associated pair of pants in terms of the length of the geodesic segments and the angle at which they meet.

In every identity we prove in the good pants homology we will state out hypothesis in terms of the Theory of Inefficiency, and every time we estimate the length of a geodesic segment or a closed geodesic we will use this theory as well.

\subsection{The inefficiency of a piecewise geodesic arc} By $\TB \Ha$ we denote the unit tangent bundle of $\Ha$. Elements of $\TB \Ha$ are pairs $(p, u)$, where $p \in \Ha$ and $u \in \TB_p \Ha$. 
For $u, v \in \TB_p \Ha$ we let $\Ang (u, v)$ denote the unoriented angle between $u$ and $v$.  Let $\alpha : [a, b] \to \Ha$ be a unit speed geodesic segment. We let $i(\alpha) = \alpha' (a)$, and $t (\alpha) = \alpha' (b)$.

Let  $\alpha_1,...,\alpha_n$ denote oriented piecewise geodesic arcs on $\Su$  such that the  terminal point of $\alpha_i$ is the initial point of $\alpha_{i+1}$.
By $\alpha_1\alpha_2...\alpha_n$ we denote the concatenation of the arcs $\alpha_i$. If the initial point of $\alpha_1$ and the terminal point of $\alpha_n$ are the same, by 
$[\alpha_1\alpha_2 ... \alpha_n]$ we denote the corresponding closed curve.

We define the inefficiency operator as follows. We first discuss the inefficiency of piecewise geodesic arcs and after that the inefficiency of piecewise geodesic closed curves.

\begin{definition}
Let $\alpha$ be an arc on a surface. By $\gamma$ we denote the  geodesic arc with the same endpoints and homotopic to $\alpha$. We let $I(\alpha)=\len(\alpha)-\len(\gamma)$. We call $I(\alpha)$ the inefficiency of $\alpha$ (the inefficiency $I(\alpha)$ is equal to 0 if and only if $\alpha$ is a geodesic arc). 
\end{definition}

We  observe  the monotonicity of inefficiency. Let $\alpha,\beta,\gamma \subset \Ha$ be three piecewise geodesic arcs in $\Ha$ such that  $\alpha \beta \gamma$ is a well defined piecewise geodesic arc. Then $I(\alpha \beta \gamma) \ge I(\beta)$. This is seen as follows. Let $\eta$ be the geodesic arc with the same endpoints as  $\alpha \beta \gamma$, and let $\beta'$ be the geodesic arc with the same endpoints as $\beta$. Then 
\begin{align*}
I(\alpha\beta\gamma) &= \len(\alpha \beta \gamma)-\len(\eta) \\
&\ge \len(\alpha \beta \gamma) -\len(\alpha \beta' \gamma) \\
&= \len(\beta)-\len(\beta') \\
&=I(\beta).
\end{align*}

We also define the inefficiency function of an angle $\theta \in [0,\pi]$ as follows. Let $\alpha_{\infty}$ and $\beta_{\infty}$ be two geodesic rays in $\Ha$ that have the same initial point, such that $\theta$ is the exterior angle between $\alpha_{\infty}$ and $\beta_{\infty}$ (then $\theta$ is also the bending angle of the  piecewise geodesic $\alpha^{-1}_{\infty} \beta_{\infty}$). 
We want to define $I(\theta)$ as the inefficiency of $\alpha^{-1}_{\infty} \beta_{\infty}$, but since the piecewise geodesic  $\alpha^{-1}_{\infty} \beta_{\infty}$ is infinite in length we need to prove that such a definition is valid.

Consider a geodesic triangle in $\Ha$ with sides $A,B$ and $C$, and let $\theta>0$ be the exterior angle opposite to $C$ (we also let $\len(A)=A$, $\len(B)=B$, $\len(C)=C$). 
Then 
$$
\cosh C = \cosh A \cosh B+\cos \theta \sinh A \sinh B, 
$$
and therefore
$$
\frac{\cosh C }{e^{A+B} }=\frac{\cosh A }{e^{A} } \frac{\cosh B }{e^{B} }+\cos \theta  \frac{\sinh A }{e^{A} } \frac{\sinh B }{e^{B} }.
$$
We conclude that
$$
\frac{\cosh C }{e^{A+B} } \to \frac{1}{4}(\cos \theta +1),
$$
when $A,B \to \infty$. Since 
$$
\frac{\cosh C }{e^{C} } \to \frac{1}{2}, \,\, C \to \infty,
$$
we get 
$$
e^{C-A-B} \to \left( \cos \frac{\theta}{2} \right) ^{2},
$$
and therefore 
$$
A+B-C \to 2 \log \sec \frac{\theta}{2}.
$$

Let $r,s>0$ and let $\alpha_r$ be the geodesic subsegment of $\alpha_{\infty}$ (with the same initial point) of length $r$. Similarly, $\beta_s$ is a geodesic subsegment of $\beta_{\infty}$ of length $s$.
Then we let
$$
I(\theta)= I(\alpha^{-1}_{\infty} \beta_{\infty}) = \lim_{r,s \to \infty} I(\alpha^{-1}_r \beta_s).
$$
It follows from the above discussion that this limit exists and
\begin{equation}\label{angle-ineff}
I(\theta)=2 \log \sec \frac{\theta}{2}.
\end{equation}

\subsection{Preliminary lemmas} We have the following lemma. 

\begin{lemma}\label{lemma-proj-arc} 
Let $\alpha$ denote an arc  on $\Su$, and let $\gamma$ be the appropriate geodesic arc with the same endpoints as $\alpha$ and homotopic to $\alpha$. 
Choose lifts of $\alpha$ and $\gamma$ in the universal cover $\Ha$ that have the same endpoints and let $\pi:\alpha \to \gamma$ be the nearest point projection. Let 
$$
E(\alpha)=\sup\limits_{x \in \alpha} d(x,\pi(x)). 
$$
Then 
$$
E(\alpha) \le \frac{I(\alpha)}{2}+ \log 2. 
$$

\end{lemma}

\begin{proof}  Let $E>0$. The minimally inefficient arc  $\alpha$ (that has the same endpoints as $\gamma$ and is  homotopic to $\gamma$),  that is at the distance $E$ from $\gamma$ is given in Figure \ref{fig-open}. Here $\gamma$ is divided into two sub-segments of length $L^{-}$ and $L^{+}$. Let $A^{-}=\len(\alpha^{-})$ and $A^{+}=\len(\alpha^{+})$. By the monotonicity of inefficiency, and using the inefficiency for angles, we have
$$
E+L^{-}-A^{-}<I(\frac{\pi}{2}),
$$
and 
$$
E+L^{+}-A^{+}<I(\frac{\pi}{2}).
$$
Summing up yields
$$
E<\frac{I(\alpha)}{2}+I(\frac{\pi}{2})=\frac{I(\alpha)}{2}+\log 2,
$$
since by (\ref{angle-ineff}) we have $I(\frac{\pi}{2})=\log 2$.

\begin{figure}
  %\centering
  {
    
    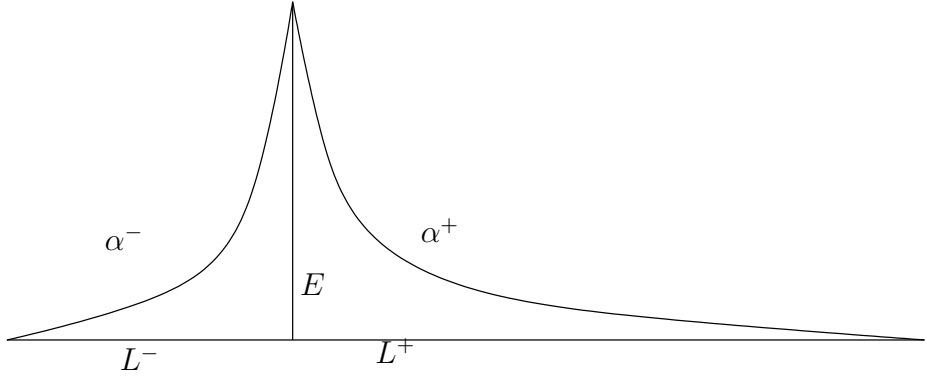
  }
  \caption{The case where $\gamma$ is an arc}
  \label{fig-open}
\end{figure}

\end{proof}

The following is the New Angle Lemma.

\begin{lemma}[New Angle Lemma]\label{lemma-new-angle} Let $\delta,\Delta>0$ and  let $\alpha \beta \subset \Ha$ be a piecewise geodesic arc, where $\alpha$ is piecewise geodesic arc  and $\beta$ is a geodesic arc. Suppose that $\gamma$ is the geodesic arc with the same endpoints as $\alpha \beta$. There exists $L=L(\delta,\Delta)>0$ such that if $\len(\beta)>L$ and $I(\alpha \beta) \le \Delta$ then the unoriented angle between $\gamma$ and $\beta$ is at most $\delta$.
\end{lemma}

\begin{proof} Denote by $\theta$ the angle between $\gamma$ and $\beta$, and let $h$ be the distance between the other endpoint of $\beta$ and $\gamma$. Then 
$$
\frac{\sinh(h)}{\sin(\theta)}=\sinh(\len(\beta)).
$$
The lemma follows from this equation and the fact that $h$ is bounded in terms of $I(\alpha \beta)$ (see Lemma \ref{lemma-proj-arc}).

\end{proof}

We also have 

\begin{lemma}\label{lemma-pomoc} Suppose that $\alpha \beta \gamma$ is a concatenation of three geodesic arcs in $\Ha$, and let $\theta_{\alpha\beta}$ and $\theta_{\beta\gamma}$ be the two bending angles. 
Suppose that  $\theta_{\alpha\beta}, \theta_{\beta\gamma}<\frac{\pi}{2}$. Then 
$$
I(\alpha\beta\gamma) \le \log(\sec(\theta_{\alpha\beta})) +\log( \sec(\theta_{\beta\gamma})).
$$
\end{lemma}

\begin{proof} Let $\eta_1$ be the geodesic that  is orthogonal to $\beta$ at the point where $\alpha$ and $\beta$ meet. Similarly let $\eta_2$ be the geodesic that is orthogonal to $\beta$ at the point where $\beta$ and $\gamma$ meet. Let $A_{\alpha}$ be the geodesic arc orthogonal to $\eta_1$ that starts at the initial point of $\alpha$, and  let $A_{\gamma}$ be the geodesic arc orthogonal to $\eta_2$ that starts at the terminal point of $\gamma$. 

Let $\eta$ be the geodesic arc with the same endpoints as $\alpha\beta \gamma$. Then $\len(\eta) \ge \len(A_{\alpha})+\len(\beta)+\len(A_{\gamma})$.

On the other hand, from the hyperbolic low of sines, we have

$$
\sinh A_{\alpha}= \sinh \alpha \cdot \cos(\theta_{\alpha\beta},
$$
and hence 
$$
\log \sinh \alpha - \log \sinh A_{\alpha} \le \log \sec \theta_{\alpha\beta},
$$
which implies
$$
\alpha-A_{\alpha} \le \log \sec \theta_{\alpha\beta},
$$
because the derivative of $\log \sinh$ is greater than one. Thus, we have proved that 
$\len(A_{\alpha})>\len(\alpha)-\log(\sec(\theta_{\alpha\beta}))$, and similarly $\len(A_{\gamma})>\len(\gamma)-\log(\sec(\theta_{\beta\gamma}))$. So

\begin{align*}
I(\alpha\beta\gamma) &<\len(\alpha)+\len(\gamma)-\len(A_{\alpha})-\len(A_{\gamma}) \\
&< \log(\sec(\theta_{\alpha\beta})) +\log( \sec(\theta_{\beta\gamma})).
\end{align*}

\end{proof}

\subsection{The Long Segment Lemmas for arcs} We now state and prove several lemmas called the Long Segment Lemmas. 
The following is the Long Segment Lemma for angles.

\begin{lemma}  [Long Segment Lemma] \label{lemma-long-angle} Let $\delta>0,\Delta>0$. There exists a constant $L=L(\delta,\Delta)>0$ such that if $\alpha$ and $\beta$ are oriented  geodesic arcs such that $I(\alpha\beta) \le \Delta$ (assuming that the terminal point of $\alpha$ is the initial point of $\beta$) and $\len(\alpha),\len(\beta)>L$, then
$I(\alpha\beta)<I(\Ang(t(\alpha),i(\beta)))<I(\alpha\beta)+\delta$.
\end{lemma}

\begin{proof} The left hand side of the above inequality follows from the monotonicity of inefficiency.  
For the right hand side, let $\alpha_{\infty}$ and $\beta_{\infty}$ denote the geodesic rays whose initial point  is the point where $\alpha$ and $\beta$ meet, and that contain $\alpha$ and $\beta$ respectively (we also assume that $\alpha_{\infty}$ has the same orientation as $\alpha$ and $\beta_{\infty}$ the same orientation as $\beta$).  Recall that $I(\alpha_\infty \beta_{\infty})$ was defined just above the formula (\ref{angle-ineff}).
Let $\eta$ be the geodesic arc  with the same endpoints as $\alpha\beta$, and $\eta_1$ the geodesic ray with the same endpoints as $\alpha_{\infty} \beta$. By $\theta_0$ we denote the angle between $\alpha$ and $\eta$, by $\theta_1$ the angle between $\eta$ and $\beta$, and by $\theta_2$ the angle between $\eta$ and $\eta_1$.

We observe that $\theta_0$, $\theta_1$ and $\theta_2$ are small (by the New Angle Lemma), and therefore 
\begin{align*}
I( \alpha_{\infty} \beta_{\infty} ) &< I(\alpha_{\infty} \beta)+I(\theta_1+\theta_2) \\
&<I(\alpha \beta) +I(\theta_0)+I(\theta_1+\theta_2) \\
&<I(\alpha \beta)+o(1),
\end{align*}
where $o(1) \to 0$ as $L \to \infty$.

\end{proof}

The following is the Long Segment Lemma for arcs.

\begin{lemma} [Long Segment Lemma for arcs]  \label{lemma-long-arc} Suppose we can write $\eta=\alpha  \beta  \gamma$, where $\alpha$ and $\gamma$ are piecewise geodesic arcs, and $\beta$ is a geodesic arc of length $l$.
Then
$$
\big| I(\alpha  \beta) +I(\beta  \gamma) -I(\alpha  \beta  \gamma) \big| \to 0,
$$
uniformly when $l \to \infty$ and  $I(\alpha  \beta) +I(\beta  \gamma)$ is bounded above.

\end{lemma}

\begin{proof} If we replace $\alpha$ and $\gamma$ by the associated geodesics arcs, then $I(\alpha  \beta) +I(\beta  \gamma) -I(\alpha  \beta  \gamma)$ will be unchanged,
and  $I(\alpha  \beta) +I(\beta  \gamma)$ will be decreased, so we can assume that $\alpha$ and $\gamma$ are geodesic arcs. We divide $\beta$ at its midpoint into
$\beta^{-}$ and $\beta^{+}$, so $\beta=\beta^{-}  \beta^{+}$, and $\alpha  \beta  \gamma= \alpha  \beta^{-}  \beta^{+}  \gamma$. We will show the following estimates 
(for $\delta$ small when $l$ is large and $I(\alpha  \beta) +I(\beta  \gamma)$ is bounded above):

\begin{enumerate} 
\item $\big| I(\alpha  \beta^{-}) + I(\beta^{+}  \gamma) -I(\alpha  \beta  \gamma) \big|<\delta$,
\item $\big| I(\alpha  \beta) +I(\alpha  \beta^{-})) \big|<\delta$,
\item $\big| I(\beta  \gamma) +I(\beta^{+}  \gamma)) \big|<\delta$,
\end{enumerate}
The lemma then follows from (1), (2) and (3).

\begin{figure}
  %\centering
  {
    
    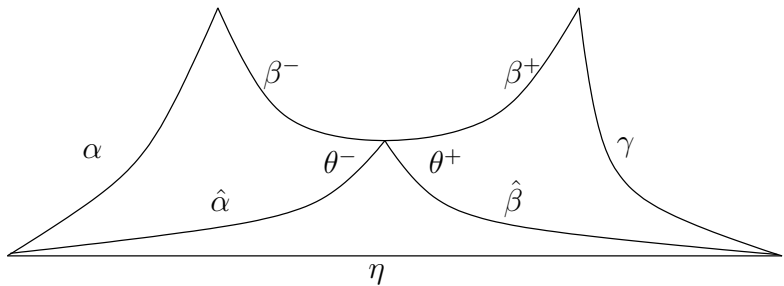
  }
  \caption{The Long Segment Lemma}
  \label{fig-long-segment}
\end{figure}

For (1), we refer to the Figure \ref{fig-long-segment}. We find that 
$$
0 \le I(\wh{\alpha}  \wh{\beta} )= I(\alpha  \beta  \gamma)-I(\alpha  \beta^{-}) - I(\beta^{+}  \gamma).
$$ 
Moreover, when $I(\alpha  \beta^{-})$ (which is at most $I(\alpha \beta)$) and $I(\beta^{+}  \gamma)$ (which is at most $I(\beta  \gamma)$) are bounded above, and $\len(\beta^{-})=\len(\beta^{+})=\frac{\len(\beta)}{2}$ is a large, then $\theta^{-}$ and $\theta^{+}$ are small (by the New Angle Lemma), so $I(\wh{\alpha}  \wh{\beta}) \le I(\theta^{-}+ \theta^{+})$ is small.
Likewise, $I(\alpha  \beta)-I(\alpha  \beta^{-})=I(\wh{\alpha}  \beta^{+})$, and $0 \le I(\wh{\alpha}  \beta^{+}) \le I(\theta^{-})$. This proves (1) and (2), and (3) is the same as (2).

\end{proof}

\subsection{The inefficiency of a closed piecewise geodesic curve} Let  $\alpha_1,...,\alpha_n$ denote oriented piecewise geodesic arcs on $\Su$  such that the  terminal point of $\alpha_i$ is the initial point of $\alpha_{i+1}$. By $\alpha_1\alpha_2...\alpha_n$ we denote the concatenation of the arcs $\alpha_i$. Assume that the initial point of $\alpha_1$ and the terminal point of $\alpha_n$ are the same. By  $[\alpha_1\alpha_2 ... \alpha_n]$ we denote the corresponding closed curve.

We define the inefficiency operator as follows. 

\begin{definition}
Let $\alpha$ be  a closed curve on a surface. By $\gamma$ we denote the appropriate  closed geodesic that is freely homotopic to $\alpha$. We let $I(\alpha)=\len(\alpha)-\len(\gamma)$. 
We call $I(\alpha)$ the inefficiency of $\alpha$ (the inefficiency $I(\alpha)$ is equal to 0 if and only if $\alpha$ is  a closed geodesic). 
\end{definition}

The following is a closed curve  version of Lemma \ref{lemma-proj-arc}.

\begin{lemma}\label{lemma-proj-closed}  Let $\alpha$ denote a closed curve on $\Su$, and let $\gamma$ be the appropriate  closed geodesic freely  homotopic to $\alpha$. Choose lifts $\wt{\alpha}$ and $\wt{\gamma}$, of $\alpha$ and $\gamma$ respectively, in the universal cover $\Ha$ that have the same endpoints. The nearest point projection $\wt{\pi}:\wt{\alpha} \to \wt{\gamma}$  descends to the map $\pi:\alpha \to \gamma$. Let 
$$
E=E(\alpha)=\sup\limits_{x \in \wt{\alpha}} d(x,\wt{\pi}(x)). 
$$
Then providing $\len(\alpha)>L_0$ for some universal constant $L_0$ we have

$$
E \le \frac{I(\alpha)}{2}+1.
$$

\end{lemma}

\begin{proof} The minimally inefficient closed curve $\alpha$ that is freely homotopic to $\gamma$ and the distance $E$ from $\gamma$ is given in Figure \ref{fig-closed}. 
Denote by $\eta$ the corresponding geodesic segment of length $E$. Then by the Long Segment Lemma and monotonicity of inefficiency
\begin{align*}
I(\eta\gamma \eta^{-1}) &\le  I(\eta \gamma)+I(\gamma \eta^{-1})+\frac{1}{7} \\
&\le I \left( \frac{\pi}{2} \right)+ I\left( \frac{\pi}{2} \right)+\frac{1}{7} \\
&< 2, 
\end{align*}
providing that $\len(\gamma)>L_0$, where $L_0$ is a universal constant. Hence
$$
\len(\alpha)-\len(\gamma)\ge 2 \len(\eta)-2,
$$
or 
$$
E \le \frac{I(\alpha)}{2}+1.
$$

\begin{figure}
  %\centering
  {
    
    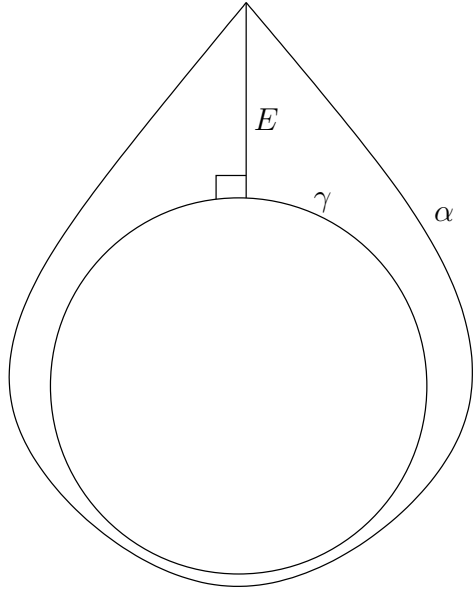
  }
  \caption{$\alpha$ is the minimally inefficient curve with $\len(\eta)=E$}
  \label{fig-closed}
\end{figure}

\end{proof}

The following is the Long Segment Lemma for closed curves.

\begin{lemma}[Long Segment Lemma for closed curves]\label{lemma-long-closed} Let $\alpha$ be  an   piecewise geodesic arc  and $\beta$ an  geodesic arc on $\Su$, such that the initial point of $\alpha$ is the terminal point of $\beta$ and 
the initial point of $\beta$ is the terminal point of $\alpha$. Then
$$
\left| I([\alpha  \beta ])-I(\beta  \alpha  \beta) \right|<\delta, 
$$
where $\delta \to 0$ when $\len(\beta) \to \infty$ and $I(\beta  \alpha  \beta)$ is bounded above.
\end{lemma}

\begin{proof} The proof is similar to the proof of Lemma \ref{lemma-long-arc} and is left to reader.

\end{proof}

\subsection{The  Sum of Inefficiencies Lemma} The following is the Sum of Inefficiencies Lemma. Let $\Su$ denote a closed hyperbolic Riemann surface

\begin{lemma}[Sum of Inefficiencies Lemma]\label{lemma-ineff+} Let $\epsilon,\Delta>0$ and $n \in \N$. There exists $L=L(\epsilon,\Delta,n)>0$ such that the following holds. Let $\alpha_1,...,\alpha_{n+1}=\alpha_1, \beta_1,...\beta_n$, be geodesic arcs on the surface $\Su$ such that $\alpha_1 \beta_1  \alpha_2  \beta_2 ... \alpha_n  \beta_n $ is a piecewise geodesic closed curve on $\Su$. If
$I( \alpha_i  \beta_i  \alpha_{i+1}) \le \Delta$, and $\len( \alpha_i ) \ge L$, then

$$
\left| I([ \alpha_1  \beta_1  \alpha_2  \beta_2  ...  \alpha_n  \beta_n  ]) - \sum\limits_{i=1}^{n} I(\alpha_i \beta_i \alpha_{i+1} ) \right| \le \epsilon.
$$
\end{lemma}

\begin{proof} It directly follows from the Long Segment Lemma for closed curves.

\end{proof}

\begin{remark} In particular,  we can leave out the $\beta$'s in the above lemma, and write 
$$
\left| I([ \alpha_1  \alpha_2   ...  \alpha_n  ]) - \sum\limits_{i=1}^{n} I(\alpha_i \alpha_{i+1}) \right| \le \epsilon,
$$
providing that $I(\alpha_i \alpha_{i+1}) \le \Delta$, and $\len(\alpha_i ) \ge L$. Moreover,  by the Long Segment Lemma for Angles (for $L$ large enough) we have
$$
\left| I([ \alpha_1  \alpha_2   ...  \alpha_n  ]) - \sum\limits_{i=1}^{n} I(\theta_i) \right| \le 2\epsilon,
$$
where $\theta_i=\Ang(t(\alpha_i ),i ( \alpha_{i+1} ))$.

\end{remark}

A  more general version of the Sum of Inefficiencies Lemma is as follows (the proof is the same).

\begin{lemma}\label{lemma-ineff+-1} Let $\epsilon,\Delta>0$ and $n \in \N$. There exists $L=L(\epsilon,\Delta,n)>0$ such that the following holds. Let $\alpha_1,...,\alpha_{n+1}=\alpha_1$ and 
$\beta_{11},...\beta_{1j_{1}},... \beta_{n1},...\beta_{nj_{n}}$, be geodesic segments on $\Su$ such that $\alpha_1  \beta_{11} ... \beta_{1j_{1}}  ...   \alpha_{n}  \beta_{n1} ...  \beta_{nj_{n}} $ is a
piecewise geodesic closed curve on $\Su$. If  $I(\alpha_i  \beta_i  \alpha_{i+1}) \le \Delta$, and $\len( \alpha_i ) \ge L$, then

$$
\left| I([ \alpha_1  \beta_{11} ... \beta_{1j_{1}}  ...   \alpha_{n}  \beta_{n1} ...  \beta_{nj_{n}} ]) - \sum\limits_{i=1}^{n} I(\alpha_i \beta_{i1}...\beta_{ij_{i}} \alpha_{i+1}) \right| \le \epsilon.
$$
\end{lemma}

\begin{proof} It directly follows from the Long Segment Lemma.

\end{proof}

We will use the theory of inefficiency for two purposes. First, to control the geometry of piecewise geodesic arcs and closed curves, and second, to precisely estimate the length of the associated geodesic arcs and closed curves.

For example, suppose $[\alpha,\beta]$ is a closed curve, and $\alpha$ and $\beta$ meet nearly at right angles at the two places they meet, and $\alpha$ and $\beta$ are both long. Then $\len(\gamma)$ is close to $\len(\alpha)+\len(\beta)-\log 4$ where $\gamma$ is the corresponding closed geodesic.

As a second example, suppose $\alpha_1,\alpha_2,\alpha_3$ is a piecewise geodesic arc, and let $\alpha_{12}$ be the geodesic arc homotopic rel endpoints to $\alpha_1 \alpha_2$, and likewise define $\alpha_{23}$ and $\alpha_{123}$.
Then

$$
\len(\alpha_{123})=\len(\alpha_1)+\len(\alpha_2)+\len(\alpha_3)+I(\alpha_{12})+I(\alpha_{23}),
$$ 
provided that $I(\alpha_{12}),I(\alpha_{23}) \le \Delta$ and $\len(\alpha_2) \ge L(\Delta)$.

\section{The Geometric Square Lemma } In this section we will prove the Geometric Square Lemma (Lemma \ref{lemma-square}), which will then be reformulated as the Algebraic Square Lemma in the next section. This section is probably the hardest and most technical section in the paper, and the reader who is still struggling to have a clear idea of where we are going may wish to skip to the next section and see how the Algebraic Square Lemma follows from the Geometric one.

From now on we can think of the surface $\Su$ as being fixed. We also fix $\epsilon>0$. However, for the reader's 
convenience we always emphasize how quantities may depend on $\Su$ and $\epsilon$.

\subsection{Notation and preliminary lemmas} By an oriented closed geodesic
$C$ on $\Su$ we will mean an isometric immersion $C : \TT_C \to \Su$, where
$\TT_C = \R / \len (C)$, and $\len (C)$ is the length of $C$. To simplify the
notation, by $C : \R \to \Su$ we also denote the corresponding lift (such a lift
is uniquely determined once we fix a covering map $\pi : \R \to \TT_C$). We
call $\TT_C$ the parameterizing torus for $C$ (because $\TT_C = \R / \len (C)$ is a $1$-torus).
By a point on $C$ we mean $C (p)$ where $p \in \TT_C$, or $p \in \R$. Given
two points $a, b \in \R$, we let $C [a, b]$ be the restriction of $C : \R \to
\Su$ to the interval $[a, b]$. If $b < a$ then the orientation of the segment
$C [a, b]$ is the negative of the orientation for $C$. Of course $C [a + n
\len (C), b + n \len (C)]$ is the same creature for $n \in \Z$. By $C' (p)$ we
denote the unit tangent vector to $C$ with the appropriate orientation.

Recall that $\TB \Ha$ denotes the unit tangent bundle, where elements of $\TB
( \Ha)$ are pairs $(p, u)$, where $p \in \Ha$ and $u \in \TB_p \Ha$. The
tangent space $\TB_p$ has a complex structure and given $u \in \TB \Su$, by
$\sqrt{-1} u \in \TB_p \Su$ we denote the vector obtained from $u$ by rotating for
$\frac{\pi}{2}$.

Recall that for $u, v \in \TB_p \Ha$ we let $\Ang (u, v)$ denote the
unoriented angle between $u$ and $v$. If $u \in \TB_p \Ha$ then $u \at q \in
\TB_q \Ha$ denotes the vector $u$  parallel transported to $q$ along the geodesic
segment connecting $p$ and $q$. We use the similar notation for points in $\TB
\Su$, except that in this case one always has to specify the segment between
$p$ and $q$ along which we parallel transport vectors from $\TB_p \Su$ to
$\TB_q \Su$.

We refer to the following lemma as the Convergence Lemma. The proof is left to the reader.

\begin{lemma}
  \label{lemma-convergence}Suppose $A$ and $B$ are oriented geodesics in $\Ha$
  that are $E$-nearly homotopic, and let
  \[ a : \left[ - \frac{\len (A)}{2}, \frac{\len (A)}{2} \right] \to \Ha,
     \hspace{0.25em} \hspace{0.25em} \hspace{0.25em} b : \left[ - \frac{\len
     (B)}{2}, \frac{\len (B)}{2} \right] \to \Ha, \]
  denote the unit time parametrization. Set $l = \frac{1}{2} \min ( \len (A),
  \len (B))$. Then there exists $0 \le t_0 \le E$, such that for $t \in [- l,
  l]$ the following inequalities hold
  \begin{enumerate}
    \item $d (a (t), b (t + t_0)) \le e^{|t| + E + 1 - l}$,
    
    \item $\Ang (a' (t) @b (t + t_0), b' (t + t_0)) \le e^{|t| + E + 1 - l}$.
  \end{enumerate}
\end{lemma}

Let $(p, u)$ and $(q, v)$ be two vectors from $\TB ( \Ha)$. We define the
distance function
\[ \dist ((p, u), (q, v)) = \max ( \Ang (u \at q, v), d (p, q)) . \]
(We do not insist that $\dist$ is a metric on $\TB \Ha$).

Let $\alpha : [a, b] \to \Ha$ be a unit speed geodesic segment. We let $i
(\alpha) = \alpha' (a)$, and $t (\alpha) = \alpha' (b)$. We have the following
lemma (we omit the proof).

\begin{lemma}
  \label{lemma-C-1}Let $\epsilon, L > 0$. There exists a constant $\epsilon'
  (L)$ with the following properties. Suppose that $\alpha : [a_0, a_1] \to
  \Ha$ and $\beta : [b_0, b_1] \to \Ha$ are $\epsilon$-nearly homotopic, that
  is $d (\alpha (a_i), \beta (b_i)) \le \epsilon$. Suppose that $a_1 - a_0 >
  L$, and $\epsilon < 1$. Then
  \[ \dist (\alpha' (a_i), \beta' (b_i)) \le \epsilon (1 + \epsilon' (L)), \]
  with $\epsilon' (L) \to 0$ as $L \to \infty$.
\end{lemma}

\subsection{The Preliminary Geometric Square Lemma (the PGSL)}

Suppose $C_{ij}$, $i,j=0,1$, are four closed geodesics on $\Su$, and imagine that $C_{ij}$ is covered by two overlapping arcs $C_{ij}^{+}$ and $C_{ij}^{-}$, where  $C_{i0}^{+}$ and $C_{i1}^{+}$ are nearly homotopic
and likewise for $C_{0j}^{-}$ and $C_{1j}^{-}$. The Geometric Square Lemma (GSL) states that

$$
\sum (-1)^{ij}C_{ij}=0.
$$
The full statement of the GSL is given in Section 5.3.

The following is the Preliminary Geometric Square Lemma. We have added the hypothesis $(5)$ to the GSL (Lemma \ref{lemma-square}), so as to find points
in the two convergence intervals  of the four curves, that are nearly diametrically opposite.

\begin{lemma}[Preliminary Geometric Square Lemma]
  \label{lemma-square-prel}Let $E, \epsilon > 0$. There exist constants $K = K
  (\epsilon, E) > 0$ and $R_0 ( \Su, \epsilon, E) > 0$ with the following
  properties. Suppose that we are given four oriented geodesics $C_{ij} \in
  \Gamma_{\epsilon, R}$, $i, j = 0, 1$, and for each $ij$ we are given 4
  real numbers $x^-_{ij} < x^+_{ij} < y^-_{ij} < y^+_{ij} < x^-_{ij} + \len
  (C_{ij})$. Assume that
  \begin{enumerate}
    \item The inequalities $x^+_{ij} - x^-_{ij} > K$, and $y^+_{ij} - y^-_{ij}
    > K$, hold.
    
    \item The segments $C_{ij} [x^-_{ij}, x^+_{ij}]$ and $C_{i' j'} [x^-_{i'
    j'}, x^+_{i' j'}]$ are $E$-nearly homotopic, and likewise the segments
    $C_{ij} [y^-_{ij}, y^+_{ij}]$ and $C_{i' j'} [y^-_{i' j'}, y^+_{i' j'}]$
    are $E$-nearly homotopic, for any $i, j, i', j' \in \{0, 1\}$.
    
    \item The segments $C_{0 j} [x^-_{0 j}, y^+_{0 j}]$ and $C_{1 j} [x^-_{1
    j}, y^+_{1 j}]$ are $E$-nearly homotopic.
    
    \item The segments $C_{i 0} [y^-_{i 0}, x^+_{i 0} + \len
    (C_{i0})]$ and $C_{i 1} [y^-_{i 1}, x^+_{i 1} + \len (C_{i1})]$ are $E$-nearly homotopic.
    
    \item $y^+_{00} - x^-_{00} \ge R + K$, and $x^+_{00} + \len (C_{00}) -
    y^-_{00} \ge R + K$.
  \end{enumerate}
  Then for $R > R_0$, we have
  \begin{equation}
    \label{square-1} \sum _{i, j = 0, 1} (- 1)^{i + j} C_{ij} = 0,
  \end{equation}
  in $\Pant_{10 \epsilon, R}$ homology.
\end{lemma}

\begin{remark} The hypothesis (5) is satisfied provided that $y^{+}_{00}-y^{-}_{00} \ge R$.
\end{remark}

\begin{proof}
  Set $K = 1+E- \log \epsilon$. For simplicity we write $\len
  (C_{ij}) = \len_{ij}$. We claim that we can find $x_{00} \in [x^-_{00} +
  \frac{K}{2}, x^+_{00} - \frac{K}{2} + 1]$ and $y_{00} \in [y^-_{00} +
  \frac{K}{2}, y^+_{00} - \frac{K}{2}]$ such that $y_{00} - x_{00} = R$. If
  $y^-_{00} \le x^-_{00} + R$, we let $x_{00} = x^-_{00} + \frac{K}{2}$ and
  $y_{00} = x_{00} + R$. If $y^-_{00} \ge x^-_{00} + R$, we let $y_{00} =
  y^-_{00} + \frac{K}{2}$, and $x_{00} = y_{00} - R$.
  
  By the Convergence Lemma (Lemma \ref{lemma-convergence}), and by the choice
  of the constant $K$, we can find $x^-_{ij} < x_{ij} < x^+_{ij}$ and
  $y^-_{ij} < y_{ij} < y^+_{ij}$ so that
  \[ \dist (C'_{ij} (x_{ij}), C'_{00} (x_{00})), \dist (C'_{ij} (y_{ij}),
     C'_{00} (y_{00})) \le \epsilon, \]
  and the pairs of geodesic segments $C_{0 i} [x_{0 i}, y_{0 i}]$ and $C_{1 i}
  [x_{1 i}, y_{1 i}]$, and $C_{i 0} [y_{i 0}, x_{i 0} + \len_{i0}]$ and $C_{i
  1} [y_{i 1}, x_{i 1} + \len_{i 1}]$ are $\epsilon$-nearly homotopic.
  
  Let $I_{ij} = y_{ij} - x_{ij}$ and $J_{ij} = x_{ij} + \len_{ij} - y_{ij}$, so $I_{ij}+J_{ij}=\len_{ij}$.
  Then $I_{00} = R$ and $J_{00} = \len_{00} - R$, so $|J_{00} - R| < 2
  \epsilon$.
  
  Also, by the triangle inequality we have $|I_{01} - R| = |I_{01} - I_{00} |
  < 2 \epsilon$. So
  \[ |J_{01} - R| \le |I_{01} - R| + | \len_{01} - 2 R| \le 4 \epsilon . \]
  Then
  \[ |J_{1 j} - R| \le |J_{0 j} - R| + |J_{1 j} - J_{0 j} | < 6 \epsilon, \]
  so
  \[ |I_{1 j} - R| \le |J_{1 j} - R| + | \len_{1 j} - 2 R| < 8 \epsilon . \]
  Therefore we get $|I_{ij} - R|, |J_{ij} - R| < 8 \epsilon$ for $i, j \in
  \{0, 1\}$.

\begin{figure}
  %\centering
  {
    
    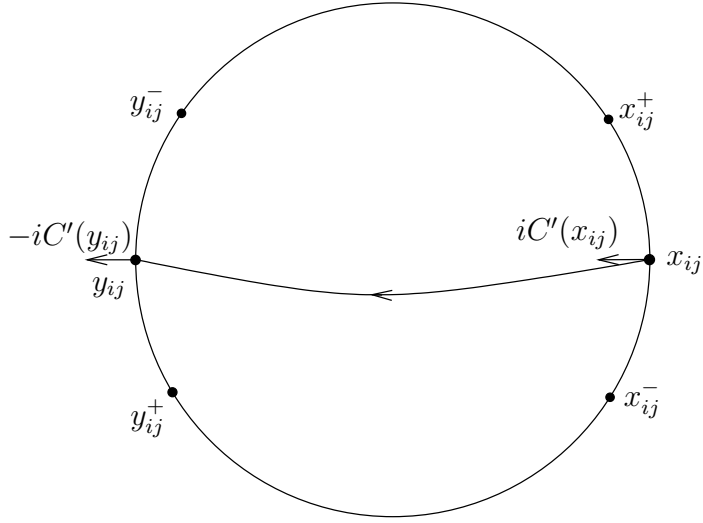
  }
  \caption{The Preliminary Geometric Square Lemma}
  \label{fig-prelim}
\end{figure}

  We take
  \[ \alpha_{00} \in \Conn_{\epsilon, R + \log 4} (\sqrt{-1}C' (x_{00}), - \sqrt{-1}C'
     (y_{00})), \]
  and let $\alpha_{ij}$ be the geodesic arc connecting $x_{ij}$ and $y_{ij}$
  that is $\epsilon$-nearly homotopic to $\alpha_{00}$ (see Figure \ref{fig-prelim}). 
Then $\dist (i  (\alpha_{ij}), i (\alpha_{00})), \dist (t (\alpha_{ij}), t (\alpha_{00}))
  \le 2 \epsilon$. Therefore, because $\dist (C'_{ij} (x_{ij}), C'_{00}
  (x_{00})) \le \epsilon$ and $\dist (C'_{ij} (y_{ij}), C'_{00} (y_{00})) \le
  \epsilon$, we have
  \[ \alpha_{ij} \in \Conn_{3 \epsilon, R + \log 4} (\sqrt{-1}C' (x_{ij}), - \sqrt{-1}C'
     (y_{ij})). \]
  Define $\Pi_{ij}$ as the pants generated from $C_{ij}$ by adding the third
  connection $\alpha_{ij}$. Denote by $A_{ij}$ and $B_{ij}$ the other two
  cuffs of $\Pi_{ij}$, oriented such that  \[ \partial \Pi_{ij} = C_{ij} - A_{ij} - B_{ij}, \]
  where $A_{ij}$ is freely homotopic to the closed broken geodesic $C_{ij}
  [x_{ij}, y_{ij}]  \alpha^{- 1}_{ij}$, and $B_{ij}$ to $C_{ij} [y_{ij},
  x_{ij} + \len_{ij}] \alpha_{ij}$.
  
  Applying Lemma \ref{lemma-ineff+}, we obtain

  \[ | \len (A_{ij}) - 2 R| < |I_{ij} - R| + 10 \epsilon < 20 \epsilon \]
  and similarly $| \len (B_{ij}) - 2 R| < 20 \epsilon$, so $\Pi_{ij} \in
  \Gamma_{10 \epsilon, R}$. Finally, $A_{i 0} = A_{i 1}$, and $B_{0 j} = B_{1
  j}$, so
  \[ 0=\sum _{i, j = 0, 1} (- 1)^{i + j} \partial \Pi_{ij} = \sum _{i, j = 0, 1}
     (- 1)^{i + j} C_{ij}, \]
  in $\Pant_{10 \epsilon, R}$ homology, which proves the lemma.
\end{proof}

\subsection{The Geometric Square Lemma}

\begin{lemma} [Geometric Square Lemma]\label{lemma-square} Let $E, \epsilon > 0$. There exist constants $K_1 =
  K_1 ( \Su, \epsilon, E) > 0$ and $R_0 ( \Su, \epsilon, E) > 0$ with the
  following properties. Suppose that we are given four oriented geodesics
  $C_{ij} \in \Gamma_{\epsilon, R}$, $i, j = 0, 1$, and for each $ij$ we are
  given 4 points $x^-_{ij} < x^+_{ij} < y^-_{ij} < y^+_{ij} < x^-_{ij} + \len
  (C_{ij})$. Assume that
  \begin{enumerate}
    \item The inequalities $x^+_{ij} - x^-_{ij} > K_1$, and $y^+_{ij} -
    y^-_{ij} > K_1$, hold.
    
    \item The segments $C_{ij} [x^-_{ij}, x^+_{ij}]$ and $C_{i' j'} [x^-_{i'
    j'}, x^+_{i' j'}]$, are $E$-nearly homotopic, and likewise the segments
    $C_{ij} [y^-_{ij}, y^+_{ij}]$ and $C_{i' j'} [y^-_{i' j'}, y^+_{i' j'}]$,
    are $E$-nearly homotopic, for any $i, j, i', j' \in \{0, 1\}$.
    
    \item The segments $C_{0 j} [x^-_{0 j}, y^+_{0 j}]$ and $C_{1 j} [x^-_{1
    j}, y^+_{1 j}]$ are $E$-nearly homotopic.
    
    \item The geodesic segments $C_{i 0} [y^-_{i 0}, x^+_{i 0} + \len
    (C_{ij})]$ and $C_{i 1} [y^-_{i 1}, x^+_{i 1} + \len (C_{ij})]$ are $E$-nearly homotopic.
  \end{enumerate}
  Then for $R > R_0$, we have
  \begin{equation}
    \label{square-2} \sum _{i, j = 0, 1} (- 1)^{i + j} C_{ij} = 0,
  \end{equation}
  in $\Pant_{100 \epsilon, R}$ homology.
\end{lemma}

\begin{proof} Below we use $L_0 = L_0 ( \Su, \epsilon, E)$ and $K_0 = ( \Su, \epsilon, E)$ to denote two sufficiently large constants whose values 
will be determined in the course of the argument. The constant $Q_0$ can depend on $K_0$ and $L_0$. The constants $K_1$ and $R_0$ (from the statement 
of the GSL) can depend on $K_0$ and $L_0$ and $Q_0$. Each of these constants will be implicitly  defined as a maximum of expressions 
in terms of constants which precede the given constant in the partial order of dependence which we just have described.

If we cannot apply the PGSL, then possibly interchanging the roles of the $x$' and the $y$'s, we find that
  
\begin{align*}
x^+_{00}  &\le y^-_{00}-{\len}_{00}+R+K({\epsilon},E)\\
&< y^-_{00}-R+K({\epsilon},E)+1,
\end{align*}
where $K = K (\epsilon, E)$ is the constant from the previous lemma. We then
let $y_{00} = y^-_{00} + Q_0$, and let $w_{00} = y_{00} - R$ (we assume that $Q_0 > K$). Then
  \begin{equation}
    \label{e-1} w_{00} > x^+_{00} +10,
  \end{equation}
  provided $Q_0>K+11$, and
$$  
y^-_{00} + Q_0 \le y_{00} \le y^+_{00} + Q_0-K_1,
$$
which implies

\begin{equation} \label{e-2} 
y^-_{00} + 2(E-\log \epsilon) +10 \le y_{00} \le y^+_{00} -2(E+\log \epsilon)-10 ,
\end{equation}
provided $Q_0 \ge 2(E-\log \epsilon) +10$ and $K_1 \ge Q_0+ 2(E-\log \epsilon) +10$.

Therefore by the Convergence Lemma we can find $y_{ij}$ in the
interval $[y^-_{ij}, y^+_{ij}]$ such that $\dist (C'_{ij} (y_{ij}), C'_{00}
  (y_{00})) \le \epsilon$. We then let
  
\begin{align*}
w_{ij} &= y_{ij}-R\\
&\ge x^+_{ij}+10 \\
(\text{provided} \, Q_0 &> E+K+12) \\
&\ge x^-_{ij}+K_1\\
&\ge x^{-}_{ij}+2(E- \log \epsilon) +10 \\
(\text{provided} \, K_1 &> 2 (E-\log \epsilon) +10).
  \end{align*}

  It follows from Lemma \ref{lemma-convergence} that $C_{i 0} [w_{i 0}, y_{i 0}]$ and $C_{i 1} [w_{i 1}, y_{i 1}]$ are
  $\epsilon$ $C^1$ nearly homotopic (two segments are $C^1$ nearly homotopic if the two initial and the two  terminal vectors are $\epsilon$ close in the tangent bundle respectively).
The point is that these two segments are contained into much larger segments that are $E$-nearly homotopic.

\begin{figure}
  %\centering
  {
    
    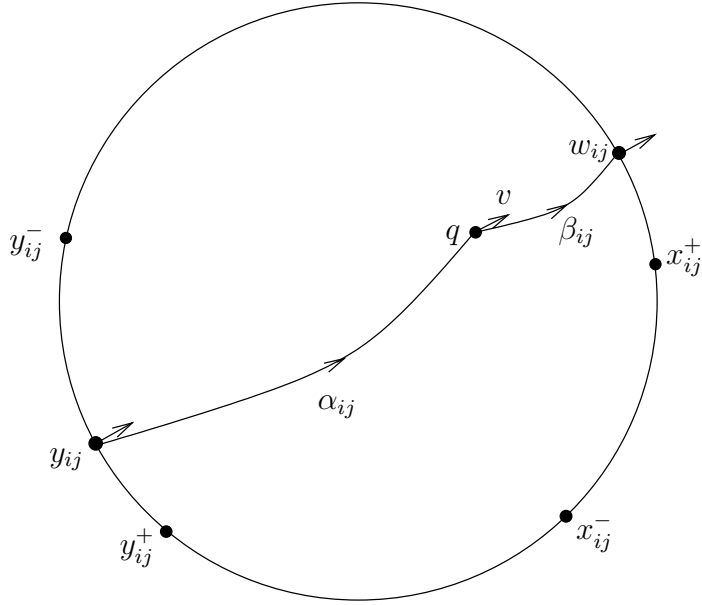
  }
  \caption{The Geometric Square Lemma}
  \label{fig-geom}
\end{figure}

  Let $(q, v) \in \TB \Su$, and take $\beta_{i 0} \in \Conn_{\epsilon, L_0}
  (v, -\sqrt{-1}C'_{i0} (w_{i 0}) )$ (where we assume that $L_0 > L_0 (\epsilon, \Su)$ and
  $L_0 (\epsilon, \Su)$ is the constant from the Connection Lemma (Lemma \ref{lemma-connection})). We take
  $\alpha_{00} \in \Conn_{\epsilon, R + \log 4 - L_0} (\sqrt{-1}C'_{00} (y_{00}), v)$ (see Figure \ref{fig-geom}).
  Then we find $\alpha_{ij} \in \Conn_{3 \epsilon, R + \log 4 - L_0} (\sqrt{-1}C'_{ij}
  (y_{ij}),  v)$, and $\beta_{ij} \in \Conn_{3 \epsilon, L_0} (v, -\sqrt{-1}C'_{ij}
  (w_{ij})$ such that $\alpha_{00}$ and $\alpha_{ij}$ are $\epsilon$-nearly
  homotopic and $\beta_{i 0}$ and $\beta_{i1}$ are $2 \epsilon$-nearly
  homotopic for every $i, j = 0, 1$.
  
  We let $\Pi_{ij}$ be the pair of pants generated by the geodesic segment
  $C_{ij} [w_{ij}, y_{ij}]$, the broken geodesic segment $\beta^{- 1} 
  \alpha^{- 1}_{ij}$, and the geodesic segment $\left( C_{ij} [y_{ij}, w_{ij}
  + \len_{ij}] \right)^{- 1}$. The reader can verify that it is a topological pair of
  pants.
  
  We let $A_{ij}$ be the closed geodesic freely homotopic to $\alpha_{ij}
   \beta_{ij}  C_{ij} [w_{ij}, y_{ij}]$ and let $B_{ij}$ be the one
  for $C_{ij} [y_{ij}, w_{ij} + \len_{ij}]  \beta^{- 1}_{ij} 
  \alpha^{- 1}_{ij}$. Then $\partial \Pi_{ij} = C_{ij} - A_{ij} - B_{ij}$.
  
  Using the second inequality from the remark just after the Sum of Inefficiencies Lemma (see  Lemma \ref{lemma-ineff+}), we find that $| \len (A_{ij}) - 2 R| \le 13 \epsilon$, and $| \len
  (B_{ij}) - 2 R| \le 15 \epsilon$.  Hence $\Pi_{ij} \in \Pant_{10 \epsilon, R}$.

  Observe that $A_{i 0} = A_{i 1}$, so
  \[ \sum _{i, j = 0, 1} (- 1)^{i + j} C_{ij} - \sum _{i, j = 0, 1} (- 1)^{i + j} \partial \Pi_{ij}=\sum _{i, j = 0, 1} (- 1)^{i + j} B_{ij} . \]
  Let the $a^{-}_{ij}$,  $a^{-}_{ij}$, $b^{-}_{ij}$ and  $b^{-}_{ij}$ be real numbers and $B_{ij} : \R \to B_{ij}$ be a
  parametrization of the geodesic $B_{ij}$ so that $B_{ij} (a^-_{ij})$,
  $B_{ij} (a^+_{ij})$, $B_{ij} (b^-_{ij})$, $B_{ij} (b^+_{ij})$ are the
  projections of points $q$, $C_{ij} (y^+_{ij})$, $C_{ij} (x^-_{ij})$ and
  $C_{ij} (x^+_{ij})$ respectively onto the geodesic $B_{ij}$. The points $q$,
  $C_{ij} (y^+_{ij})$, $C_{ij} (x^-_{ij})$ and $C_{ij} (x^+_{ij})$ belong to
  the broken geodesic $C_{ij} [y_{ij}, w_{ij} + \len_{ij}]  \beta^{-
  1}_{ij}  \alpha^{- 1}_{ij}$, and we project them to $B_{ij}$ by
  choosing lifts of $B_{ij}$ and $C_{ij} [y_{ij}, w_{ij} + \len_{ij}] 
  \beta^{- 1}_{ij}  \alpha^{- 1}_{ij}$ in $\Ha$ that have the same
  endpoints and then use the standard projection onto the lift of $B_{ij}$.
  
  It follows from the Convergence Estimate that each of $q$, $C_{ij}
  (y^+_{ij})$, $C_{ij} (x^-_{ij})$ and $C_{ij} (x^+_{ij})$ are within distance
  $1$ of the corresponding projections on $B_{ij}$. Then
  
\begin{align*}
b^+_{ij}-b^-_{ij} &\ge x^+_{ij}-x^-_{ij}-2\\
&\ge K_0\\
\text{(provided} \, K_1 &>K_0+2),
\end{align*}
and
\begin{align*}
a^+_{ij}-a^-_{ij} &\ge R-L_0-3+ K_1-Q_0-E-1 \\
&\ge R+K_0 \\
\text{(provided} \, K_1&>K_0+Q_0+L_0+E+4).
  \end{align*}
  
  Assuming that $K_0 \ge K (10 \epsilon, E + 2)$ (where $K$ is the constant
  from the PGSL) we find that the differences
  $b^+_{ij} - b^-_{ij}$ and $a^+_{ij} - a^-_{ij}$ satisfy the lower bound from
  the PGSL (observe that $b^{+}_{ij}$ and $b^{-}_{ij}$ are $E+2$ close and similarly for the $a$'s).
  
  Also, the $B_{ij}$'s are in $\Gamma_{10 \epsilon, R}$.  So we apply the
  PGSL to show that
  \[ \sum _{i, j = 0, 1} (- 1)^{i + j} B_{ij} = 0 \]
  in $\Pant_{100 \epsilon, R}$ homology.

Here we explain why the assumptions of the PGSL are satisfied. For each $i$, piecewise geodesics $\alpha{i0}^{-}C_{i0}[y_{i0},x_{i0}^{+}]$ and  $\alpha{i1}^{-}C_{i1}[y_{i1},x_{i1}^{+}]$, are $E$-nearly homotopic;
it follows from Lemma \ref{lemma-proj-arc} that $B_{i0}[a_{i0}^{-},a_{i0}^{+}]$ and  $B_{i1}[a_{i1}^{-},a_{i1}^{+}]$ are $(E+4)$-nearly homotopic.

Likewise, $C_{0j}[x_{0j}, w_{0j}] \beta_{0j}^{-1}\alpha_{0j}^{-1}C_{0j}[y_{0j},y_{0j}^{+}]$ and  $C_{1j}[x_{1j}, w_{1j}] \beta_{1j}^{-1}\alpha_{1j}^{-1}C_{1j}[y_{1j},y_{1j}^{+}]$ are $E$-nearly homotopic
(because the individual segments  are) and hence $B_{0j}[b_{0j}^{-}, a_{0j}^{+}]$ and $B_{1j}[b_{1j}^{-}, a_{1j}^{+}]$ are.

\end{proof}

\begin{random} Randomization remarks for the GSL. Let $\epsilon,E>0$. Every constant $K$ below may depend only on $\epsilon$, $\Su$ and $E$.

Below we will define a partial map $g:\left(\overset{....}\Gamma_{1,R}\right)^{4} \to \R \Pant_{100\epsilon,R}$ 
 such that

\begin{enumerate}
\item $g$ is defined on any input  $(C_{ij},x^{\pm}_{ij},y^{\pm}_{ij})$ that satisfies the hypothesis $1-4$ of GSL, 
\item $\sum(-1)^{i+j} C_{ij}=\partial g(C_{ij},x^{\pm}_{ij},y^{\pm}_{ij})$,
\item $g$ is $K$-semirandom with respect to  measures classes $\Sigma^{\boxtimes 4}_{\overset{....}{\Gamma}}$ on $\left(\overset{....}{\Gamma}_{1,R}\right)^{4}$ and
$\sigma_{\Pant}$ on $\Pant_{1,R}$.

\end{enumerate}

We first define a partial function $g_0:\left(\overset{....}\Gamma_{1,R}\right)^{4} \to \R \Pant_{10\epsilon,R}$ that is defined on inputs  $(C_{ij},x^{\pm}_{ij},y^{\pm}_{ij})$ that satisfy the extra hypothesis $(5)$ from the PGSL. Given such an input, we follow the construction of the PGSL to construct $x_{ij}$ and $y_{ij}$, and we observe that because these new points are bounded distance from the old ones, the map $(C_{ij},x^{\pm}_{ij},y^{\pm}_{ij}) \to (C_{ij},x_{ij},y_{ij})$ is  $K$-semirandom  as a partial map from 
$\left(\overset{....}\Gamma_{1,R}\right)^{4}$ to     $\left(\overset{..}\Gamma_{1,R}\right)^{4}$, with respect to the measure classes $\Sigma^{\boxtimes 4}_{\overset{....}{\Gamma}}$ on $\left(\overset{....}{\Gamma}_{1,R}\right)^{4}$ and $\Sigma^{\boxtimes 4}_{\overset{..}{\Gamma}}$ on $\left(\overset{..}{\Gamma}_{1,R}\right)^{4}$.

Then we take a random third connection 
$$
\alpha_{00} \in \Conn_{\epsilon,R+\log 4}(\sqrt{-1}C'_{00}(x_{00}),-\sqrt{-1}C'_{00}(y_{00})).
$$
Likewise for $\alpha_{ij}$. Adding the third connection $\alpha_{ij}$ to $C_{ij}$ we obtain the pants $\Pi_{ij}(\alpha_{ij})$. 
We claim that distinct $\alpha_{ij}$ lead to distinct pants $\Pi_{ij}(\alpha_{ij})$. The third connection $\alpha_{ij}$ is $\epsilon$-close to the unique simple  geodesic arc on $\Pi_{ij}(\alpha_{ij})$ that is orthogonal to $\gamma_{ij}$ at both ends. On the other hand, no two distinct  $\alpha_{ij}$ are $\epsilon$-close, so assuming that the injectivity radius of the surface $\Su$ is at least $2\epsilon$   we find that distinct $\alpha_{ij}$ give distinct $\Pi_{ij}(\alpha_{ij})$.

So, for each input $(C_{ij},x_{ij},y_{ij})$, by adding a random third connection $\alpha_{ij}$ we construct the pants $\Pi_{ij}(\alpha_{ij})$. 
So far, we have been using the term ``random" to mean arbitrary. In these randomization remarks we will also interpret the phrase ``a random element of a finite set $S$" as ``the random element
of $\R S$, namely $\frac{1}{|S|}\sum_{x \in S} x$.

We can then think of every map $f:S \to T$ that we have implicitly constructed in the text as the associated linear map $f:\R S \to \R T$ defined by $f(\sum a_i x_i)=\sum a_if(x_i)$. 
So, for example, we let 
$$
\ul{\alpha}_{ij} \in \Conn_{\epsilon,R+\log 4}(\sqrt{-1}C'_{ij}(x_{ij}),-\sqrt{-1}C'_{ij}(y_{ij}))
$$ 
be the random element of 
$$
\Conn_{\epsilon,R+\log 4}(\sqrt{-1}C'_{ij}(x_{ij}),-\sqrt{-1}C'_{ij}(y_{ij})), 
$$
and then $\Pi_{ij}(\ul{\alpha}_{ij})$ is the image of $\ul{\alpha}_{ij}$ by the linear form of the map $\alpha_{ij} \to \Pi_{ij}(\alpha_{ij})$.

In this manner we have constructed a partial map from $\overset{..}{\Gamma}_{1,R} \to \Pi_{1,R}$ (defined by $(C_{ij},x_{ij},y_{ij}) \to \Pi_{ij}(\ul{\alpha}_{ij})$, 
compare with Lemma \ref{lemma-number} ), and we claim that it is $K$-semirandom with respect to $\Sigma_{\overset{..}{\Gamma}}$  and $\sigma_{\Pant}$. To verify this claim we need to show that 
for any given pants $\Pi \in \Pant_{1,R}$, the weight of $\Pi$ is at most $Ke^{-3R}$. Let $C$ be a cuff of $\Pi$, and choose points $x,y \in C$ which lie in certain unit length intervals on  $C$. Let $\Conn$
be the set of all good third connections between $x$ and $y$ (by this we mean all connections $\alpha$ so that $C$ and $\alpha$ produce a pair of pants in $\Pant_{1,R}$). The set $\Conn$ has approximately $e^{R-K}$ elements. Moreover, there is a unique third connection $\alpha \in \Conn$ so that  $\alpha$ and $C$ yield the given pair of pants $\Pi$. So, the total weight of $\Pi$ is at most $e^{K-R}$ times the total weight for the three choices of $C \in \partial \Pi$ (with associated unit intervals), and we conclude that the total weight for $\Pi$ is at most $3e^{K-R}e^{-2R}=Ke^{-3R}$.

Also, the map $\big( \Pi_{ij} \big)_{i,j \in \{0,1\} } \mapsto \sum(-1)^{i+j} \Pi_{ij}$, is of course  $4$-semirandom from  $(\Pant^{4}_{1,R},\Sigma^{\boxtimes 4}_{\Pant})$ to $(\R \Pant_{1,R},\Sigma_{\Pant})$. Composing the above maps we construct the map $g_0$ and see that $g_0$ is $K$-semirandom.

For the general case, similarly as above we first define the map 
$$
h:(C_{ij},x^{\pm}_{ij},y^{\pm}_{ij}) \to \R \Pant_{1,R}
$$ 
according to our second construction, on every input $(C_{ij},x^{\pm}_{ij},y^{\pm}_{ij})$ that satisfies  conditions $(1)-(4)$ of the GSL, but not  condition $(5)$ of the PGSL.

We construct $y_{ij}$ and $w_{ij}$ as before. The map $(C_{ij},x^{\pm}_{ij},y^{\pm}_{ij}) \to (C_{ij},y_{ij},w_{ij})$ is $K$-semirandom. Then we find the connections  
$\alpha_{ij}$ and $\beta_{ij}$. There are at least  $e^{R-K}$ of the $\alpha_{ij}$ (we only fix a single $\beta_{ij}$), and each third connection $\alpha_{ij} \beta_{ij}$ leads to a new pair of pants $\Pi_{ij}(\alpha_{ij}\beta_{ij})$. Let $N$ denote the number of connections $\alpha_{ij}$ (by construction, the number $N$ does not depend on $i$ and $j$).
This defines the map 
$$
h(C_{ij},x^{\pm}_{ij},y^{\pm}_{ij})=\sum \frac{1}{N} \Pi_{ij}(\alpha_{ij}\beta_{ij}),
$$ 
and we can verify that $h$ is $K$-semirandom.

Then we observe that $\partial_{B}:\Pi_{ij} \to B_{ij}$ formed by taking the appropriate boundary curve of the $\Pi_{ij}$ we constructed is $K$-semirandom, so the induced map $\wt{h}:(C_{ij},x^{\pm}_{ij},y^{\pm}_{ij}) \to (B_{ij},a_{ij},b_{ij})$ is as well. So the map $g_1$ defined by 
$g_1(C_{ij},x^{\pm}_{ij},y^{\pm}_{ij})=\sum(-1)^{i+j}\Pi_{ij}+g_0(B_{ij},a^{\pm}_{ij},b^{\pm}_{ij})$ is $K$-semirandom, and hence $g=g_0+g_1$ is as well.

\end{random}

\section{The Algebraic Square Lemma}

We prove the Algebraic Square Lemma which will be used in almost all of our subsequent identities in the Good Pants Homology. In particular it will allow us to encode an element of $\pi_1(\Su,*)$ as a sum of good pants and then prove that the encoding of products of elements of $\pi_1(\Su,*)$ is the sum of their encodings.

\subsection{Notation} Let $* \in \Su$ denote a point that we fix once and for all. By $\pi_1(\Su,*)$ we denote the fundamental group
of a pointed surface.  If $A \in \pi_1(\Su,*)$, we  let $\cdot A \cdot$ be the geodesic segment from $*$ to $*$ homotopic to $A$. 
By $[A]$ we denote the closed geodesic on $\Su$ that is freely homotopic to  $\cdot A \cdot$. 
If $A_1,...,A_n \in \pi_1(\Su,*)$ we let $\cdot A_1 \cdot A_2...\cdot A_n \cdot$ be the piecewise geodesic arc that is the concatenation
of the arcs $\cdot A_i \cdot$. We let  $[\cdot A_1 \cdot A_2...\cdot A_n \cdot]$ be the closed piecewise geodesic that arises from the arc
$\cdot A_1 \cdot A_2...\cdot A_n \cdot$ by noticing that the starting and the ending point of $\cdot A_1 \cdot A_2...\cdot A_n \cdot$ are the same.
By $\len([A])$ is the length of the closed geodesic $[A]$. By $\len(\cd A \cd )$ we mean of course the length of the geodesic arc $\cdot A \cdot$, and in general by $\len(\cdot A_1 \cdot ...\cdot A_n \cdot)$ the length of the corresponding piecewise geodesic arc.

\begin{remark} Observe that for any $X_i \in \pi_1(\Su,*)$, $i=0,...,n-1$, the closed geodesics $[X_jX_{j+1}...X_{n+j-1}]$ are one and the same (we are taking the indices modulo $n$). We will call this rotation and often use it  without warning. 
\end{remark}

We remind the reader that $\cdot AB \cdot$ is a geodesic arc from $*$ to $*$ representing $AB$, while $\cdot A \cdot B \cdot$ is a concatenation of two geodesic arcs.
Similarly $\cdot AB \cdot C \cdot$ is a concatenation of two geodesic arcs, while  $\cdot A \cdot B \cdot C \cdot$ is a concatenation of three, and so on. 

In particular, we have the following statements about the inefficiency function,
$$
I(\cdot A_1 \cdot ...\cdot A_n \cdot)=\sum \len(\cd A_i \cd )-\len(\cd A_1 A_2....A_n \cd ),
$$
and
$$
I([\cdot A_1 \cdot ...\cdot A_n \cdot])=\sum \len( \cd A_i \cd )-\len([A_1....A_n]).
$$
Notice that we may have (and will usually have)
$$
I([\cdot A_1 \cdot ....\cdot A_n \cdot])>I(\cdot A_1 \cdot ....\cdot A_n \cdot).
$$

\subsection{The Algebraic Square Lemma (the ASL)} 

The following is the Algebraic Square Lemma.

\begin{lemma}[Algebraic Square Lemma]\label{lemma-algebraic}  Let $\epsilon, \Delta>0$. There exist  constants $K(\Su,\epsilon,\Delta)=K$ and $R_0=R_0(\Su,\epsilon,\Delta)$ so that for $R>R_0$ the following holds. 
Let $A_i,B_i,U,V \in \pi_1(\Su,*)$, $i=0,1$, be such that

\begin{enumerate}
\item $\big| \len([A_iUB_jV])-2R \big| <2 \epsilon$,  $i,j=0,1$,
\item $I([\cdot A_i \cdot U \cdot B_j \cdot V \cdot ])<\Delta$,
\item $\len(\cd U \cd ),\len(\cd V \cd ) >K$.
\end{enumerate}
Then 
$$
\sum\limits_{ij} (-1)^{i+j} [A_iUB_jV]=0
$$
in $\Pant_{100\epsilon,R}$ homology. 

\end{lemma}

\begin{proof} For each $i,j \in \{0,1\}$ we project the closed piecewise geodesic $[\cdot A_i \cdot U \cdot B_j \cdot V \cdot ]$ onto the closed geodesic 
$\gamma_{ij}=[A_iUB_jV]$. By Lemma \ref{lemma-proj-closed} we find that each appearance of $*$ is moved at most distance $E=\frac{\Delta}{2}+1$ by the projections. Let 
$\gamma_{ij}(x^{\pm}_{ij})$ and $\gamma_{ij}(y^{\pm}_{ij})$ be the projections of $*$ on $\gamma_{ij}$ before and after $U$, and before and after $V$, respectively. 
Then providing that our $K$ is at least $2E$ plus the corresponding constant from the GSL, we have 
$x^{-}_{ij}<x^{+}_{ij}<y^{-}_{ij}<y^{+}_{ij}<x^{-}_{ij}+\len(\gamma_{ij})$ and the hypotheses of the  Geometric Square Lemma.  We conclude that
$$
\sum\limits_{ij} (-1)^{i+j} [A_iUB_jV]=0
$$
in $\Pant_{100\epsilon,R}$ homology. 

\end{proof}

\begin{random} Randomization remarks for the ASL. Let $\epsilon,\delta>0$. By $K$ we denote any constant that may depend only on $\epsilon$, $\Su$, and $\Delta$.

Below we will define a partial map 
$$
f:G^{2} \times G \times G^{2} \times G \to \R \Pant_{1,R},
$$
such that

\begin{enumerate}

\item $f$ is defined on any input  $(A_i,U,B_j,V)$ that satisfies the assumptions of the ASL, 

\item $\sum(-1)^{i+j}  [A_iUB_jV] =\partial f(A_i,U,B_j,V)$,

\item $f$ is $K$-semirandom with respect to the classes of measures  $\Sigma^{\boxtimes 2}_G \times \Sigma_G \times \Sigma^{\boxtimes 2}_G \times \Sigma_G$ on $
G^{2} \times  G \times G^{2} \times G$ and  $\sigma_{\Pant}$ on $\Pant_{1,R}$.

\end{enumerate}

Let  $h$ be a partial map
$$
h:  G^{2} \times G \times G^{2} \times G \to   \left(\overset{....}\Gamma_{1,R}\right)^{4} 
$$
defined by letting $h(A_i,U,B_j,V)=(C_{ij},x^{\pm}_{ij},y^{\pm}_{ij})$, where $C_{ij}=[A_iUB_jV]$, and $x^{\pm}_{ij}$ and $y^{\pm}_{ij}$ are the points on the parameterizing torus for $C_{ij}$ such that the points  $C_{ij}(x^{\pm}_{ij})$ and $C_{ij}(y^{\pm}_{ij})$ are the corresponding projections of the 4 copies of the base point $*$ (that belong to  the closed piecewise geodesic 
$[\cd A_i \cd U \cd B_j \cd V \cd ]$) to the closed geodesic $C_{ij}$ (these projections were defined above). It follows from Lemma \ref{lemma-R-product} and  Lemma \ref{lemma-R-pro} that $h$ is $K$-semirandom.
Let  $g:\left(\overset{....}\Gamma_{1,R}\right)^{4} \to \R \Pant_{1,R}$ be the $K$-semirandom map from the previous section (see the Randomization remarks for the GSL). Then $f=g \circ h$ is $K$-semirandom.

\end{random}

\subsection{The  Sum of Inefficiencies Lemma in the algebraic notation} The following lemma follows from   Lemma \ref{lemma-ineff+}.

\begin{lemma}[Sum of Inefficiencies Lemma in the algebraic notation]\label{lemma-ineff} Let $\epsilon,\Delta>0$ and $n \in \N$. There exists $L=L(\epsilon,\Delta,n)>0$ such that if $U_1,...,U_{n+1}=U_1,X_1,...X_n \in \pi_1(\Su,*)$, and 
$I(\cdot U_i \cdot X_i \cdot U_{i+1}) \le \Delta$, and $\len(\cd U_i \cd ) \ge L$, then

$$
\left| I([\cdot U_1 \cdot X_1 \cdot U_2 \cdot X_2 \cdot ... \cdot U_n \cdot X_n \cdot ]) - \sum\limits_{i=1}^{n} I(\cdot U_i \cdot   X_i \cdot U_{i+1} \cdot) \right| \le \epsilon.
$$
\end{lemma}

\begin{remark} In particular,  we can leave out the $X$'s in the above lemma, and write 
$$
\left| I([\cdot U_1  \cdot U_2 \cdot  ... \cdot U_n  \cdot ]) - \sum\limits_{i=1}^{n} I(\cdot U_i \cdot  U_{i+1} \cdot ) \right| \le \epsilon,
$$
providing that $I(\cdot U_i  \cdot U_{i+1} \cdot ) \le \Delta$, and $\len(\cd U_i \cd ) \ge L$. Moreover,  by the Long Segment Lemma for Angles (for $L$ large enough) we have
$$
\left| I([\cdot U_1  \cdot U_2 \cdot  ... \cdot U_n  \cdot ]) - \sum\limits_{i=1}^{n} I(\theta_i) \right| \le 2\epsilon,
$$
where $\theta_i=\Ang(t(\cd U_i \cd ),i (\cd U_{i+1} \cd ))$.

\end{remark}

Similarly, the following lemma follows from Lemma \ref{lemma-ineff+-1}.

\begin{lemma}\label{lemma-ineff-1} Let $\epsilon,\Delta>0$ and $n \in \N$. There exists $L=L(\epsilon,\Delta,n)>0$ such that if $U_1,...,U_{n+1}=U_1 \in \pi_1(\Su,*)$ and 
$X_{11},...X_{1j_{1}},... X_{n1},...X_{nj_{n}} \in \pi_1(\Su,*)$, and 
$I(\cdot U_i \cdot X_i \cdot U_{i+1}) \le \Delta$, and $\len(\cd U_i \cd ) \ge L$, then

$$
\left| I([\cdot U_1 \cdot X_{11} \cdot ...\cdot X_{1j_{1}} \cdot ...  \cdot U_{n} \cdot X_{n1} \cdot ... \cdot X_{nj_{n}} \cdot  ]) - \sum\limits_{i=1}^{n} I(\cdot U_i \cdot X_{i1}...X_{ij_{i}} \cdot U_{i+1} \cdot ) \right| \le \epsilon.
$$
\end{lemma}

Finally, we have the Flipping Lemma.

\subsection{The Flipping Lemma} For $X \in \pi_1(\Su,*)$ we let $\ov X=X^{-1}$ denote the inverse of $X$.

\begin{lemma}[Flipping Lemma]\label{lemma-flip} Let $\epsilon, \Delta>0$. There exists a constant $L=L(\epsilon,\Delta)>0$ with the following properties.
Suppose $A,B,T \in \pi_1(\Su,*)$, and 
$$
I(\cd T \cd A \cd \ov T \cd), I(\cd \ov T \cd B \cd T \cd) \le \Delta, 
$$
and $\len(\cd T \cd ) \ge L$. Then
$$
\left| I([\cd T \cd A \cd \ov T \cd B \cd ]) - I( [\cd T \cd \ov A \cd \ov T \cd B \cd ]) \right| < \epsilon,
$$
and therefore 
$$
\left|\len([TA \ov T B])- \len ([ T \ov A  \ov T B]) \right|< \epsilon.
$$
\end{lemma}

\begin{proof} By the Long Segment Lemmas,

$$
\left| I(\cd T \cd A \cd \ov T \cd B \cd T \cd )-I(\cd T \cd A \cd \ov T \cd )-I( \cd \ov T \cd B \cd T \cd ) \right| < \frac{\epsilon}{4} 
$$
and
$$
\left| I(\cd T \cd A \cd \ov T \cd B \cd  )-I(\cd T \cd A \cd \ov T \cd B \cd T \cd ) \right| < \frac{\epsilon}{4}. 
$$
Likewise
$$
\left| I(\cd T \cd  \ov A \cd \ov T \cd B \cd T \cd )-I(\cd T \cd \ov A \cd \ov T \cd )-I( \cd \ov T \cd B \cd T \cd ) \right| < \frac{\epsilon}{4} 
$$
and
$$
\left| I(\cd T \cd \ov A \cd \ov T \cd B \cd  )-I(\cd T \cd \ov A \cd \ov T \cd B \cd T \cd ) \right| < \frac{\epsilon}{4}. 
$$
But $I(\cd T \cd A \cd \ov T \cd )=I(\cd T \cd \ov A \cd \ov T \cd )$, because $\cd T \cd A \cd \ov T \cd $ is the same as $\cd T \cd \ov A \cd \ov T \cd$ with reversed orientation.
So
$$
\left| I([\cd T \cd A \cd \ov T \cd B \cd ]) - I( [\cd T \cd \ov A \cd \ov T \cd B \cd ]) \right| < \epsilon.
$$
Similarly we conclude $\left|\len([TA \ov T B])- \len ([ T \ov A \ov T B]) \right|< \epsilon$.

\end{proof}

\section{Applications of the Algebraic Square Lemma} In this section we will describe the encoding of an element $A$ of $\pi_1(\Su,*)$ as a sum $A_T$ of good pants (the encoding depends on a choice of a sufficiently large  element $T$ of $\pi_1(\Su,*)$).

In brief, we let
$$
A_T=\frac{1}{2}\left( [TA\ov{T}B]-[T\ov{A}\ov{T}B]\right)
$$
for suitable $B$, and then observe that the Algebraic Square Lemma implies that different choices of $B$ give the same element of the good pants homology.

We can then easily prove that 
$$
[TA\ov{T}B]=A_T+B_{\ov{T}},
$$
which we call the Two-Part Itemization Lemma. We want to go one step further and prove that
$$
[TA\ov{T}BTC\ov{T}D]=A_T+B_{\ov{T}}+C_T+D_{\ov{T}},
$$
for suitable $A,B,C,D$ and $T$. 

It turns out that  in order to prove this Four-Part Itemization Lemma we must first prove that
$$
[TA\ov{T}BTC\ov{T}D]=[TA\ov{T}DTC\ov{T}B].
$$
This is indeed the most difficult lemma of this section.

We would then be able to go ahead and prove a Six-Part Itemization Lemma and so forth but the Four-Part Itemization Lemma is sufficient for our purposes.

We state several results and  definitions (notably the definition of $A_T$ in the next lemma), 
which depend on an element $T \in \pi_1(\Su,*)$ and $\Delta>0$. We treat both $T$ and $\Delta$ as parameters, and the exact value of both $T$ and $\Delta$ (that are then 
used in the proof of the main theorem) will be determined in Section 9.

\subsection{The definition of $A_T$} For $A,T \in \pi_1(\Su,*)$, and $\epsilon,R>0$, we let $\FConn_{\epsilon,R}(A,T)$ be the set all
$B \in \pi_1(\Su,*)$ such that $[TA \ov T B],[T \ov A \ov T B] \in \Gamma_{\epsilon,R}$, and $I(\cd \ov T \cd B \cd T \cd )<1$.

\begin{lemma}\label{lemma-A-T} Let $\epsilon,\Delta>0$. There exists a constant $L=L(\Su,\epsilon,\Delta)$ such that if $A,T \in \pi_1(\Su,*)$ and $I(\cd T \cd A \cd \ov T \cd ) \le \Delta$,
and $\len(\cd T \cd ) \ge L$, and  $2R-\len(\cd A \cd )-2\len(\cd T \cd ) \ge L$, then 
\begin{enumerate}
\item  $\FConn_{\epsilon,R}(A,T)$ is non-empty, and $\log  \big| \FConn_{\epsilon,R}(A,T) \big| \ge 2R-\len(\cd A \cd )-2\len(\cd T \cd )-\Delta-L$,
\item $[TA \ov T B]-[T \ov A \ov T B]= [TA \ov T B']-[T \ov A \ov T B']$ in $\Pant_{100\epsilon,R}$ homology for any $B,B' \in \FConn_{\epsilon,R}(A,T)$.
\end{enumerate}
We then let 
$$
A_T=\frac{1}{2} \big( [TA \ov T B]-[T \ov A \ov T B] \big)
$$
for a random $B \in \FConn_{\epsilon,R}(A,T)$.
\end{lemma}

\begin{remark} The part $(2)$ of Lemma \ref{lemma-A-T} implies that for any  given $B \in \FConn_{\epsilon,R}(A,T)$ we have
$$
A_T=\frac{1}{2} \big( [TA \ov T B]-[T \ov A \ov T B] \big)
$$
in $\Pant_{100\epsilon,R}$ homology. Also, it is important to note that 
$[A]$ is equal to $A_T$ in the standard homology ${\bf H}_1$.
\end{remark}

\begin{proof} Suppose that $\cd B \cd \in \Conn_{\epsilon,R'}(-i(\cd T \cd ),i(\cd T \cd ) )$, where $R'=2R-\len(\cd A \cd )-2 \len(\cd T \cd )-I(\cd T \cd A \cd \ov T \cd )$.
The set $\Conn_{\epsilon,R'}(-i(\cd T \cd ),i(\cd T \cd ) )$ will be non-empty  (by the Connection lemma (Lemma \ref{lemma-connection})) provided $L$ is large.  Then, by the Sum of Inefficiencies Lemma,
$$
\left| \len([TA \ov T B])-2R \right|<\epsilon+O(\epsilon^{2}),
$$
and 
$$
\left| \len([T \ov A \ov T B])-2R \right|<\epsilon+O(\epsilon^{2}),
$$
provided $\len(\cd T \cd )$ is large. Thus, with slight abuse of notation we have
$$
\Conn_{\epsilon,R'}(-i(\cd T \cd ),i(\cd T \cd ) )  \subset \FConn_{\epsilon,R}(A,T ),
$$
and 
$$
\log \left| \Conn_{\epsilon,R'}(-i(\cd T \cd ), i( \cd T \cd ) \right| \ge 2R- \len(\cd A \cd )-2\len(\cd T \cd )-L
$$
if $L$ is large, so we have proved the statement (1) of the lemma.

Again, by the Sum of Inefficiencies Lemma, the inefficiency of the piecewise geodesic $[\cdot T \cdot A \cdot \ov{T} \cdot B \cdot]$ is at most $\Delta+2$. Then the 
statement (2) then follows, provided $L$ (and hence ($\len(\cd T \cd )$) is large, from the Algebraic Square Lemma.

\end{proof}

\begin{random} Randomization remarks for $A_T$.  All constants $K$ may depend only on $\epsilon$, $\Delta$, and $\Su$ and $T \in \pi_1(\Su,*)$. 

Letting $\ul{B}_A \in \R G$ denote the random element of $\FConn_{\epsilon,R}(A,T)$, we consider the map $A \to \ul{B}_A$ from $G$ to $\R G$. If $\len(\cd A \cd ) \in [a,a+1]$, we find that $\len(\cd B \cd ) \in [L_a,R_a]$, where $L_a=2R-a-2 \len(\cd T \cd ) -\Delta -4$ and $R_a=2R-a-2 \len(\cd T \cd )$, for all $B \in \FConn_{\epsilon,R}(A,T)$. Because 

$$
\big| \FConn_{\epsilon,R}(A,T) \big|> e^{L_{a}-L}
$$
(where $L=L(\epsilon,\Su)$ from the Connection Lemma), and $\sigma_a(G) \le K$ (see Appendix for  the definition of $\sigma_a$), we find that for any $X \in G$
$$
(A \to \ul{B}_A )_{*} \sigma_a(X) \le K e^{L-L_{a}},
$$
if $\len(\cd X \cd ) \in [L_a,R_a]$, and $(A \to \ul{B}_A)_{*} \sigma_a(X)=0$ otherwise. This implies

$$
(A \to \ul{B}_A)_{*} \sigma_a \le K \sum_{k=\lfloor L_a \rfloor}^{\lfloor R_b \rfloor} e^{k+L-L_{a}} \sigma_k,
$$
which in turn implies that the map $A \to \ul{B}_A$ is $K$-semirandom with respect to $\Sigma_G$ and $\Sigma_G$.

We define $[A \ov T \ul{B}_A T]$ by
$$
[A \ov T \ul{B}_A T]=\frac{1}{|\FConn_{\epsilon,R}(A,T)|}\sum_{B \in \FConn_{\epsilon,R}(A,T) }[A \ov T B T]
$$

The map $A \to (A,T)$, is $e^{\len(\cd T \cd)}$ semirandom by the remark stated just before the Principles of andomization section in the Appendix.

Then the partial maps from $G$ to $\R G^{4}$ defined by $A \to (A,\ov T ,\ul{B}_A,T)$ and $A \to (\ov A, \ov T, \ul{B}_A,T)$, are $Ke^{2\len(\cd T \cd )}$-semirandom with respect to $\Sigma_G$ and $\Sigma^{4}_G$ and hence the map
$$
A \to A_T=\frac{1}{2}\left([A \ov T \ul{B}_A T]-[\ov A \ov T \ul{B}_A T] \right)
$$
is $Ke^{2\len(\cd T \cd )}$-semirandom with respect to  $\Sigma_G$ and $\sigma_{\Gamma}$.

The map $(A,B') \to (A,\ul{B}_A,B')$ is $K$-semirandom with respect to $\Sigma^{\times 2}_G$ and $\Sigma^{\times 3}_G$, and $(A,\ul{B}_A,B') \to (A,\ov A,\ov T,\ul{B}_A,B',T)$ is $Ke^{2\len(\cd T \cd )}$-semirandom with respect to $\Sigma^{\times 3}_G$ and $\Sigma^{\boxtimes 2}_G \times \Sigma_{G} \times \Sigma^{\times 2}_G \times \Sigma_G$.  

Also, by the Algebraic Square Lemma, the map $(A,\ov A, \ov T, B, B',T) \to  \Pi \in \R \Pant_{1,R}$, such that 
$$
\partial \Pi=[TA \ov T B]-[T \ov A \ov T B]-[TA \ov T B']+[T \ov A \ov T B'],
$$
is $K$-semirandom from  $G^{2} \times G \times G^{2} \times G$ to $\Pant_{1,R}$, with respect to the measure classes  
$\Sigma^{\boxtimes 2}_G \times \Sigma_G \times \Sigma^{\boxtimes 2}_{G} \times \Sigma_G$ and $\sigma_{\Pant}$.
Composing the above mappings we find a $Ke^{2\len(\cd T \cd )}$-semirandom map $g:G^{2} \to \R \Pant_{1,R}$ such that
$$
\partial g(A,B')= A_T -\frac{1}{2}\left([TA \ov T B' ]-[T \ov A \ov T B'] \right).
$$

\end{random}

\begin{remark} At the end of the paper we will see that $T$ and $\Delta$ only depend on $\Su$ and $\epsilon$.

\end{remark}

\subsection{The Two-part Itemization Lemma}

The following lemma is a corollary of the previous one and we refer to it as the Two-part Itemization Lemma.

\begin{lemma}[Two-part Itemization Lemma] Let $\epsilon,\Delta>0$. There exists a constant $L=L(\Su,\epsilon,\Delta)>0$ such that for any $A,B,T \in \pi_1(\Su,*)$ such that $[TA \ov{T} B] \in \Gamma_{\epsilon,R}$, we have
$[TA \ov T B]=A_T+B_{\ov T}$  in $\Pant_{200\epsilon,R}$ homology,  provided that $\len(\cd T \cd ),\len(\cd A \cd ), \len(\cd B \cd )>L$  and $I([\cd T \cd A \cd \ov T \cd B \cd ]) \le \Delta$.
\end{lemma}

\begin{proof} It follows from Lemma \ref{lemma-flip} that $[T\ov{A} \ov{T} B] \in \Gamma_{2\epsilon,R}$.  In order to apply Lemma \ref{lemma-A-T} we need an upper bound on $I(\cdot T \cdot \ov{A} \cdot \ov{T} \cdot)$ and a lower bound on  $2R-\len(\cdot A \cdot) - 2\len(\cdot T \cdot)$.
These follow from $\len(\cdot B \cdot), \len(\cdot T \cdot) \ge L$,  $I([\cd T \cd A \cd \ov T \cd B \cd ]) \le \Delta$,  and the Sum of Inefficiencies Lemma. 

We observe
\begin{align*} 
[TA \ov T B] &=\frac{1}{2} \big( [TA \ov T B]-[\ov B T \ov A \ov T ] \big) \\
&=\frac{1}{2} \big( [TA \ov T B]-[T \ov A \ov T \ov B ] \big) \\
&=\frac{1}{2} \big( [TA \ov T B]-[T \ov A \ov T B ] \big) \\
&+ \frac{1}{2} \big( [\ov{T} BT \ov A ]-[\ov T \ov B T \ov A  ] \big) \\
&=A_T+B_{\ov T},
\end{align*}
in $\Pant_{200\epsilon,R}$ homology.

\end{proof}

\begin{random} The randomization remarks for the Two-part Itemization Lemma. We have implicitly defined a map $g:G^{2} \to \R \Pant_{100\epsilon,R}$ such that $\partial g(A,B)=A_T+B_T-[A \ov T B T]$. The map $g$ is $Ke^{2\len(\cd T \cd)}$-semirandom with respect to $\Sigma_G \times \Sigma_G$ and $\sigma_{\Pant}$.
\end{random}

\begin{remark}  In fact it should be true that
\begin{equation}\label{most} 
[TA_1 \ov T B_1 ...TA_n \ov T B_n]=\sum\limits_{i=1}^{n} (A_i)_T+(B_i)_{\ov T},
\end{equation}
provided $\len(\cd T \cd )$ is large given $I(TA_i \ov T )$ and $I(\ov T B_i T)$. Above we  proved this when $n=1$ (provided $\len(\cd A \cd )$ and $\len(\cd B \cd )$ are large) and we will prove it
in the rest of this section for $n=2$, using the $ADCB$ lemma which we prove next. The general case can be proved by induction using the cases $n=1$ and $n=2$ (but we will only need this statement for $n=1,2$).

\end{remark}

\begin{remark}

\end{remark}

We also observe that under the usual conditions we have $A_{TU}=A_U$ in $\Pant_{100\epsilon,R}$ homology. This follows from the fact that
$2A_{TU}=[TUA \ov U \ov T B]+[TUA \ov U \ov T \ov B]=[UA \ov U \ov T B  T]+[UA \ov U \ov T \ov B T]=2A_{U}$.

\subsection{The $ADCB$ Lemma}  

\begin{claim} Let $\delta, \Delta>0$. There exists $L=L(\Delta,\delta)>0$ with the following properties. Let $A_i,B_i,T \in \pi_1(\Su,*)$, $i=0,1$.
If $\len(\cd T \cd )> L$ and $I(\cdot \ov{T} \cdot  A_i \cdot T \cdot)$  then
$$
\left| \len([A_0TB_0 \ov T A_1 T B_1 \ov T])-\sum\limits_{i=0}^{1}\len( \cd \ov T A_i T \cd )-\sum\limits_{i=0}^{1}\len( \cd  T B_i \ov T \cd )+4\len(\cd T \cd )\right|< \delta.
$$
\end{claim}

\begin{proof} By the Sum of Inefficiencies Lemma we have that
$$
I([\cdot A_0 \cdot B_0 \cdot \ov{T} \cdot A_1 \cdot T \cdot B_1 \cdot \ov{T} \cdot]) 
$$
is close to
$$
\sum_{i=0,1} I(\cdot \ov{T} \cdot A_i \cdot T \cdot) + I(\cdot T \cdot B_j \cdot \ov{T} \cdot).
$$
By the definition of inefficiency, the number 
$$
\len([A_0TB_0\ov T A_1 T B_1 \ov T]) -4\len(T) -\sum_{i=0,1}\big( \len(A_i)+\len(B_i) \big)
$$
$$
-\sum_{i=0,1} \left( \len(\cdot \ov T A_i T \cdot)+\len(\cdot T B_i \ov T \cdot ) -4 \len(T) - \len (A_i) - \len(B_i) \right)
$$
is small in absolute value. The claim now follows from the Sum of Inefficiencies Lemma.

\end{proof}

\begin{lemma}[ADCB Lemma]\label{lemma-ADCB} Let $\epsilon,\Delta>0$. There exists $L=L(\Su,\epsilon,\Delta)>0$ and $R_0=R_0(\Su,\epsilon,\Delta)>0$ with the following properties.
Let $A,B,C,D,T \in \pi_1(\Su,*)$ such that $\len(\cd B \cd ),\len( \cd D \cd ),\len(\cd T \cd )>L$. If $R>R_0$ and 
$$
I(\cd T \cd A \cd \ov T \cd ),I(\cd T \cd C \cd \ov T \cd ),I(\cd T \cd B \cd  \ov T \cd ),I( \cd  T \cd D \cd \ov T \cd ) \le \Delta
$$
then $[ATB \ov T C T D \ov T]=[ATD \ov T C T B \ov T]$ in $\Pant_{200\epsilon,R}$ homology provided that the curves in question are in $\Gamma_{\epsilon,R}$.
\end{lemma}

\begin{proof} Let $\left< X,Y \right> =[ATX \ov T C TY \ov T]$, for $X,Y \in \pi_1(\Su,*)$,
and let $\{ X, Y \} = \left< X, Y \right> - \left<Y, X\right>$ when both are in $\Gamma_{\epsilon, R}$. 
We claim that
\begin{equation}\label{e-ob}
\{ X, Y_0 \} = \{X, Y_1 \}
\end{equation}
in $\Pant_{100\epsilon,R}$ whenever $I(TX \ov T),I(TY_i \ov T) \le \Delta$, and the curves in question are in $\Gamma_{\epsilon,R}$.
To verify $(\ref{e-ob})$ we let $A_i=Y_i$, $B_0=ATX \ov T C$, $B_1=CT X\ov T A$, and $U=\ov T $ and $V=T$, where $A_i,B_i,U,V$ are from the statement of the Algebraic Square Lemma.
Since by rotation
\begin{align*}
\left< X,Y_0 \right>&=[Y_0 \ov T A T X \ov T C T] \\
\left< Y_0,X \right>&=[Y_0 \ov T C T X \ov T A T] \\
\left< X,Y_1 \right>&=[Y_1 \ov T A T X \ov T C T]  \\
\left<Y_1,X \right>&=[Y_1 \ov T C T X \ov T A T],
\end{align*}
the equation  $(\ref{e-ob})$ follows from the Algebraic Square Lemma (the hypotheses in the Algebraic Square Lemma follow from the hypothesis of this lemma and the Sum of Inefficiencies Lemma) . 
Likewise, 
\begin{equation}\label{e-ob1}
\{ X_0, Y\} = \{X_1, Y\}
\end{equation}
under the appropriate hypotheses. 

In order to prove the lemma we first suppose that $|\len(\cd TB \ov T \cd  )- \len ( \cd T D \ov T \cd ) |<\frac{\epsilon}{4}$. If $L$ is large enough  (and hence $\len(\cd B \cd )$,$\len(\cd D \cd )$ and $\len(\cd T \cd )$ are large enough),
it follows from the Connection Lemma that we can find a random geodesic arc 
$$
\cd E \cd \in \Conn_{\epsilon,\len(\cd TB \ov T \cd ) -2\len(\cd T \cd ) } (-i(\cd T \cd ),i (\cd T \cd ) ).
$$
For any such $E$ we have   $|\len(\cd TB \ov T \cd )- \len (\cd T E \ov T \cd  )|<\epsilon+O(\epsilon^{2})$. Therefore, by the previous Claim we have that the curves 
$\left< B,E\right>$, $\left< E,B\right>$, $\left< D,E\right>$, and $\left< E,D\right>$ are in $\Gamma_{2\epsilon,R}$.
Then from  $(\ref{e-ob})$ and (\ref{e-ob1}) it follows that  
$$
\{B, D\} = \{B, E\} = \{D, E\} = \{D, B\},
$$
and therefore $\{B, D\} = 0$, in $\Pi_{200\epsilon, R}$ homology. 

More generally, if $\len(\cd B \cd ),\len( \cd D \cd )>L$ let $k$ be the smallest integer such that
$$
k>4\frac{|\len(\cd TD \ov T \cd )- \len (\cd TB \ov T \cd )|}{\epsilon}.
$$
Set
$$
r_i=\frac{i}{2k} \len(\cd TD \ov T \cd )+\frac{2k-i}{2k} \len(\cd TB \ov T \cd )-2\len(\cd T \cd ).
$$
For $0 < i < 2k$ we take random $\cd E_i \cd  \in \Conn_{\epsilon,r_{i}}(t(\cd T \cd ), i( \ov T))$ (observe that $r_i>L-\Delta$), and we let $E_0 = D$ and $E_{2k} = B$. 
Then 
\begin{align*}
\{E_0, E_{2k}\} &= \{E_1, E_{2k}\}&&& \\
	&= \{E_1, E_{2k-1}\} &= \{E_2, E_{2k-1}\}& &\\
	&&= \{E_2, E_{2k-2}\} &= \{E_3, E_{2k-2}\}&\\
	&&& \ddots&\\
	&&&=\{E_{k-1}, E_{k+1}\} &= \{E_k, E_{k+1}\}\\
\end{align*}
and $\{E_k,E_{k+1}\}=0$ in $\Pant_{200\epsilon,R}$ homology by the first case so we are finished.
\end{proof}

\begin{random} The randomization remarks for the $ADCB$ Lemma. All constants $K$ may depend only on $\epsilon$, $\Su$ and $\Delta$. 
We have defined a map $g:G^{4} \to \R \Pant_{1,R}$ such that
$$
\partial g(A,B,C,D)=[ATB \ov T C T D \ov T]-[ATD \ov T C T B \ov T].
$$
In particular, we defined $h:G^5 \to \R \Pant_{1,R}$ so that
$$
\partial h (A,C,X,Y_0,Y_1)=\{X, Y_0\} - \{X, Y_1\}.
$$
This map $h$ is $e^{4\len(\cd T \cd )}K$-semirandom with respect to the measure classes $\Sigma^{\times 2}_G \times \Sigma_G \times \Sigma^{\boxtimes 2}_G$ and $\Sigma_{\Pant}$.

Then $g(A,B,C,D)$ is a sum of $2k$ terms of the form $\partial h(A,C,X,Y_0,Y_1)$, where each of $X,Y_0,Y_1$ is either $B$ or $D$, or $E_i$, which is a 
$K$-semirandom element of $G$ with respect to $\Sigma_G$. Moreover,  the $Y_i$ are always independent from $X$. Therefore, for each choice we make of $X,Y_0,Y_1$ (such as $X=E_i$, , $Y_0=E_{2k-i}$, $Y_1 = E_{2k-i+1}$ or $X=B$, and $Y_0=D$, $Y_1=E$) the map from $(A,B,C,D)$ to $(A,C,X,Y_0,Y_1)$ is $K$-semirandom with respect to $\Sigma^{\times 4}_G$ and $\Sigma_G^{\times 2} \times \Sigma_G \times \Sigma^{\boxtimes 2}_G $. Therefore, noting that $k<\frac{\lfloor 8R \rfloor }{\epsilon}$, we find that $g$ is $K Re^{4 \len(\cd T \cd)}$-semirandom, with respect to $\Sigma^{\times 4}_G $.
\end{random}

\begin{remark} This is a remark to the previous randomization remark. Where $B$ and $D$ are close in length, we can write $\{ B,D \}=\{B,B \}$ by (\ref{e-ob}), and hence $\{ B,D \}=0$. But we are letting $(X,Y_0,Y_1)$ be $(B,D,B)$, and the map $(B,D) \to (B,D,B)$ is not $1$-semirandom for $\Sigma^{\times 2}_G$ and $\Sigma_G \times \Sigma^{\boxtimes 2}_G$ (because $X$ and  $Y_1$ are not independent). This map is only $e^{\len(\cd B \cd )}$-semirandom, which is no good. It is for this reason that we introduce $E$. 
\end{remark}

The following lemma is a corollary of the $ADCB$ Lemma. We call it the Four-part Itemization Lemma.

\begin{lemma}[Four-part Itemization Lemma]\label{lemma-item} Let $\epsilon,\Delta>0$. There exists $L=L(\epsilon,\Delta)>0$ such that for any  $A,B,C,D,T \in \pi_1(\Su,*)$ we have
$$
[A\ov TB T C\ov TD T]-[\ov T \ov D T \ov C \ov T \ov B T \ov A]=2(A_{\ov T}+B_T+C_{\ov T}+D_T)
$$
in $\Pant_{200\epsilon,R}$ homology provided that $\len(\cd T \cd )>L$ and $I(\cd A \cd \ov T \cd B \cd T \cd C  \cd \ov T \cd D \cd T \cd )<\Delta$, and the curve $[A\ov TB T C\ov TD T]$ is in $\Gamma_{\epsilon,R}$.
\end{lemma}

\begin{proof}  Recall the remark after the statement of Lemma \ref{lemma-A-T}. We have
$$
[A\ov TB T C\ov TD T]-[\ov A \ov T B T C\ov TD T]=2A_{\ov T}
$$
$$
[\ov A \ov T B T C\ov TD T]-[\ov A \ov T \ov B T C \ov T D T]=2B_T
$$
$$
[\ov A \ov T \ov B T C \ov T D T]-[\ov A \ov T \ov B T \ov C \ov T D T]=2C_{\ov T}
$$
$$
[\ov A \ov T \ov B T \ov C \ov T D T]-[\ov A \ov T \ov B T \ov C \ov T \ov D T]=2D_T,
$$
in $\Pant_{100\epsilon,R}$ homology  (all the curves in question lie in $\Gamma_{2\epsilon,R}$ by the Flipping Lemma). So 
$$
[A\ov TB T C\ov TD T]-[\ov A \ov T \ov B T \ov C \ov T \ov D T]=2(A_{\ov T}+B_T +C_{\ov T}+D_T)
$$
in $\Pant_{100\epsilon,R}$ homology.  But 
$$
[\ov A \ov T \ov B T \ov C \ov T \ov D T]-[\ov A \ov T \ov D T \ov C \ov T \ov B T]=0
$$
in $\Pant_{200\epsilon,R}$ homology by the $ADCB$ Lemma so we are finished.

\end{proof}

\begin{random} The randomization remark for the Four-part Itemization Lemma. We have defined $g:G^{4} \to \R \Pant_{1,R}$ such that
$$
\partial g(A,B,C,D)=[ATB \ov T CTD \ov T]-(A_{\ov T}+B_T+C_{\ov C}+D_T).
$$
This map is $KRe^{4\len(\cd T \cd )}$-semirandom with respect to  $\Sigma^{\times 4}_{G}$ and $\sigma_{\Pant}$, for some $K=K(\epsilon,\Su)$.

\end{random}

\section{The $XY$ Theorem}  In this section we prove the $XY$ Theorem which states that 
$$
(XY)_T=X_T+Y_T
$$
for suitable $X,Y$ and $T$. 

The $XY$ Theorem will be the central identity in the last section of the paper;  it will allow us to reduce the encoding of long elements of $\pi_1(S,*)$ to encoding of the generators. To prove the $XY$ Theorem we will first prove two related statements called the First and the Second Rotation Lemmas. These are in turn proven with the Four-Part Itemization Lemma and the estimates from the Theory of Inefficiency.

\subsection{The Rotation Lemmas}

Let $X,Y,Z \in \pi_1(\Su,*)$. Then we have the three geodesic arcs $\cd X \cd$, $\cd Y \cd $, and $\cd Z \cd $. Consider the union of these three geodesic arcs as a $\theta$-graph on the surface $\Su$. This $\theta$-graph generates an immersed pair of pants in $\Su$ if and only if the triples of unit vectors $i(\cd X \cd ),i(\cd Y \cd ),i(\cd Z \cd )$ and $t(\cd X \cd ),t(\cd Y \cd ),t(\cd Z \cd )$, have the opposite cyclic orderings.

The following is the First Rotation Lemma.

\begin{lemma} [First Rotation Lemma] \label{lemma-rotation-1} Let $\epsilon,\Delta>0$. There exists $K=K(\epsilon,\Delta)>0$ with the following properties.
Let $R_i,S_i,T \in \pi_1(\Su,*)$, $i=0,1,2$, such that  
\begin{enumerate}
\item $I(\cd T \cd R_i \cd \ov R_{i+1} \cd \ov T \cd ), I(\cd T \cd S_i \cd \ov S_{i+1} \cd \ov T \cd )<\Delta$, 
\item $\len(\cd T \cd ) \ge K$,
\item $\len(\cd R_i \cd )+\len(\cd S_i \cd )+2\len(\cd T \cd )<R-K$,
\item The triples of vectors $\big(t(\cd T  R_i \cd ) \big)$ and $\big(t(\cd T  S_i \cd ) \big)$, $i=0,1,2$, have opposite cyclic ordering in $\TB_{*} \Su$ (one of them is clockwise and the other one anti-clockwise).
\end{enumerate}
Then 
\begin{equation}\label{equat-1}
\sum\limits_{i=0}^{2}(R_{i+1} \ov{R}_{i})_{T}+\sum\limits_{i=0}^{2}(S_i \ov{S}_{i+1})_{T}=0,
\end{equation}
in $\Pant_{300\epsilon,R}$ homology.

\end{lemma}

\begin{remark} It follows by relabeling that if $\max(\len(\cd R_i \cd ))+\max(\len(\cd S_i \cd ))+2\len(\cd T \cd )<R-K$ and  the triples of vectors $\big(t(\cd T  R_i \cd ) \big)$ and $\big(t(\cd T  S_i \cd ) \big)$, $i=0,1,2$, have the same cyclic ordering in $\TB_{*} \Su$, then 
\begin{equation}\label{equat-2}
\sum\limits_{i=0}^{2}(R_{i} \ov{R}_{i+1})_{T}+\sum\limits_{i=0}^{2}(S_i \ov{S}_{i+1})_{T}=0,
\end{equation}
in $\Pant_{300\epsilon,R}$ homology.

\end{remark}

\begin{proof}  Let $r_i \ge 0$, $i=0,1,2$,  be the solutions of the equations
\begin{equation}\label{equ-con}
r_i+r_{i+1}=2R-\len(\cd T R_{i+1} \ov{R}_{i} \ov T \cd )-\len( \cd T S_i \ov{S}_{i+1} \ov T \cd ). 
\end{equation}

Then we let $A_i$ be a random element of $\Conn_{\epsilon,r_{i}}(-i( \cd T \cd ),i( \cd T \cd ))$. 

Consider the three elements $\ov{R}_{i} \ov T A_i T S_i$ of $\pi_1(\Su,*)$ and the corresponding geodesic arcs  $\cd \ov{R}_{i} \ov T A_i T S_i \cd$. We will show that the corresponding $\theta$-graph
generates an immersed  pair of pants $\Pi_{A}$ in $\Su$. The three cuffs of $\Pi_A$ are the closed curves $[\ov{R}_{i+1} \ov T A_{i+1} T S_{i+1} \ov{S}_{i} \ov T \ov{A}_{i} T R_{i} ]$. We will also show that these closed geodesics have  length $3\epsilon$ close to $2R$, which implies that $\Pi_{A} \in \Pant_{3\epsilon,R}$.

\begin{figure}
  %\centering
  {
    
    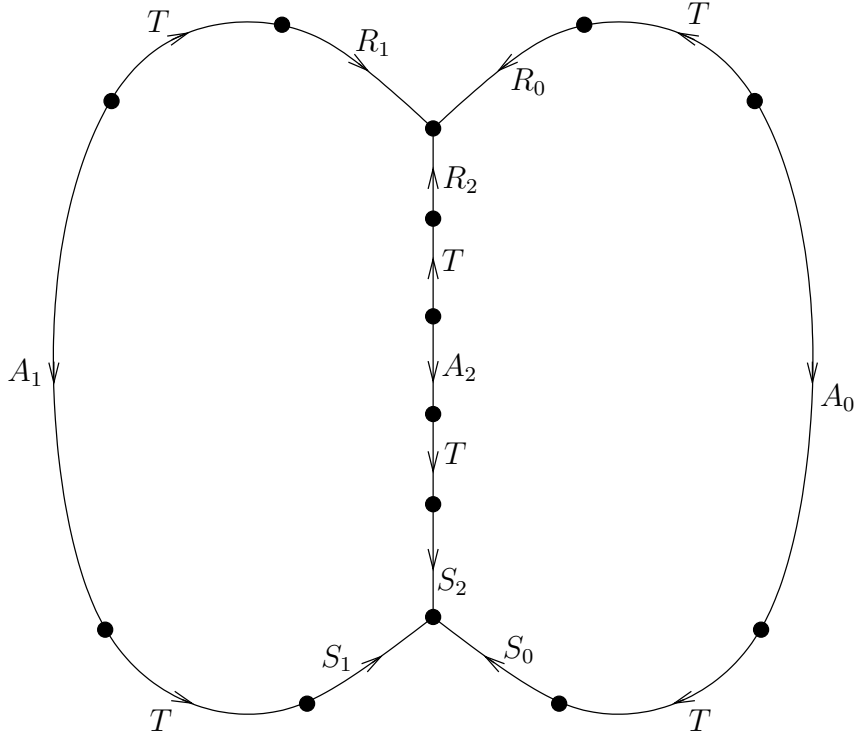
  }
  \caption{The Rotation Lemma}
  \label{fig-rotation}
\end{figure}

We finish the argument as follows. Taking the boundary of $\Pi_A$, we obtain
$$
\sum\limits_{i=0}^{2}  [\ov{R}_{i} \ov T A_{i} T S_{i} \ov{S}_{i+1} \ov T \ov{A}_{i+1} T R_{i+1} ]      =0
$$
in $\Pant_{3\epsilon,R}$ homology. Applying the Four-part Itemization Lemma we find

\begin{align*}
0 &= \sum_{i=0}^{2}  [R_{i+1} \ov{R}_{i} \ov T A_{i} T S_{i} \ov{S}_{i+1} \ov T \ov{A}_{i+1} T  ] \\
&= \sum_{i=0}^{2} \left( (R_{i+1} \ov{R}_{i})_{T}+ (A_{i})_{\ov T}+ (S_{i} \ov{S}_{i+1})_{T}+(\ov{A}_{i+1})_{\ov T}  \right)  \\
&= \sum_{i=0}^{2}(R_i \ov{R}_{i+1})_{T}+\sum_{i=0}^{2}(S_i \ov{S}_{i+1})_{T},
\end{align*}
in $\Pant_{300\epsilon,R}$, because $(A_i)_{\ov T}=-(\ov{A}_i)_{\ov T}$.

We now verify that  $ [\ov{R}_{i} \ov T A_{i} T S_{i} \ov{S}_{i+1} \ov T \ov{A}_{i+1} T R_{i+1} ]  \in \Gamma_{2\epsilon,R}$. By the New Angle Lemma (Lemma \ref{lemma-new-angle}), applied to $\beta=\cdot  \ov T \cdot$ and $\alpha=\cdot T \cdot R_{i} \cdot$,  
for $K$ large enough (and therefore $\len(\cd T \cd )$ large) the angle  $\Ang \big( i(\cd T \cd ),i(\cd TR_{i+1} \ov{R}_{i} \ov T \cd ) \big) \le \frac{\epsilon}{10}$, and likewise 
$\Ang \big( t(\cd \ov T \cd ), t(\cd TR_{i+1} \ov{R}_{i} \ov T \cd ) \big) \le \frac{\epsilon}{10}$, and for the same with $R_i$ replaced with $S_i$. It follows that
$\Ang \big( t(\cd \ov{A}_{i+1} \cd ),i(\cd TR_{i+1} \ov{R}_{i} T \cd ) \big)<2 \epsilon$, and so on, so by the Sum of Inefficiencies Lemma (using equation 
(\ref{angle-ineff}))

\begin{multline*}
\big| \len( [\ov{R}_{i} \ov T A_{i} T S_{i} \ov{S}_{i+1} \ov T \ov{A}_{i+1} T R_{i+1} ]  ) - \len(\cd A_{i} \cd ) -\len(\cd T R_{i+1} \ov{R}_{i} \ov T \cd ) -\len(\cd A_{i+1} \cd)- \\
-\len(\cd T S_i \ov{S}_{i+1} \ov T \cd  ) \big|<O(\epsilon ^{2}),
\end{multline*}

and moreover by (\ref{equ-con}) we have 
$$
\left|\len(\cd A_i \cd ) +\len(\cd TR_{i+1} \ov{R}_{i} \ov T \cd ) +\len(\cd A_{i+1} \cd )
+\len( \cd T S_i \ov{S}_{i+1} \ov T \cd ) -2R \right|<2 \epsilon,
$$
which proves the claim.

We now verify that the $\theta$-graph associated to the geodesic arcs $\cd \ov{R}_{i} \ov T A_i T S_i \cd$ generates an immersed  pair of pants $\Pi_{A}$ in $\Su$.
We find the unique $\theta_0 \in [0,\pi]$ such that $I(\pi-\theta_0)=\Delta+1$ (that is $\pi-\theta_0=2\sec^{-1}(e^{\frac{\Delta+1}{2}}) $). Observe that $I(\cd \ov{R}_i \cd \ov{T} \cd )<
I(\cd  T \cd R_{i+1} \cd \ov{R}_i \cd \ov T \cd ) \le \Delta$. Then $\len(\cd \ov{R}_i  \ov T \cd )>\len(\cd T \cd )-\Delta>K-\Delta$.  
By the Sum of Inefficiencies Lemma for Angles 
\begin{align*}
I(\pi- \Ang \big (i(\cd \ov{R}_i \ov T \cd ) , i(\cd \ov{R}_{i+1} \ov T  \cd ) \big))-1 &<I(\cd T R_{i+1} \cd \ov{R}_i \ov T \cd ) \\ 
& \le  I(\cd T \cd R_{i+1} \cd \ov{R}_i \cd \ov T \cd ) < \Delta.
\end{align*}
Therefore $\Ang \big (i(\cd \ov{R}_i \ov T \cd ),i(\cd \ov{R}_{i+1} \ov T  \cd ) )>\theta_0 $.

On the other hand, by the New Angle Lemma, because the geodesic arc  $\cd \ov{R}_i \ov T \cd$ is long for large enough $K$ (we showed above that $\len(\cd \ov{R}_i \ov T \cd )>K-\Delta$), we have
$\Ang \big( i(\cd \ov{R}_i \ov T \cd ), i(\cd \ov{R}_i \ov T A_i T S_i \cd ) \big)<\frac{\theta_{0}}{2}$, so the cyclic order of the triple of vectors  $i(\cd \ov{R}_i \ov T A_i T S_i \cd )$, $i=0,1,2$,
is the same as of the triple of vectors  $i(\cd \ov{R}_i \ov T \cd )$, and likewise the cyclic order of the triple of vectors  $t(\cd \ov{R}_i \ov T A_i T S_i \cd )$, $i=0,1,2$,
is the same as of the triple of vectors  $t(\cd T S_i  \cd )$. So the corresponding cyclic orderings are opposed and we are finished.

\end{proof}

\begin{random} The randomization remark for the First Rotation Lemma. We let $K=K(\epsilon,\Su)$. 
We have defined $g:G^{6} \to \R \Pant_{1,R}$ such that
$$
\partial g (R_0,R_1,R_2,S_0,S_1,S_2)=\sum_{i=0}^{2}(R_{i+1}\ov{R}_i)_T+(S_i\ov{S}_{i+1})_T.
$$
Let $\Pi$ denote the pants whose $\theta$-graph is made out of the three connections $\cd \ov{R}_i \ov T A_i T S_i \ov{S}_{i+1} \ov T A_{i+1} T R_{i+1} \cd$, $i=0,1,2$.
We can write 
$$
g((R_i),(S_i))=\Pi+\sum_{i-0}^{2}g_1(  R_{i+1} \ov{R}_i,A_i,S_i\ov{S}_{i+1},\ov{A}_{i+1}) ,
$$
where $g_1$ is the map from the $Four-part-Itemization$ Lemma (see the randomization remark). So $g$ is $K(e^{12\len(\cd T \cd)} +Re^{4 \len(\cd T \cd )})$-semirandom with respect to 
$\Sigma^{\times 6}_G $ and $\Sigma_{\Pant}$.

\end{random}

The Second Rotation Lemma is:

\begin{lemma}[Second Rotation Lemma ]\label{lemma-rotation-2} Let $\epsilon,\Delta>0$. There exists $K=K(\epsilon,\Delta)>0$ with the following properties.
Let $R_i,T \in \pi_1(\Su,*)$, $i=0,1,2$, such that  
\begin{enumerate}
\item $I(\cd T \cd R_i \cd \ov R_{i+1} \cd \ov T \cd )<\Delta$, 
\item $\len(\cd T \cd ) \ge K$.
\end{enumerate}
Then 
\begin{equation}\label{equat-3}
\sum\limits_{i=0}^{2}(R_i \ov{R}_{i+1})_{T}=0,
\end{equation}
in $\Pant_{300\epsilon,R}$ homology.

\end{lemma}

\begin{proof}  Given $T$; We choose $v \in \TB_{*} \Su$ and let $\rho=e^{\frac{2\pi i}{3} }$. We take $L$ sufficiently large so that $\Conn_{\epsilon,L}(t(\cd T \cd ),\rho^{i} v)$ is non-empty, for $i=0,1,2$. Then we choose $\cd S_i \cd  \in \Conn_{\epsilon,L}(t(\cd T \cd ),\rho^{i} v)$. Then $I(\cd T \cd S_i \cd \ov{S}_{i+1} \cd \ov T \cd ) \le \log \frac{4}{3} +O(\epsilon) \le 1$, by the Sum of Inefficiencies for Angles Lemma, so when $\len(\cd T \cd )$ is large we can apply the previous Lemma (see the Remark after  Lemma \ref{lemma-rotation-1}) with $R_i : = S_i$  to obtain
$$
2\sum_{i=0}^{2}(S_i \ov{S}_{i+1})_{T}=0
$$
in $\Pant_{300\epsilon,R}$ homology. 

Then given $R_i$ as in the hypothesis to this lemma, we obtain

$$
\sum_{i=0}^{2}(S_i \ov{S}_{i+1})_{T} +(R_i \ov R_{i+1})_{T}=0,
$$
so 

$$
\sum_{i=0}^{2}(R_i \ov R_{i+1})_{T}=0.
$$

\end{proof}

\begin{random} The randomization remark for the Second Rotation Lemma. All constants $K$ may only depend on $\epsilon$ and $\Su$. We have defined $g:G^{3} \to \R \Pant_{1,R}$ such that
$$
\partial g (R_0,R_1,R_2)=\sum_{i=0}^{2}(R_iR_{i+1})_T.
$$

We are fixing $S_0,S_1,S_2$ of length $L$, so the triple $(S_0,S_1,S_2)$ is $e^{3L}$-semirandom, and the maps $(R_0,R_1,R_2) \to (R_0,R_1,R_2,S_0,S_1,S_2)$ and $(R_0,R_1,R_2) \to (S_0,S_1,S_2,S_0,S_1,S_2)$ are $e^{3L}$ and $e^{6L}$ semirandom respectively.

Then letting $g_{1R}$ be the $g$ for the First Rotation Lemma, we can letter
$$
g(R_0,R_1,R_2)=g_{1R}\big((R_i),(S_i)\big)-\frac{1}{2}g_{1R}\big((S_i),(S_i)\big)
$$
and $g_{1R}$ is $KRe^{12\len(\cd T \cd)}$ semirandom, so $g$ is $KRe^{6L+12\len(\cd T\cd)}$ semirandom.

\end{random}

\subsection{The $XY$ Theorem}

The following theorem follows from  the Second Rotation Lemma. We call it the $XY$ Theorem.

\begin{theorem}[XY Theorem]\label{thm-XY} Let $\epsilon,\Delta>0$. There exists $K=K(\epsilon,\Delta)>0$ with the following properties.
Let $X,Y,T \in \pi_1(\Su,*)$, $i=0,1,2$, such that  
\begin{enumerate}
\item $I(\cd T \cd X \cd Y \cd \ov T \cd ),I(\cd T \cd X \cd \ov T \cd ),I(\cd T \cd Y \cd \ov T \cd ) <\Delta$, 
\item $\len(\cd T \cd ) \ge K$.
\end{enumerate}
Then $(XY)_T=X_T+Y_T$ in $\Pant_{300\epsilon,R}$ homology.

\end{theorem}

\begin{proof} Set $R_0=\id$, $R_1=X$, and $R_2=\ov Y$, and apply the previous lemma.

\end{proof}

\begin{random} We have defined the map $g_{XY}:(X,Y) \to \R \Pant_{1,R}$ such that $\partial g_{XY} (X,Y)=(XY)_T-X_T-Y_T$ (the map $g_{XY}$ is defined on the appropriate subset of $\Sigma^{2}_G$ described in the statement of Theorem \ref{thm-XY}).  This map is $RKe^{12\len(\cd T \cd )}$-semirandom with respect to 
$\Sigma^{\times 2}_G$ and $\sigma_{\Pant}$, where $K=K(\epsilon,\Delta)$. 

\end{random}

\section{The Endgame} In this last section of the main text of the paper, we prove that every good curve is good pants homologous to a sum of encodings of a given set of standard generators for $\pi_1(S,*)$. 

We prove this in three steps:

\noindent
1. We prove in Lemma \ref{lemma-cut-1} that every good curve is good pants homologous to a sum of two encodings of two elements of $\pi_1(S,*)$; these elements are represented by geodesic segments of length about $R$.

\noindent
2. We use the $XY$ Theorem and Lemma \ref{lemma-division} to repeatedly write $X_T=(X_1)_T+(X_2)_T$ where $X_1$ and $X_2$ have length about half that of $X$. This allows us to reduce an encoding of an arbitrary element of $\pi_1(S,*)$ to a sum of encodings of elements of bounded length (bounded in terms of $S$ and $\epsilon$)

\noindent
3. We  use the $XY$ Theorem to reduce the encoding of an element of $\pi_1(S,*)$ of bounded length to a sum of encodings of generators. This requires the proper choice of $T$, which is discussed in Lemma \ref{lemma-direction}.

\subsection{ The good pants homology of short words} The following is the Good Direction Lemma.

\begin{lemma}[Good Direction Lemma]\label{lemma-direction}
For any finite set $W \subset \pi_1(\Su,*)$, we can find $\Delta=\Delta(\Su,W)$, such that for any 
$L$ we can find $T \in \pi_1(\Su,*)$ such that $\len (\cd T \cd )>L$ and $I(\cd T \cd X \cd \ov T \cd )<\Delta$, when $X \in W$.
\end{lemma}

\begin{proof} For any $v \in \TB_{*} \Su$, and $t>0$, we let $\alpha_t(v)$ be the geodesic segment of length $t$ such that $i(\alpha_t(v))=v$, and we let $\alpha_{\infty}(v)$ be the corresponding infinite geodesic ray. We claim that for any $X \in \pi_1(\Su,*)$, and $X \ne id$, there are at most two $v \in \TB_{*} \Su$ such that 
\begin{equation}\label{eq-v}
\lim\limits_{t \to \infty} I(\alpha^{-1}_{t}(v)\cd X \cd \alpha_t(v))=\infty.
\end{equation}

To prove the claim we  lift $\cd X \cd $ to the universal cover $\Ha$, and thus get two lifts of $*$, and hence two lifts of $v$. 
We observe that  $(\ref{eq-v})$ holds if and only if the two lifts of $\alpha_{\infty}(v)$ end at the same point of $\partial{\Ha}$. 
The map that maps one lift of $\alpha_{\infty}(v)$ to the other is the deck transformation that maps one lift of $*$ to the other.
The relation  $(\ref{eq-v})$ holds if and only if $\alpha_{\infty}(v)$  is  a fixed point of the  M\"obius transformation $M$, and since $M$ is not the identity
this can be true for at most two vectors $v$.

\end{proof}

In the remainder of this section  we fix a set of standard generators $g_1,..g_{2n}$ of $\pi_1(\Su,*)$ (here $n$ is the genus of $\Su$). Recall that ${\bf H}_1$ denotes the standard homology on $\Su$.
Let $[g_i]$ denote the corresponding closed curves. For any closed curve $\gamma \subset \Su$ there are unique $a_1,...a_{2n}$ such that $\gamma=\sum a_i [g_i]$ in ${\bf H}_1$.
We define $q:\Gamma \to \R \pi_1(\Su,*)$ by $q(\gamma)=\sum a_i g_i$, where $\Gamma$ is the set of all closed curves on $\Su$.
We extend the definition of $q$ to a map  $q:\pi_1(\Su,*) \to \R \{ g_1,..,g_{2n} \}$ by $q(X)=q([X])$.

For $l \in \N$, we define the set $W_l$ as the set of elements $X \in \pi_1(\Su,*)$ that can be written as a product of at most $l$ generators (or their inverses).

\begin{theorem}\label{thm-short} Let $\epsilon>0$. For all $l \in \N$, and $L>0$, we can find $T \in \pi_1(\Su,*)$ and $R_0$ such that $\len(\cd T \cd )>L$, and for $R>R_0$, and $X \in W_l$, we have
$$
X_T=(q(X))_T
$$
in $\Pant_{300\epsilon,R}$ homology.
\end{theorem}

\begin{remark} Here we extended the partial map $(\cd )_T:\pi_1(\Su,*) \to \R \Gamma_{\epsilon,R}$ (given by $X \mapsto X_T$) to a partial map  $(\cd )_T:\R \pi_1(\Su,*) \to \R \Gamma_{\epsilon,R}$.
We remind the reader that $X_T$ depends implicitly on $R$ and $\epsilon$.
\end{remark}

\begin{proof} We take $\Delta=\Delta(W_l)$ and $T=T(W_l,L)$ from the previous lemma, so $\len(\cd T \cd )>L$ and $I(\cd T \cd X \cd \ov T)< \Delta$, for all $X \in W_l$. If $X \in W_1$, then $q(X)=X$
or $q(X)=- \ov X$, so $X_T=(q(X))_T$.

Take $1 \le k<l$, and assume $X_T=((q(X))_T$ in $\Pant_{300\epsilon,R}$ homology for all $X \in W_{k}$. Then for any $X \in W_{k+1}$ we can write $X=g_{i}^{\sigma}Y$, for some $i \in \{1,..,2n \}$, and  $\sigma=\pm 1$, and $Y \in W_{k}$. Then $X_T=(g^{\sigma}_{i})_T +Y_T$ by the $XY$ Theorem (see Theorem \ref{thm-XY}) which requires 

$$
I(\cd T \cd X \cd \ov T \cd ), I(\cd T \cd g^{\sigma}_i \cd \ov T \cd ), I(\cd T \cd Y \cd \ov T \cd )<\Delta,
$$
and $Y_T=(q(Y))_T$ by assumption, so $X_T=((q(X))_T$. We conclude the theorem by induction.

\end{proof}

\begin{random} The Randomization remarks for Theorem \ref{thm-short}. Given $l,L,T$ and $R$ (and $\epsilon$) we have implicitly defined the map
$g_W:W_l \to \Pant_{300\epsilon,R}$ such that $\partial g_W(X)=X_T-(q(X))_T$. The map $g_W$ arises from a sum of at most $l$ applications of the $XY$ Theorem
so $g_W$ is $K(\Su)RKe^{12 \len(\cd T \cd )}$-semirandom, because every measure in $\Sigma_G$ has total mass at most $K(\Su)$.
\end{random}

\subsection{Preliminary lemmas}
We now observe that every good curve is good pants homologous to $(X_0)_T+(X_1)_T$ for suitable $X_0$ and $X_1$ from $\pi_1(\Su,*)$.

\begin{lemma}\label{lemma-cut-1} There exists a universal constant $\wh{\epsilon}>0$ such that for every  $0<\epsilon <\wh{\epsilon}$, there exist  constants $L=L(\epsilon,\Su)>0$ and $R_0=R_0(\epsilon,\Su)>0$, with the following properties.  For any $\gamma \in \Gamma_{\epsilon,R}$ and $T \in \pi_1(\Su,*)$, $\len(\cd T \cd )>L$, we can find $X_0,X_1 \in \pi_1(\Su,*)$ such that 
\begin{enumerate}
\item $|\len(\cd X_i \cd ) -(R+2L- \log 4 )|<\frac{1}{2}$, 
\item  $ \Ang(t(\cd T \cd ), i(\cd X_i \cd ) ), \, \Ang(t(\cd X_i \cd ), i(\cd \ov T  \cd ) ) \le \frac{\pi}{6}$,
\item  $\gamma=(X_0)_T+(X_1)_T$ in $\Pant_{300\epsilon,R}$ homology, 
\end{enumerate}
for $R>R_0$.
\end{lemma}

\begin{proof} We take at random two points $x_0$ and $x_1$ on the parameterizing torus $\TT_{\gamma}$ that are $\hl(\gamma)$ apart and we let $w_i \in \TB_{x_{i}} \Su$ be $-\sqrt{-1}\gamma'(x_i)$.
We let $\gamma_i$ be the subsegment of $\gamma$ from $x_i$ to $x_{i+1}$ (where $x_2=x_0$). 

For $i=0,1$ we take $\alpha_i \in \Conn_{ \frac{\epsilon}{10},L}(t(\cd T \cd ),w_i)$, where $L=L(\epsilon,\Su)$ is the constant from the Connection Lemma (that is, we choose $L$ so that  the set 
$\Conn_{ \frac{\epsilon}{10},L}(t(\cd T \cd ),w_i)$ is non-empty).  Observe that the piecewise geodesic arc $\alpha_0 \gamma_0 \alpha^{-1}_1$ begins and ends at the point $*$, so we let $X_0 \in \pi_1(\Su,*)$ denote the corresponding element of $\pi_1(\Su,*)$. Similarly we let $X_1 \in \pi_1(\Su,*)$ be the element that corresponds to the curve   $\alpha_1 \gamma_1 \alpha^{-1}_0$.

It follows from the Remark after Lemma \ref{lemma-ineff+} that the inequality (1) of the statement of the lemma holds. On the other hand,  by the New Angle Lemma the angle $\Ang(i(\cd X_0 \cd ),i(\alpha_0))$ is as small as we want providing that $\len(\alpha_0)>L$ is large enough (here we use that the inefficiency $I(\alpha_0 \gamma_0 \alpha^{-1}_1)$ is bounded above). Since by construction the angle $\Ang(i(\alpha_0),t(\cd T \cd ))$ is less than $\frac{\epsilon}{10}$ we conclude that for $L$ large enough we have $\Ang(t(\cd T \cd ), i(\cd X_0 \cd ) )<\frac{\pi}{6}$. Other cases are treated similarly.

Let $\cd A \cd $ be a random element of  $\Conn_{ \frac{\epsilon}{10},R'}(-i(\cd T \cd ),i(\cd T \cd ))$, where $R'=R+\log 4 -2L-2\len(\cd T \cd )$. Then
$$
\left| \len([X_0 \ov T \ov A T ])-2R \right|<\epsilon,
$$
$$
\left| \len([X_1 \ov T A T])-2R \right| <\epsilon,
$$
so $\gamma=[X_0 \ov T \ov A T ]+[X_1 \ov T A T]$ in $\Pant_{\epsilon,R}$ homology, because the three curves bound a good pair of pants.

Moreover, $[X_0 \ov T \ov A T ]=(X_0)_T+ (\ov A)_{\ov T}$, and $[X_1 \ov T A T]=(X_1)_T+A_{\ov T}$ in $\Pant_{100\epsilon,R}$ homology by the Two-part Itemization Lemma. 
Since $(\ov A)_{\ov T}=-A_{\ov T}$ we conclude $\gamma=(X_0)_T+(X_1)_T$ in $\Pant_{300\epsilon,R}$ homology.

\end{proof}

\begin{random} The randomization remarks for Lemma \ref{lemma-cut-1}. We have defined the maps $q_C:\Gamma_{1,R} \to \R G$ (by $q_C(\gamma)=X_0+X_1$) and $g_C:\Gamma_{1,R}  \to \R \Pant_{1,R}$, such that $\partial g_C(\gamma)=\gamma-(q_C(\gamma))_T$ (where $A \to A_T$ maps $\R G \to \R \Gamma_{1,R}$).

The map $q_C$ is $e^{L}K$-semirandom   with respect to  $\sigma_{\Gamma}$ and $\Sigma_G$. The map $g_C$ is $e^{2\len(\cd T \cd )}K(\Su,\epsilon)$ semirandom with respect to $\sigma_{\Gamma}$ and $\sigma_{\Pant}$, where  $K=K(\Su,\epsilon)$.

\end{random}

We have the following definition.  For any $X,T \in \pi_1(\Su,*)$, $X \ne \id$, we let
$$
\theta^{T}_{X}= \max \{ \Ang(t(\cd T \cd ), i(\cd X \cd ) ),  \Ang(t(\cd X \cd ), i(\cd \ov T  \cd ) ) \}.
$$

\begin{lemma}\label{lemma-division} For $L>L_0(\Su)$, and   $X,T \in \pi_1(\Su,*)$, $X \ne \id$, 
then we can write $X=X_0 X_1$, for some $X_0,X_1 \in \pi_1(\Su,*)$, such that

\begin{enumerate}

\item $\left| \len(\cd X_i \cd ) - (\frac{\len(\cd X \cd )}{2}+L- \log 2) \right|<\frac{1}{2}$,

\item $I(\cd X_0 \cd X_1 \cd ) \le 2L+3$,

\item  $\theta^{T}_{X_{i}} \le \max \{\theta^{T}_{X}+e^{L+4} e^{-\len(\cd X_{i} \cd )},\frac{\pi}{6} \}$.

\end{enumerate}

\end{lemma}

\begin{proof} We let $\alpha=\cd X \cd $, then $\alpha:[0,\len(\cd X \cd )] \to \Su$ is the unit speed parametrization with $\alpha(0)=\alpha(\len(\cd X \cd ))=*$. We let $y=\frac{\len(\cd X \cd )}{2}$. Then for $L$ large enough, we can find
$\beta \in \Conn_{\frac{1}{20},L}(t(\cd T \cd ), \sqrt{-1} \alpha'(y))$ (as always, $L$ is determined by the Connection Lemma).

Then $\alpha[0,y]\beta^{-1}$ begins and ends at $*$, so it represents some $X_0 \in \pi_1(\Su,*)$.  Likewise $\beta\alpha[y,\len(\cd X \cd )]$ represents some $X_1 \in \pi_1(\Su,*)$, and $X=X_0X_1$.
Moreover, it follows from the Remark after Lemma  \ref{lemma-ineff+} that
$$
\left| \len(\cd X_i \cd ) - \left( \frac{\len(\cd X \cd )}{2}+L- \log 2 \right) \right|<\frac{1}{2}.
$$
The condition (2) follows immediately from (1).

Let $\theta=\Ang(i(\cd X \cd ), i(\cd X_0 \cd ) )$. Then by the hyperbolic law of sines, assuming that $\len(\cd X_i \cd ) \ge 1$ (which follows if we assume that $\len(\alpha) \ge L-1$ is at least 1) 
we obtain
$$
\sin(\theta) \le \frac{\sinh(L+1)}{\sinh(\len(\cd X_0 \cd ) )} \le e^{L+2-\len(\cd X_{0} \cd ) } ,
$$
Therefore
$$
\Ang(t(\cd T \cd ), i(\cd X_0 \cd ) ) \le \Ang(t(\cd T \cd ), i(\cd X \cd ) ) +e^{L+4-\len(\cd X_{0} \cd ) }.
$$

By similar reasoning we find that $\Ang(t(\cd X_0 \cd ),-i(\beta)) \le e^{2-L}<\frac{\pi}{12}$, assuming that $L$ is large enough. Also by construction 
$\Ang(-i(\beta),t( \cd T \cd ))<\frac{1}{20}<\frac{\pi}{12}$, so $\Ang(t(\cd X_0 \cd ),i(\cd \ov T \cd ) ) \le \frac{\pi}{6}$. We proceed similarly for $X_1$.

\end{proof}

\begin{random} The randomization remarks for Lemma \ref{lemma-division}. We have defined $\wh{q}_D:G \to G^{2}$ such that  $\wh{q}_D(X)=(X_0,X_1)$. 
If $\len(\cd X \cd ) \in [a,a+1]$, then  $\len(\cd X_0 \cd ),\len(\cd X_1 \cd ) \in [\frac{a}{2}+L',\frac{a}{2}+L'+1]$, where $L'=L-\log 2 -\frac{1}{2}$. 

Moreover, given $(X_0,X_1) \in G^2$ there is at most one $X$ such that $\wh{q}_D(X)=(X_0,X_1)$ (because $X=X_0X_1$). We conclude that 
$$
(\wh{q}_D)_{*} \sigma_a \le e^{2L'+2} \sigma_{ \frac{a}{2} +L'} \times \sigma_{ \frac{a}{2} +L'}, 
$$
and hence $\wh{q}_D$ is $e^{2L'+2}$-semirandom. It follows that the map $q_D:G \to \R G$ defined by $X \to X_0+X_1$, is  $2e^{2L'+2}$-semirandom.

\end{random}

\subsection{Proof of Theorem \ref{thm-correction-1}} The following theorem implies Theorem \ref{thm-correction-1}. 
Recall  $\{g_1, ...,g_{2n} \}$ denotes a standard basis for $\pi_1 (\Su,*)$, where $n$ is the genus of $\Su$.

\begin{theorem} \label{thm-correction-2} Let $\epsilon > 0$. There exists $R_0=R_0(\Su,\epsilon) > 0$  with the following properties. 
There exists $T \in \pi_1(\Su,*)$, where $T$ depends only on $\epsilon$ and $\Su$, such that  for every $R > R_0$ and every $\gamma \in \Gamma_{\epsilon,R}$ we have
$$
\gamma= \sum_{i=1}^{2g} a_i (g_i)_{T},
$$
in $\Pant_{300\epsilon,R}$ homology, for some $a_i \in \Q$.
\end{theorem}

\begin{remark} To prove Theorem \ref{thm-correction-1} we take $h_i=(g_i)_T$. Since $(g_i)_T$ is equal to the closed curve on $\Su$  that corresponds to $g_i$ in the standard homology ${\bf H}_1$, it follows that $h_i$ is a basis for   ${\bf H}_1$ (with rational coefficients).
\end{remark}

\begin{proof} We take $L$ that is sufficiently large for Lemma \ref{lemma-cut-1} and Lemma \ref{lemma-division}. We let $l \in \N$ be such that $X \in W_l$ whenever $\len(\cd X \cd )<2L+5$.
Then by Theorem \ref{thm-short} we can find $T$ such that $\len(\cd T \cd )>L$ and $\len(\cd T \cd )> K(\epsilon,2L+3)$, where $K(\epsilon,\Delta)$ is the constant from Theorem \ref{thm-XY}, and such that $X_T=(q(X))_T$ in $\Pant_{300\epsilon,R}$ homology for all $X \in W_l$. We take $R>R_0(\Su,\epsilon,L)$ from Lemma \ref{lemma-cut-1}, and $R>R_0(L,T)$ from Theorem \ref{thm-short}. 

Fix any $\gamma \in \Gamma_{\epsilon,R}$. 

By Lemma \ref{lemma-cut-1} we can find $X_0,X_1 \in \pi_1(\Su,*)$ such that $|\len(\cd X_i \cd ) -(R+2L -\log 4)|<\frac{1}{2}$, and 

\begin{equation}\label{name-1}
\gamma=(X_0)_T+(X_1)_T
\end{equation}
in $\Pant_{300\epsilon,R}$ homology. Observe that $q(\gamma)=q(X_0)+q(X_1)$.

By Lemma \ref{lemma-division} we can write $X_0=X_{00}X_{01}$, where 
\begin{equation}\label{name-2}
\len(\cd X_{0i} \cd ) \in \big[ \frac{R}{2}+2L,\frac{R}{2}+2L+1 \big]
\end{equation}
and the conclusions of Lemma \ref{lemma-division} hold. And likewise for $X_1$.

Let $N=\lfloor \log_{2} R \rfloor -1$. For every $0 \le k \le N$, we define sets $\Xe_k$ by letting $\Xe_0=\{X_0,X_1 \}$ and the set $\Xe_{k+1}$ is the set of children of elements of $\Xe_k$.
Each set $\Xe_k$ has  $2^{k+1}$ elements and the elements of $\Xe_k$ are not necessarily distinct. ( For the pedantic reader we proceed as follows: $\Xe_k$ is a set of ordered pairs of the form $(a,X)$, when $0 \le a <2^k$ and $X \in \pi_1(\Su,*)$. If $X_0, X_1$ are constructed from $X$ according to Lemma \ref{lemma-division}  we let the children of $(a,X)$ be $(2a+i,X_i)$ for $i=0,1$. Then we let $\Xe_0=\{(i,X_i) :\, i=0,1\}$ and let $\Xe_{k+1}$ be the set of children of $X_i$.)

Moreover, for any $X \in \Xe_k$ we have
$$
\len(\cd X \cd ) \in \big[ R 2^{-k}+2L,R 2^{-k}+2L+1 \big].
$$

We claim that 
$$
\theta^{T}_{X}<\frac{\pi}{3}
$$
for every $X$ in any $\Xe_k$. For any such $X$ we can find a sequence $Y_0, Y_1,...,Y_k$, so that $Y_0=X_0$ or $Y_0=X_1$ and $Y_k=X$, and where $Y_{i+1}$ is a child of $Y_i$.
It follows from the equation $(\ref{name-2} )$ that $\len(\cd Y_{i+1} \cd ) \le \len(\cd Y_i \cd )-1$, and $\len(\cd Y_k \cd ) \ge 2L-2$. 
\begin{align*}
\theta^{T}_{Y_{k}} &\le \frac{\pi}{6}+\sum_{i=0}^{k} e^{L+4-\len(\cd Y_{i} \cd )} \\
&\le \frac{\pi}{6}+\frac{e}{e-1} e^{L+4-(2L-3) } <\frac{\pi}{3},
\end{align*}
assuming $L>8$.

By Lemma \ref{lemma-pomoc} we have $I(\cd T \cd X \cd \ov T \cd ) \le \log 4$ for every $X$ in every $\Xe_k$. Therefore, we can apply The $XY$ Theorem (see Theorem \ref{thm-XY}) and conclude that

\begin{equation}\label{name-3}
Y_T=(Y_0Y_1)_T
\end{equation}
whenever $Y$ is a non-trivial node of our tree and $Y_0$ and $Y_1$ are its two children.

It follows by $(\ref{name-1})$ and $(\ref{name-3})$ applied recursively that 
$$
\gamma=\sum_{X \in \Xe_{N}} X_T,
$$
in $\Pant_{300\epsilon,R}$ homology. We know that if $X \in \Xe_N$ then $X \in W_l$, so 
$X_T=(q(X))_T$. Therefore
$$
\gamma=\sum_{X \in \Xe_{N}} (q(X))_T=\sum ((q(\gamma))_T,
$$
so we are finished.

\end{proof}

\begin{random} Randomization remarks for the proof of Theorem \ref{thm-correction}. We have determined  $T\equiv T(\Su,\epsilon)$, so $e^{\len(\cd T \cd )}=K(\Su,\epsilon)$.
We have implicitly defined the map $g:\Gamma_{\epsilon,R} \to \R \Pant_{300\epsilon,R}$ such that 
$\partial g(\gamma)=\gamma-(q(\gamma))_T$. We note that $q=q^{N}_D \circ q_C$, where $q_C(\gamma)=X+X'$ from Lemma \ref{lemma-cut-1}, $q_D(X)=X_0+X_1$, from Lemma \ref{lemma-division}.

Moreover, 
$$
g(\gamma)=g_{C}(\gamma)+\sum_{i=0}^{N-1} g_{XY}(\wh{q}_{D}(q^{i}_{D}(q_{C}(\gamma))))+g_W(q^{N}_{D}(q_C(\gamma)))
$$
where $g_C$ is the map from Lemma \ref{lemma-cut-1}, $q_D$ and $\wh{q}_D$ are the maps from Lemma \ref{lemma-division}, $g_{XY}$ is the map from Theorem \ref{thm-XY}, $N$ is the number of times we iterate the division (the application of Lemma \ref{lemma-division}), and $g_W$ is from Theorem \ref{thm-short}.

By far the most important point is that $q_D$ is $K=K(\Su,\epsilon)$-semirandom, so $q^{i}_D$ is $K^{i}$-semirandom, for any $i \le N$ (recall that $N \le \lfloor \log_{2} R \rfloor$) and therefore 
$K^{i} \le R^{\log_{2} K}$ so the map $q^{i}_D$ is $P(R)$-semirandom, where $P(R)$ denotes a polynomial in $R$).

\end{random}

\subsection{The proof of Theorem \ref{thm-correction} } The map $\phi$ from  Theorem \ref{thm-correction} is defined to be equal to the map $g$ from the Randomization remarks for Theorem \ref{thm-correction-2}. We take $h_i=(g_i)_T$.  Then $\partial \phi(\gamma)=\gamma-(q(\gamma))_T$, and $(q(\gamma))_T  \in  \R \{h_1,..,h_{2n} \}$,  for any $\gamma \in \R \Gamma_{\epsilon,R}$. 
Moreover, the map $\phi$ is $P(R)$-semirandom as shown in those Randomization remarks. This implies the estimate $(3)$ of the statement of Theorem \ref{thm-correction}  and we are finished.

\section{Appendix 1}  

\subsection*{Introduction to randomization} Let $(X,\mu)$ and $(Y,\nu)$ denote two measure spaces (where $\mu$ and $\nu$ are  positive measures). 

\begin{definition} We say that a  map $g:(X,\mu) \to (Y,\nu)$ is $K$-semirandom with respect to $\mu$ and $\nu$ if $g_{*} \mu \le K \nu$.
\end{definition}

By $\R X$  we denote the vector space of finite formal sums (with real coefficients) of points in $X$. There is a natural inclusion map  $\iota:X \to \R X$, where $\iota(x) \in \R X$ represents  the corresponding sum. Then every map $\wt{f}:\R X \to S$, where $S$ is any set, induces the map $f:X \to S$ by letting $f=\wt{f} \circ \iota$.

Let $f:X \to \R Y$ be a map. Then we can write $f(x)=\sum_{y} f_x(y) y$, where the function $f_x:Y \to \R$ is non-zero for at most finitely many points of $Y$. We define $|f|:X \to \R Y$ by
$$
|f|(x)=\sum_{y}|f_x(y)|y.
$$
We define the measure $|f|_{*} \mu$ on $Y$ by
$$
|f|_{*} \mu (V)=\int\limits_{X} \left( \sum_{y} |f_x(y)| \chi_V (y) \right) d\mu(x),
$$
for any measurable set  $V\subset Y$, and   $\chi_V (y)=1$, if $y \in V$ and  $\chi_V (y)=0$, if $y \notin V$.

\begin{definition} Let $(X,\mu)$ and $(Y,\nu)$ be two measure spaces (with positive measures $\mu$ and $\nu$). A map $f:X \to \R Y$ is $K$-semirandom if  $|f|_{*} \mu \le K \nu$.
A linear map $\wt{f}:\R X \to \R Y$ is $K$-semirandom with respect to measures $\mu$ and $\nu$ on $X$ and $Y$ respectively,  if the induced map $f:X \to \R Y$ is $K$-semirandom.
\end{definition}

The following propositions are elementary.

\begin{proposition}\label{prop-assoc} Let $X$, $Y$ and $Z$ denote three measure spaces. If $f:\R X \to \R Y$ is $K$-semirandom, and $f:\R Y \to \R Z$ is $L$-semirandom, then $g \circ f:\R X \to \R Z$ is $KL$-semirandom.
\end{proposition}

\begin{proposition}\label{prop-conv} If $f_i:\R X \to \R Y$ is $K_i$-semirandom, $i=1,2$, and $\lambda_i \in \R$, then the map $(\lambda_1 f_1+\lambda_2 f_2):\R X \to \R Y$ 
is $(|\lambda_1|K_1+|\lambda_2|K_2)$-semirandom.
\end{proposition}

\begin{remark} We say that $f:X \to Y$ is a partial map if it is defined on some measurable subset  $X_1 \subset X$. The notion of a semirandom maps generalizes to the case of  partial maps by letting  a partial map $f:X \to Y$ be $K$-semirandom if the restriction $f:X_1 \to Y$ is $K$-semirandom, where  the corresponding  measure on $X_1$ is the restriction of the measure from $X$.
Every  statement we make about semirandom maps has its  version for a partial semirandom map. In particular, if  $f:X \to Y$ is  $K$-semirandom then the restriction of $f$ onto any $X_1 \subset X$ is $K$-semirandom. Moreover,  trivial partial maps (those that are defined on an empty set) are $K$-semirandom for any $K \ge 0$. 
\end{remark}

A measure class on a  space $X$ is a subset of $\Mes(X)$ where $\Mes(X)$ is the set of measures on $X$.

\begin{definition} Let $X$ and $Y$ be measure spaces and let $\Me \subset \Mes(X)$ and $\Ne \subset \Mes(Y)$ be measures classes on $X$ and $Y$ respectively (all measures from $\Me$ and $\Ne$ are  positive measures). We say $f:X \to Y$ is $K$ semirandom with respect to $\Me$ and $\Ne$ if for every $\mu \in \Me$ there is $\nu \in \Ne$ such that $f$ is $K$-semirandom with respect to $\mu$ and $\nu$, that is $f_{*} \mu \le K \nu$.
\end{definition}

In a similar fashion as above we define the notion of a semirandom map $f:\R X \to \R Y$ with respect to classes of measures $\Me$ and $\Ne$ on $X$ and $Y$ respectively. The following proposition follows from Proposition \ref{prop-assoc}.

\begin{proposition}\label{prop-assoc-1} Let $X$, $Y$ and $Z$ denote three measure spaces, with classes of measures $\Me$, $\Ne$ and $\Ze$ respectively. If $f:\R X \to \R Y$ is $K$-semirandom with respect to $\Me$ and $\Ne$, and $f:\R Y \to \R Z$ is $L$-semirandom with respect to $\Ne$ and $\Ze$, then $g \circ f:\R X \to \R Z$ is $KL$-semirandom with respect to $\Me$ and $\Ze$.
\end{proposition}

We say that a class of measures $\Me$ is convex if it contains all convex combinations of its elements. The following proposition then follows from Proposition \ref{prop-conv}

\begin{proposition}\label{prop-conv-1} If $f_i:\R X \to \R Y$ is $K_i$-semirandom with respect to classes of measures $\Me$ and $\Ne$, $i=1,2$, and if $\Ne$ is convex, then for  $\lambda_i \in \R$, the map $(\lambda_1 f_1+\lambda_2 f_2):\R X \to \R Y$  is $(|\lambda_1|K_1+|\lambda_2|K_2)$-semirandom with respect to $\Me$ and $\Ne$.
\end{proposition}

\begin{remark} The space $\R X$ is naturally contained in the space $\Mes(X)$, and in a similar way we can define the notion of a semirandom map $f:\Mes(X)  \to \Mes(Y) $. 
\end{remark}

\subsection*{Natural measure classes} Let $X_i$, $i=1,..,k$, denote  measure spaces with measure classes $\Me_i$. Let $X_1 \times ... \times X_k$ denote the product space and by 
$\pi_i:(X_1\times ... \times X_k) \to X_i$ we denote the coordinate projections. By $\Me_1 \times \Me_2...\times \Me_k$ we denote the set of measures on $X_1 \times ... \times X_k$
that arise as the convex combinations of all standard products $\mu_1 \times ... \times \mu_k$ with $\mu_i \in \Me_i$.
We also define a natural class of measures $\Me_1 \boxtimes \Me_2 ...\boxtimes \Me_k$ on  $X_1\times ... \times X_k$ as
$$
\Me_1 \boxtimes \Me_2 ...\boxtimes \Me_k = \{\mu \in \Mes(X_1\times ... \times X_k) :\, (\forall i) (\exists \mu_i \in \Me_i)((\pi_i)_{*} \mu \le \mu_i )  \}.
$$

This produces a large class of measures even if each $\Me_i$ consists of a single measure. If each $\Me_i$ is convex then $\Me_1 \boxtimes \Me_2 ...\boxtimes \Me_k$ is as well.
If $X_i=X$ and $\Me_i=\Me$, then the standard product measure on $X^{k}$ is $\Me ^{\times k}$ and the other class of measures is denoted by $\Me^{\boxtimes  k}$.

We define the class $\Le_1$ of Borel  measures on $\R$ by saying that $\mu \in \Le_1$ if $\mu[x,x+1) \le 1$, for all $x \in \R$. This is a closed convex class of measures.
Likewise we define the class of measures $\Le_1$ on $\R / \lambda \Z$, for $\lambda>1$, by  saying that $\mu \in \Le_1$ if $\mu[x,x+1) \le 1$, for all $x \in \R / \lambda \Z$.
The class of measures $\Le_1$ is the class of measures that are controlled by the Lebesgue measure at the unit scale.

We consider the following spaces  and their measure classes. In this paper, we define  several maps (or partial maps) between these spaces (or their powers) and prove they are semirandom.
We have

\begin{enumerate}

\item The space of curves $\Gamma_{1,R}$ with the measure class containing the single measure $\sigma_{\Gamma}$ which is defined by setting $\sigma_{\Gamma}(\gamma)=Re^{-2R}$, for every $\gamma \in \Gamma_{1,R}$. We may assume that $\epsilon$ is small enough so that $\Gamma_{\epsilon,R} \subset \Gamma_{1,R}$.

\item The space of pants $\Pant_{1,R}$ with the measure class containing the single  measure $\sigma_{\Pant}$ given by $\sigma_{\Pant}(\Pi)=e^{-3R}$. We may assume 
that $\epsilon$ is small enough so that $\Pant_{300\epsilon,R} \subset \Pant_{1,R}$.

\item Let $\overset{.}{\Gamma}_{1,R}=\{(x,\gamma): \, \gamma \in \Gamma_{1,R},\, x \in \TT_{\gamma} \}$  denote  the space of pointed curves (recall that $\TT_{\gamma}=\R / \len(\gamma)\Z$ is the parameterizing torus for $\gamma$). The space  $\overset{.}{\Gamma}_{1,R}$ is really just the union of parameterizing tori $\TT_{\gamma}$ for curves $\gamma \in \Gamma_{1,R}$. By $\Sigma_{\overset{.}{\Gamma}}$ we denote the measure class on $\overset{.}{\Gamma}_{1,R}$, such that $\mu \in \Sigma_{\overset{.}{\Gamma} }$ if the restriction  $\mu_{\gamma}=\mu|_{\TT_{\gamma} }$ is in $e^{-2R}\Le_1$, where $\Le_1$ is the measure class on the circle $\TT_{\gamma}$ that was defined above. 

\item Let $\overset{.k}{\Gamma}_{1,R}=\{(x_1,...,x_k,\gamma): \, \gamma \in \Gamma_{1,R},\, x_i \in \TT_{\gamma} \}$  denote  the space of curves with $k$ marked points. The space $\overset{.k}{\Gamma}_{1,R}$ is canonically contained in $\big( \overset{.}{\Gamma}_{1,R} \big)^{k}$. The measure class $\Sigma_{ \overset{k.}{\Gamma} }$ on $\overset{.k}{\Gamma}_{1,R}$ is the restriction of 
$\Sigma^{\boxtimes k}_{\overset{.}{\Gamma} }$ on the image of $\overset{.k}{\Gamma}_{1,R}$ in  $\big( \overset{.}{\Gamma}_{1,R} \big)^{k}$.

\item The space $G=\pi_1(\Su,*)$ with the measure class $\Sigma_G$ that is the  convex closure of the collection of measures $\sigma_{a}$ on $G$, where $\sigma_a$ is defined  
so that for $X \in G$ we have  $\sigma_a(X)=\nu_{a}(\len(\cd X \cd ))e^{-\len(\cd X \cd)}$, where  $\nu_{a}(x)=1$, if $ x \in [a,a+1]$, and $\nu_{a}(x)=0$ otherwise.

\end{enumerate}

We observe that there exists a constant $K=K(\Su)$ such that for any measure $\mu$ in any of the above defined measure classes, the total measure of $\mu$ is bounded by $K$.

Finally we consider the map $\partial: \Pant_{1,R} \to \R \Gamma_{1,R}$ defined by $\partial \Pi=\gamma_0+\gamma_1+\gamma_2$, where $\gamma_i$ are the three oriented boundary curves of $\Pi$. 
We observe that $\partial$ is $K(\Su)$-semirandom from $\sigma_{\Pant}$ to $\sigma_{\Gamma}$.

\subsection*{Standard maps are semirandom}  We consider several standard mappings and prove they are semirandom. 

\begin{lemma}\label{lemma-R-1} Let $l>0$ and  $a,b \le l-1$. Then for any $Z \in G=\pi_1(\Su,*)$ such that $\len(\cd Z \cd )=l$, there are at most $Ke^ {\frac{a+b-l}{2}}$  ways of writing $Z=XY$, with
$\len(\cd X \cd ) \in [a,a+1]$ and $\len(\cd Y \cd ) \in [b,b+1]$, for some $K=K(\Su)$.

\end{lemma}

\begin{proof}  Suppose that $X$ and $Y$ satisfy the given conditions. Consider a triangle in $\Ha$ whose sides are lifts of $\cd X \cd $, $\cd Y \cd $ and $\cd Z \cd$ (these lifts are denoted the same as the arcs we are lifting). Then we drop the  perpendicular $t$ from the vertex $z$ opposite to $\cd Z \cd$ to the side $\cd Z \cd$, and let $a'$ and $b'$ be the lengths of the subintervals of $\cd Z \cd$
that meet at the endpoint of $t$ on $\cd Z \cd$ (then $a'+b'=\len(\cd Z \cd)$). For simplicity, set $t= \len(t)$
We find that
$$
a \le \len(\cd X \cd) \le t+a' \le \len(\cd X \cd) +\log 2 \le a+2,
$$
and likewise $b \le t+b' \le b+2$. 
So 
$$
t \in \left[ \frac{a+b-l}{2},\frac{a+b-l}{2}+2 \right],
$$
and 
$$
a' \in \left[ \frac{a-b+l}{2}-2,\frac{a-b+l}{2}+2 \right].
$$
Therefore, the  vertex $z$  must lie in a disc of radius $\frac{a+b-l}{2}+4$ around the point on $Z$ that is $\frac{a-b+l}{2}$ away from the initial point of $Z$. 
It follows that there are at most  $K(\Su)e^{\frac{a+b-l}{2}}$ lifts of the base point in this disc, and we are finished.

\end{proof}

Let $p:G \times G \to G$ be the product map, that is $g(X,Y)=XY$.

\begin{lemma}\label{lemma-R-product} The map $p:G \times G \to G$ is $K$-semirandom with respect to $\Sigma_G \times \Sigma_G$ on $G^{2}$ and $\Sigma_G$ on $G$, for some $K=K(\Su)$

\end{lemma}

\begin{proof} Let $a,b \in [0,\infty)$, and assume $b \ge a$. Recall the measures $\sigma_a$ on $G$, and let $\sigma=p_{*}(\sigma_a \times \sigma_b)$. We must show that $\sigma \le K\Sigma_G$.

Let $Z \in G$, and let $l=\len(\cd Z \cd)$. If $a \le b \le l-1$, then 
$$
\sigma(Z) \le Ke^{ \frac{a+b-l}{2}} e^{-a}e^{-b}=Ke^{-l}e^{-\frac{a+b-l}{2}}.
$$
(If $l>a+b+2$ then $\sigma(Z)=0$). 

If $l-1 \le b$, then because there are at most $Ke^{a}$ $X$'s in $G$ for which $\sigma_a(X)>0$, we find 
$$
\sigma(Z) \le Ke^{a}e^{-a}e^{-b}=Ke^{-l}e^{-(b-l)}.
$$
Then we see that 
$$
\frac{1}{K} \sigma  \le  \sum_{k=\lfloor b-a-1 \rfloor}^{\lfloor b+1 \rfloor }e^{-(b-k)}\sigma_k+\sum_{k=\lfloor b \rfloor}^{\lfloor a+b+3 \rfloor }e^{-\frac{(a+b-k)}{2} } \sigma_k,
$$
so $\sigma \le K\Sigma_G$.

\end{proof}

We define  a partial map $\proj:G \to \overset{.}{\Gamma}_{1,R}$ as follows. Given $A \in G$,  we let $\gamma=[A]$, and $z \in \gamma$ be the projection of the base point $*$ to $\gamma$. As always, the projection  is defined by  choosing lifts of $\cd A \cd$ and $\gamma$ in $\Ha$ that have the same endpoints and then we project a lift of $*$ to the lift of $\gamma$, where the lift of $*$ belongs to the lift of $\cd A \cd$. We let $\proj (A)=(\gamma,z)$.

\begin{lemma}\label{lemma-R-pro}  The map $\proj:G \to \overset{.}{\Gamma}_{1,R}$ is  $K(\Su)$-semirandom with respect to $\Sigma_G$ and $\Sigma_{\overset{.}{\Gamma}}$.
\end{lemma}

\begin{proof} Let $J$ be a unit interval on a curve $\gamma \in \Gamma_{1,R}$. We have seen in the two previous proofs that there are at most $Ke^{\frac{l-2R}{2} }$ many $Z \in G$ for which $\len(\cd Z \cd) \le l$, and $\proj(Z)  \in J$. Therefore, if $\sigma \in \Sigma_G$, then
$$
\proj_{*} \sigma  (\gamma,J) \le K \sum_{k=\lfloor 2R \rfloor }^{\infty} e^{ \frac{k-2R}{2} } e^{-k} \le K e^{-2R},
$$ 
and we are finished.

\end{proof}

Another standard map we consider is the projection  map $\overset{.}{\Gamma}_{1,L} \to \Gamma_{1,L}$ given by $(\gamma,x) \to \gamma$. This map is clearly $1$-semirandom. 
Going in the opposite direction, we have the map $\gamma \to (\gamma,x)$ which assigns to $\gamma \in \Gamma_{1,R}$ a random point $x \in \gamma$. This map is really defined as a map $\Mes(\Gamma_{1,R}) \to \Mes(\overset{.}{\Gamma}_{1,R})$, and we observe that it is $1$-semirandom as well.

\begin{remark} We also observe that for $T \in G$, the map $\{1\} \to G$ defined by $1 \to T$ is $e^{\len(\cd T \cd)}$ semirandom with respect to the unit measure on $\{1\}$ and $\Sigma_{G}$.

\end{remark}

\subsection*{The principles of randomization} After almost every lemma or theorem we prove in Sections 4-9, we have added a ``Randomization remark" which considers the functions
that we have implicitly defined, states their domain and range, and argues that the functions are semirandom with respect to a certain measure class. In the remarks we have followed the following principles:

1. When we write ``a random element" (of a finite set $S$) which the reader was previously told to read as ``an arbitrary element", we now mean ``the random element" of $\R S$, namely
$$
\frac{1}{|S|}\sum_{x \in S} x.
$$
If $a \in \R S \subset \Mes(S)$ and $\Me$ is a measure class on $S$ we say that $a$ is a $K$-semirandom element of $S$, with respect to $\Me$, if there exists $\mu \in \Me$ such that $a \le K \mu$.

2. We can replace at will any map $f:X \to Y$ (or $f:X \to \R Y$) by the linear extension $f:\R X \to \R Y$. This can cause confusion if you think about it the wrong way so we offer the following example to clarify what is going on. 

In the hard case of the GSL, we take a random third connection (meaning the random third connection), and then cancel out one square $(A_{ij})$, $i,j=0,1$, of boundaries to get a formal sum of squares $(B_{ij})$ of curves. We then find, for each new square (in the formal sum) a second third connection at random from a set depending on $(B_{ij})$, to complete the argument. The right way to think of the randomization (and linearizion) is that the first operation defines a partial map 
$$
q_1:\big( \overset{....}{\Gamma}_{1,R} \big)^{4} \to \R  \big( \overset{....}{\Gamma}_{1,R} \big) ^{4} 
$$
and the second operation defines 
$$
g_0:\Gamma^{4}_{1,R} \to \R \Pant_{1,R},
$$
so we can write $g_0 \circ q_1$ by extending $g_0$ to a map from $\R \Gamma^{4}_{1,R}$ to $\R \Pant_{1,R}$ linearly.
The danger is that one may try to imagine $g_0$ acting on a formal sum of curves by taking the random element from  $\R \Conn_{\epsilon,R}(\cd , \cd )$. 

So we will imagine that we are defining functions from $X$ to $Y$, or from $X$ to $\R Y$, and only think of them as functions from $\R X$ to $\R Y$ when we want to compose them.

3. We want to use the measure class $\Sigma_G \times \Sigma_G=\Sigma^{\times 2}_G$ on $G^{2}=\{ (X,Y) \}$ when we want to form the product $XY$. We want to use the measure class 
$\Sigma_G \boxtimes \Sigma_G=\Sigma^{\boxtimes 2}_G$ on $G^2$ if we want to be able to let $X=Z$ and $Y=Z$ for some $Z \in G$. 

For example, for the ASL, we use the measure class $\Sigma^{\boxtimes 2}_G \times \Sigma_G \times \Sigma^{\boxtimes 2}_G \times \Sigma_G$ on six-tuples 
$(A_0,A_1,U,B_0,B_1,V)$ in $G^{6}=G^2 \times G \times G^2 \times G$. This is basically the largest measure class for which the maps $\pi_{ij}:G^6 \to G^4$ defined by 
$\pi_{ij}(A_0,A_1,U,B_0,B_1,V)=(A_i,U,B_j,V)$ are $1$-semirandom with respect to the measure class $\Sigma^{\times 4}_G$ on $G^4$.

This is exactly what we want, because we have to form the words $A_iUB_jV$, but we need the freedom to assign to $A_0$ and $A_1$ (or $B_0$ and $B_1$) the same value.

\section{Appendix 2}  

We develop the theory of equidistribution and counting, based on the uniformly exponential mixing of the geodesic flow, that we need to prove Theorem \ref{main estimate-0}.

\subsection*{Equidistribution of equidistant lines}  

The group $\PSLR$ acts on the unit tangent bundle $\TB \Ha$ on the left (we refer to this action as the action by isometries). Namely, if $v \in \TB \Ha$ and $h \in \PSLR$ then $h \cdot v=h(v)$ is the resulting vector in $\TB \Ha$. Moreover, if $u$ and $v$ are two vectors in $\TB \Ha$ then there exists a unique element $h \in \PSLR$ such that $h \cdot v=u$. This enables us to identify the unit tangent bundle $\TB \Ha$ with $\PSLR$ as follows. Choose a vector $v_0 \in \TB \Ha$. We identify $v_0$ with the identity element $\J$ in $\PSLR$. A vector $v \in \TB \Ha$ is identified with the unique element $h \in \PSLR$ so that $h \cdot v_0=v$.

The group $\PSLR$ also acts on $\TB \Ha$ on the right (we sometimes refer to this action as the action by instructions). This action is uniquely defined by the following two properties:

\begin{enumerate}

\item $h \cdot v_0=v_0 \cdot h$, 
\item $g\cdot (v \cdot h)=(g \cdot v) \cdot h$,
\end{enumerate}
for any $h,g \in \PSLR$.

Now we describe the following factorization of elements of $\PSLR$. Let $t \in \R$ and let $g_t \in \PSLR$ denote a one-parameter Abelian group of hyperbolic transformations. 
By $\text{Axis}(g_t)$ we denote the oriented axis of the hyperbolic transformations $g_t$ (it is oriented from the repelling fixed point to the attracting fixed point of a transformation $g_t$ for $t>0$).  
Let $Y_a$, where $a \in \R$, denote another such family so that the oriented angle between $\text{Axis}(Y_a)$ and  $\text{Axis}(g_t)$ is $\frac{\pi}{2}$. Denote by $p_0 \in \Ha$ the point 
$p_0= \text{Axis}(Y_a) \cap \text{Axis}(g_t)$  and let $v_0 \in \TB \Ha$ denote the vector at $p_0$ that is tangent to  $\text{Axis}(g_t)$. Also, let $R_c \in \PSLR$, $c \in \R$, denote the  rotation for angle $c$ about the point $p_0$.

\begin{lemma}\label{factor} Any $h \in \PSLR$ can be uniquely written as $h=Y_a \cdot  g_t \cdot R_c$, for some choice of $a,t,c \in \R$.
\end{lemma}

\begin{proof} We identify $\PSLR$ with the unit tangent bundle $\TB \Ha$ in the usual way so that the identity in $\PSLR$ is identified with the vector $v_0 \in \TB \Ha$. Let $v \in \TB \Ha$ be the vector corresponding to some $h \in \PSLR$ and suppose that $v$ is based at the point $p \in \Ha$. For simplicity set $\alpha(g)=\text{Axis}(g_t)$ and $\alpha(Y)=\text{Axis}(Y_a)$.  Let $\beta$ be the oriented geodesic in $\Ha$ that contains $p$ and is orthogonal to $\alpha(Y)$ and that points to the left of $\alpha(Y)$. 

\begin{enumerate} 

\item Let $u$ be the tangent vector to $\beta$ at the point $p$. By $c$ we denote the oriented angle between the vectors $v$ and $u$.

\item Let $t$ denote the signed distance from $p$ to $\alpha(Y)$ along the oriented geodesic $\beta$.

\item Let $q=\beta \cap \alpha(Y)$ and let $a$ denote the signed distance  from $p_0$ to $q$ along the oriented geodesic $\alpha(Y)$.

\end{enumerate}

Then clearly $h=Y_a \cdot g_t \cdot R_c$ and such $a$, $t$ and $c$ are  unique.

\end{proof}

We equip $\PSLR$ with the following distance function. Let $h_j \in \PSLR$, $j=1,2$, and let $v_j \in \TB \Ha$ denote the  vectors corresponding to $h_j$ (the vector $v_j$ is based at the point $p_j \in \Ha$).
We let 
$$
d_{\PSLR}(h_1,h_2)=d(p_1,p_2)+\Ang(u_1,v_2),
$$
where $u_1 \in \TB \Ha$ is the parallel transport of the vector $v_1$ at the point $p_2$ (recall that $\Ang(u,v)$ is the unoriented angel between vectors $u$ and $v$).
For $h \in \PSLR$ we denote by $||h||$ the distance between $h$ and the identity element  $\J  \in \PSLR$. 

We leave the proof of the following lemma to the reader.

\begin{lemma}\label{wave} There are universal constants $\delta_0,K_0>0$ such that providing $||h|| \le \delta_0$ then 

$$
h \cdot g_t = Y_a \cdot g_{t+b} \cdot  R_c,
$$
where $|a|+|b|+|c| \le  K_0 ||h||$.

\end{lemma}

We now discuss the equidistribution of the equidistant lines on a closed Riemann surface $S$. Let $\alpha:\R \to S$ be a unit speed geodesic, and let $\wh{\alpha}:\R \to \TB S$ be the leftward normal unit vector field, given by
$\wh{\alpha}(s)=i\alpha'(s)$ (here $i$ denotes the imaginary unit in the tangent space to $S$ at the point $\alpha(s) \in S$). Let $t \in \R$ and consider the vector field $g_t(\wh{\alpha})$. Then the vectors from the field   $g_t(\wh{\alpha})$
are orthogonal to the line that is equidistant (at distance $t$) from the geodesic $\alpha$.  In what follows  we assume that the Liouville measure $\Lambda$ on $\TB S$ is normalized so that $\Lambda(\TB S)=1$.

\begin{theorem}\label{equidistant} Let $f:\TB S \to \R$ be any $C^1$ function. Then for $a \ge C_1 e^{-qt}$, we have
$$
\left| \frac{1}{a}\int\limits_{0}^{a} f(g_t(\wh{\alpha}(s)) \, ds -\int\limits_{\TB S} f \, d\Lambda \right| \le C_2 e^{\frac{-qt}{5}} \big( \frac{1}{a}+||f||_{C^{1}} \big),
$$
where the positive constants $C_1$, $C_2$ and $q$ depend only on $S$. 
\end{theorem}

\begin{proof} We let $\psi_{\eta}:\PSLR \to [0,\infty)$ be such that

\begin{enumerate}

\item $\psi_{\eta}$ is supported in $B_{\eta}(\J)$ which is the ball of radius $\eta$ centered at $\J$.

\item $\int\limits \psi_{\eta} =1$, where we integrate with respect to the Haar measure on $\PSLR$. 

\item $||\psi_{\eta}||_{C^{1}} \le K_1 \eta^{-4}$, for some universal constant $K_1$.

\item $\psi_{\eta}(X)=\psi_{\eta}(X^{-1})$ for $X \in \PSLR$.

\end{enumerate}

We can arrange that $(3)$ holds because $\psi_{\eta}$ needs to reach the height of $\eta^{-3}$ in a space of size $\eta$ (so the derivative of $\psi_{\eta}$ is proportional to $\eta^{-4}$).
For simplicity we let $\psi=\psi_{\eta}$.

If $u,v \in \TB \Ha$ then there is a unique $g \in \PSLR$ such that $u \cdot g=v$. We let $\psi(u,v)=\psi(g)$ (condition $(4)$ above implies that $\psi(u,v)=\psi(v,u)$). Then for $a<b$ and $X \in \TB S$ we let
$$
\wh{\alpha}_{a,b} (X)=\int\limits_{a}^{b} \psi(\wh{\alpha}(s),X) \, ds,
$$
and we let $\wh{\alpha}_a=\wh{\alpha}_{0,a}$. 

Then 

$$
||\wh{\alpha}_{a,b}||_{C^{1}} \le (b-a)||\psi||_{C^{1}}
$$
and 
$$
\int\limits_{\TB S} \wh{\alpha}_{a,b} \, d\Lambda=b-a.
$$

Applying the factorization lemma above (Lemma \ref{wave}) we find that 

\begin{align*}
\int\limits_{\TB S} f(X) \wh{\alpha}_a (g_{-t} X) \, d\Lambda(X) &= \int\limits_{0}^{a} \int\limits_{B_{\eta}(\J)}   f( \wh{\alpha}(s) \cdot h \cdot g_t)  \psi(h)  \, ds d\Lambda(h) \\
&= \int\limits_{0}^{a} \int\limits_{B_{\eta}(\J)}   f( \wh{\alpha}(s) \cdot Y_a \cdot g_{t+b} \cdot  R_c )  \psi(h)  \, ds d\Lambda(h) \\
&= \int\limits_{0}^{a} \int\limits_{B_{\eta}(\J)}   f( \wh{\alpha}(s+a(t,h)) \cdot g_{t+b(t,h)} \cdot R_{c(t,h)} )  \psi(h) \, ds d\Lambda(h) \\
&= \int\limits_{0}^{a} \int\limits_{B_{\eta}(\J)}   f( \wh{\alpha}(s+a(t,h))  \psi(h) \, ds d\Lambda(h) +O\big( ||f||_{C^{1}} \eta a \big),
\end{align*}
where the last equality follows from the upper bounds on $b(t,h)$ and $c(t,h)$ from Lemma \ref{wave}. This yields the inequalities

\begin{align*}
\int\limits_{K_0 \eta}^{a-K_0 \eta} f(g_t(\wh{\alpha}(s)) \, ds - K_0 ||f||_{C^{1}} \eta &\le  
\int\limits_{\TB S} f(X)\wh{\alpha}_a(g_{-t} X) \, d\Lambda(X) \\ 
&\le \int\limits_{-K_0 \eta}^{a+K_0 \eta} f(g_t(\wh{\alpha}(s)) \, ds + K_0||f||_{C^{1}} \eta.
\end{align*}

On the other hand, by exponential mixing 

\begin{align*}
\left| \int\limits_{\TB S} ((g_t)_{*} \wh{\alpha}_a)(X)f(X) \, d\Lambda(X) - \int\limits_{\TB S} \wh{\alpha}_a \, d\Lambda  \int\limits_{\TB S} f \, d\Lambda\right| \\
\le Ce^{-qt}||\wh{\alpha}_a||_{C^{1}}||f||_{C^{1}} \le Ce^{-qt}a\eta^{-4}||f||_{C^{1}},
\end{align*}
where $C,q>0$ depend only on $S$. So we obtain, for $a>2K_0\eta$,

\begin{align*}
\int\limits_{0}^{a}f(g_t(\wh{\alpha}(s))) \, ds  &\le \int\limits_{\TB S} f ( (g_t)_{*} \wh{\alpha}_{-K_{0}\eta,a+K_{0}\eta} ) \, d\Lambda +K_0||f||_{C^{1}}\eta \\
&\le (a+2K_0 \eta) \int\limits_{\TB S} f \, d\Lambda + Ce^{-qt}a \eta^{-4} ||f||_{C^{1}} +K_0\eta ||f||_{C^{1}}
\end{align*}

and likewise
\begin{align*}
\int\limits_{0}^{a}f(g_t(\wh{\alpha}(s))) \, ds \ge (a-2K_0\eta)\int\limits_{\TB S} f\, d\Lambda- (Ce^{-qt} a \eta^{-4} +K_0\eta)||f||_{C^{1}}.
\end{align*}
Letting $\eta=e^{-\frac{1}{5} qt}$, the theorem follows.

\end{proof}

\subsection{Counting good connections}  Let $\beta$ be another geodesic on $\Su$ and define $\wh{\beta}:\R \to \TB \Su$ in analogy to $\wh{\alpha}$. For intervals $I$ and $J$ in $\R$ we let 
$$
M_{I,J}=\{g_t(\wh{\beta}(s)): \, (s,t)\in I \times J\},
$$
be a $2$-submanifold of $\TB \Su$. 

We recall the normal flow $Y$ on $\TB \Su$. Then $\wh{\alpha}'(s)=Y(\wh{\alpha}(s))$ for any geodesic $\alpha:\R \to \TB \Su$, and we let 
$$
Y^t \equiv \frac{1}{\cosh t}(g_t)_{*}Y
$$
be the distance $t$ flow, so
$$
\frac{\partial}{\partial s}\big ( g_t(\wh{\alpha}(s) ) \big)=(\cosh t) Y^t(g_t(\wh{\alpha}(s))).
$$

For $t$ large $Y^t$ is close to the negative horocyclic flow.

\begin{theorem}\label{thm-count-0} Let $f$ be a $C^1$ function with compact support on $M_{I,J}$. Then 

\begin{align*}
&\left| \frac{1}{a\cosh t} \sum_{g_t(\wh{\alpha}(s)) \in M_{I,J}}  f(g_t(\wh{\alpha}(s))) - \int\limits_{M_{I,J}}  f \, \iota_{Y^{t}}dV \right| \\
&\le C e^{\frac{-qt}{5}} \big( \frac{1}{a}+||f||_{C^{1}} \big),
\end{align*}
provided that $|I|,|J| <\delta$ and $1 \ge a \ge C_1 e^{-qt}$, where $C,C_1, \delta>0$ depend only on $\Su$ and $||f||_{\infty} \le 1$. (Here $\iota_{Y^{t}}dV$ is the contraction of the volume form $dV$ by the vector $Y^t$.)

\end{theorem}

\begin{proof} The assumptions on $M_{I,J}, t$ and $a$ imply that the map $Q:M_{I,J} \times (0,\epsilon) \to \TB \Su$, defined by $Q(q,r,s)=Y^t_{s}\big( g_r(\wh{\beta}(q)) \big)$, is injective for some $\epsilon=\epsilon(\Su)$. We let $\psi$ be a $C^1$ bump function on $(0,\epsilon)$, and let $\wt{f}:Q\big( M_{I,J} \times (0,\epsilon) \big) \to \R$ be defined by
$$
\wt{f}\big( Q(q,r,s) \big)= f( g_r(\wt{\beta}(q)) ) \psi(s). 
$$

Then 
$$
||\wt{f}||_{C^{1}} \le C(\Su)||f||_{C^{1}},
$$
and 
$$
\int\limits_{\TB \Su} \wt{f} \, d\Lambda= \int\limits_{M} f\, \iota_{Y^{t}} d\Lambda.
$$

Moreover,
$$
\left| \sum_{g_t(\wh{\alpha}(s)) \in M_{I,J}}  f(g_t(\wh{\alpha}(s))) - \cosh t \int\limits_{0}^{a}   \wt{f}(g_t(\wh{\alpha}(s)))  \,ds  \right| \le ||f||_{\infty}.
$$
This inequality holds because every time the curve $Y_s \big( g_t(\wh{\alpha}(0)) \big)$ (for $s \in [0,\cosh t]$) crosses $M_{I,J}$, it goes through $Q$ (and contributes the same amount to the sum and the integral), except that the curve may start in $Q$ and miss $M_{I,J}$, and the terminal point may end in $Q$, contributing more to the sum than to the integral. For both endpoints the error is at most $|f||_{\infty}$, and the error has different signs at the two endpoints, so the total error is at most $||f||_{\infty}$.

Therefore, by Theorem \ref{equidistant},

\begin{align*}
&\left| \frac{1}{a\cosh t} \sum_{g_t(\wh{\alpha}(s)) \in M_{I,J}}  f(g_t(\wh{\alpha}(s))) - \int\limits_{M_{I,J}}  f \, \iota_{Y^{t}}dV \right| \\
&\le \frac{1}{a \cosh t}||f||_{\infty}+  C e^{\frac{-qt}{5}} \big( \frac{1}{a}+||f||_{C^{1}} \big) \\
&\le C e^{\frac{-qt}{5}} \big( \frac{1}{a}+||f||_{C^{1}} \big) .
\end{align*}

\end{proof}

If $\alpha$ and $\beta$ are two geodesic segments, and $\epsilon,L>0$, we let $\Conn_{\epsilon,L}(\alpha,\beta)$ be the set of $(r,s,t)$ such that $g_t(\wh{\alpha}(r))=\wh{\beta}(s)$ and $t \in [L,L+\epsilon]$.

\begin{theorem}\label{thm-count-1} Letting $\delta=e^{-\frac{qL}{40}}$, and $\alpha, \beta$ geodesic segments of length $\delta^2$, the number of orthogeodesics connections from one side of $\alpha$ to  one side of $\beta$, of length in the interval $[L,L+\delta^2]$, is given by
$$
\frac{1}{8\pi^2 \chi(\Su)} \delta^6 e^{L}\big(1+O(\delta) \big),
$$ 
where the big $O$ constant depends only on $\Su$.
\end{theorem}

\begin{proof}

We let $M=M_{I,J}$, where $I=[0,\delta^2]$ and $J=[-\delta^2,0]$. We want to count the number of $s \in [0,\delta^2]$ for which $g_L(\wh{\alpha}(s)) \in M$.

Let 
$$
M^{+}=M_{[-\delta^3,\delta^2+\delta^3],[-\delta^2-\delta^3,\delta^3] }
$$
be a slightly larger surface and we let $f^+$ be a $C^1$ function on $M^+$ that is equal to $1$ on $M$. We can arrange 
$$
||f^+||_{C^{1}} \le 10\delta^{-3},
$$
and $f^+$ takes values in in $[0,1]$.

Then

$$
\int\limits_{M^{+}}\big| f^+-\chi_M \big| \, \iota_{Y^{L}}dV \le 10\delta^5,
$$
and given our normalization of the Liouville volume form $dV$ we have
$$
\int\limits_{M} \, \iota_{Y^{L}}dV=\frac{1}{4\pi^2 |\chi(\Su)|} \delta^4+O(\delta^8).
$$

Putting all this together and applying Theorem \ref{thm-count-0} we have 
\begin{align*}
\frac{1}{ \delta^2 \cosh L } \# \Conn_{\epsilon, L } (\alpha,\beta) &\le \frac{1}{\delta^2 \cosh L } \sum_{g_L (\wh{\alpha}(s) ) \in M^{+}, s\in [0,a] } \, f^{+} \big( g_L(\wh{\alpha}(s) ) \big) \\
&\le \int\limits_{M^{+}} f^+ \, \iota_{Y^{L}}dV + C\delta^8 \big( \frac{1}{a}+||f||_{C^{1}} \big) \\
&\le \frac{\delta^4}{4\pi^2|\chi(\Su)|} +O(\delta^8)+10\delta^5+C\delta^8(\delta^{-2}+\delta^{-3}) \\
&\le \frac{\delta^4}{4\pi^2|\chi(\Su)|} +C\delta^5.
\end{align*}

We can analogously define $f^-$ supported on $M$, with $f^- \equiv 1$ on $M_{[\delta^3,\delta^2-\delta^3],[-\delta^2+\delta^3,-\delta^3] }$ and prove that
$$
\frac{1}{ \delta^2 \cosh L } \# \Conn_{\epsilon, L } (\alpha,\beta) \ge \frac{\delta^4}{4\pi^2|\chi(\Su)|} -C\delta^5.
$$
Since $\cosh L=\frac{e^L}{2}\big(1+O(e^{-2L}) \big)$, the theorem follows.

\end{proof}

\end{document}